\documentclass[12pt, twoside]{article}
\usepackage{amsmath, amsthm, amssymb, times, enumerate, mathrsfs, bbold, hyperref, stix, stmaryrd, color, graphicx, wrapfig}

\SetSymbolFont{stmry}{bold}{U}{stmry}{m}{n}
\graphicspath{{pic/}}
\usepackage[utf8]{inputenc}
\usepackage[T1]{fontenc}
\usepackage[open, openlevel=2, depth=3, atend]{bookmark}
\hypersetup{pdfstartview=XYZ}
\usepackage{anyfontsize}

\def\titlerunning#1{\gdef\titrun{#1}}
\makeatletter
\def\author#1{\gdef\autrun{\def\and{\unskip, }#1}\gdef\@author{#1}}
\def\address#1{{\def\and{\\\hspace*{18pt}}\renewcommand{\thefootnote}{}%
\footnote {#1}}%
\markboth{\autrun}{\titrun}}
\makeatother
\def\email#1{e-mail: #1}
\def\subjclass#1{{\renewcommand{\thefootnote}{}%
\footnote{\emph{Mathematics Subject Classification (2010):} #1}}}
\def\keywords#1{\par\medskip
\noindent\textbf{Keywords.} #1}

\numberwithin{equation}{section}

\newtheorem{theorem}{Theorem}
\newtheorem{proposition}{Proposition}[section]{\bf}{\it}
\newtheorem{definition}[proposition]{Definition}{\bf}{\it}
\newtheorem{defprop}[proposition]{Definition-Proposition}{\bf}{\it}
\newtheorem{lemma}[proposition]{Lemma}{\bf}{\it}
\newtheorem{corollary}[proposition]{Corollary}{\bf}{\it}
\newtheorem{remark}[proposition]{Remark}{\bf}{\it}
\newtheorem{example}[proposition]{Example}{\bf}{\it}
\newtheorem{conjecture}[proposition]{Conjecture}{\bf}{\it}

\newcommand{\C}{\mathbb{C}}
\newcommand{\R}{\mathbb{R}}
\newcommand{\Ss}{\mathbb{S}}
\newcommand{\N}{\mathbb{N}}
\newcommand{\Z}{\mathbb{Z}}
\newcommand{\Hh}{\mathbb{H}}

\newcommand{\PointM}{\zeta}

\DeclareMathOperator{\Op}{Op}
\DeclareMathOperator{\Res}{Res}
\DeclareMathOperator{\WF}{WF}
\DeclareMathOperator{\supp}{supp}
\DeclareMathOperator{\vol}{vol}
\DeclareMathOperator{\Span}{span}
\DeclareMathOperator{\End}{End}

\DeclareMathOperator{\Ell}{ell}
\DeclareMathOperator{\specb}{Spec_b}
\DeclareMathOperator{\proj}{\mathsf{Proj}}
\DeclareMathOperator{\comp}{\mathsf{comp}}
\DeclareMathOperator{\coord}{\mathsf{Coord}}
\DeclareMathOperator{\IndicialRes}{\mathbf{R}}
\DeclareMathOperator{\FibreM}{\mathsf{F}}
\DeclareMathOperator{\FibreL}{\mathsf{L}}
\DeclareMathOperator{\groupG}{\mathbb{G}}
\DeclareMathOperator{\groupK}{\mathbb{K}}
\DeclareMathOperator{\groupM}{\mathbb{M}}
\DeclareMathOperator{\groupA}{\mathbb{A}}
\DeclareMathOperator{\groupN}{\mathbb{N}}
\newcommand{\Xgr}{X_{\mathsf{gr}}}
\renewcommand{\Re}{\textup{Re}}
\renewcommand{\Im}{\textup{Im}}
\newcommand{\escapeparam}{\gamma}

\begin{document}
 
\titlerunning{Resonances and hyperbolic cusps}

\title{Ruelle-Pollicott resonances for manifolds with hyperbolic cusps.}

\author{Yannick Guedes Bonthonneau \and Tobias Weich}

\date{}

\maketitle

\address{Y. G. Bonthonneau: IRMAR, CNRS, Rennes, France; \email{yannick.bonthonneau@univ-rennes1.fr}
\and
T. Weich: Fakultät für Elektrotechnik, Informatik und Mathematik, Institut für Mathematik, Warburger Str. 100; \email{weich@math.uni-paderborn.de}}


\begin{abstract}
We present a method to construct a Ruelle-Pollicott spectrum for the geodesic 
flow on manifolds with strictly negative curvature and a finite number of hyperbolic cusps.

\end{abstract}

The spectrum of Ruelle-Pollicott resonances is a notion that was developped in the 1980's \cite{Pollicott-85,Ruelle-86,Rue87} 
to associate Axiom A flows \cite{Sma67} with a discrete set of complex numbers that describe its mixing properties. Let us recall 
their definition: If $\varphi_t$ is a flow on some manifold $M$, $d \mu$ an invariant measure and $A,B\in C_c^\infty(M)$ two observables, then we can define the correlation function 
\[
\rho_{A,B}(t) := \int_M \left(A\circ\varphi_t\right)\cdot B d\mu,
\]
as well as its Laplace transform $\hat\rho_{A,B}(s)$ which is holomorphic
for $\Re(s)>0$. Pollicott \cite{Pollicott-85} and Ruelle \cite{Rue87} proved that for Axiom A flows, this Laplace transform $\hat\rho_{A,B}$ extends meromorphically to a small strip $\Re(s)>-\varepsilon$ for a certain class of measures. Its poles are called Ruelle-Pollicott resonances of $\varphi_t$ with respect to $\mu$. In the following decade, several works \cite{Hay90,Fri95,Rug96,Kit99} were dedicated to obtaining sharp bounds on the maximal strip on which the continuation is possible in terms of the regularity of the flow. More recently it has been understood that these resonances can be seen as the discrete spectrum in the usual sense of the generator of the flow on some carefully chosen Banach spaces. They appear as the poles of the meromorphic continuation of the resolvent kernel. See \cite{Liverani-04,Butterley-Liverani-07,Faure-Sjostrand-10,Giulietti-Liverani-Pollicott-13,Dyatlov-Zworski-16,Dyatlov-Guillarmou-16}, and also \cite{BKL02,GL06,BT07,BT08,FRS08,Bal16} for the related case of hyperbolic diffeomorphisms. A wide generality of dynamical systems is considered in these articles, however all these results have in common that they assume the system has a compact \emph{trapped set}:
\[
K(\varphi_t): = \{ x\in M\ |\ \liminf_{t\to \pm \infty} d(x,\varphi_t(x)) < +\infty\}.
\] 
Since this is where the ``non trivial'' part of the dynamics happen, one can crucially use Fredholm theory.
In this paper, we consider a family of hyperbolic flows whose trapped set is not compact. For this class, we explain how Ruelle-Pollicott resonances can be defined. We are convinced that the methods developed in this article will also apply to more general settings and that they will lead to subsequent results such as meromorphic continuation of zeta functions, and decay of correlations results. However, the new arguments that we introduce to handle the noncompact trapped set are already a bit more involved than the usual ones. We have thus chosen to restrain ourselves to the following setting, where they can be cleanly developed:

The class of dynamical systems we are considering are geodesic flows on manifolds with cusps. We assume that $(N, g)$ is a complete, smooth, $(d+1)$-dimensional Riemannian manifold, that decomposes into a compact core with strictly negative variable sectional curvature, and a finite union of hyperbolic cusps with constant negative curvature, that are attached to this core (see Definition~\ref{def:admissible_cusp_manifold} for more precision). 
We consider the geodesic flow $\varphi_t$ acting on the cosphere bundle $M=S^*N$ and denote its vectorfield by $X$. When endowed with the Sasaki metric, $M$ is a (2d+1)-dimensional Riemannian manifold. From the inclusion $M \subset T^\ast N$, $M$ inherits the Liouville measure $\mu_L$ which is preserved by the geodesic flow, and gives finite volume to $M$. As $N$ has strictly negative curvature, the geodesic flow is uniformly hyperbolic and due to the particular structure of the cusps, its trapped set (in forward \emph{and} backward times) has full measure in $M$.

As $X$ is an antisymmetric unbounded operator on $L^2(M) := L^2(M, \mu_L)$, we deduce that its resolvent $\mathscr{R}(s):=(X-s)^{-1}:L^2(M)\to L^2(M)$ is a holomorphic family of bounded operators for $\Re(s)>0$. We prove:
\begin{theorem}\label{thm:continuation-resolvent}
The resolvent has a meromorphic continuation as a family of continuous operators $\mathscr R(s):C_c^\infty(M)\to\mathcal D'(M)$ from $\Re(s)>0$ to the whole complex plane and for any pole of this meromorphic continuation, the residue is a finite rank operator.
\end{theorem}

Note that for $A,B\in C_c^\infty(M)$ and $\rho_{A,B}$ the correlation function with respect to $\mu_L$, it is easy to check that for $\Re(s)> 0$
\[
\hat\rho_{A,B}(s) = \langle \mathscr R(s)A, B\rangle_{\mathcal D',C_c^\infty}\ .
\]
Thus the meromorphic continuation of the resolvent gives the continuation of $\hat \rho_{A,B}$ to the whole complex plane and its poles coincide with the poles of $\mathscr R(s)$ --- when $A$ and $B$ vary in $C^\infty_c(M)$. Consequently we call the poles of $\mathscr R(s)$ \emph{Ruelle-Pollicott resonances} of the geodesic flow (with respect to the Liouville measure). 

To the best of our knowledge such a global definition of Ruelle-Pollicott resonances of the geodesic flow on cusp manifolds was so far not known, even in the case of constant negative curvature manifolds. We therefore want to mention another consequence:
\begin{corollary}
 Let $\Gamma \subset PSL(2, \R)$ be a co-finite Fuchsian group such 
 that there is a torsion free, normal subgroup $\tilde \Gamma \unlhd \Gamma$
 of finite index\footnote{Note that a particular example of this situation is $\Gamma=PSL(2,\Z)$ and $\tilde\Gamma=\Gamma(2)$
the principal congruence subgroup (see e.g. \cite[Section 2.3]{Iwaniec-02}).}. Let us consider the orbifold $M_\Gamma:=S^*(\Gamma\backslash \mathbb H) \cong \Gamma\backslash PSL(2,\R)$. Then the resolvent $\mathscr{R}(s):=(X-s)^{-1}:L^2(M_\Gamma) \to L^2(M_\Gamma)$ has a meromorphic continuation 
 $\mathscr R(s): C_c^\infty(M_\Gamma)\to\mathcal D'(M_\Gamma)$.
\end{corollary}
\begin{proof}
 Theorem~\ref{thm:continuation-resolvent} applies to the smooth manifold $M_{\tilde \Gamma} 
 = \tilde \Gamma\backslash PSL(2,\R)$. The corollary then follows by definition of smooth functions and distributions on orbifolds and because the resolvent commutes with isometries. 
\end{proof}

We now want to mention some results related to Theorem~\ref{thm:continuation-resolvent}:  

In order to study eigenvalues of the Laplacian on moduli spaces Avila and Gou\"ezel \cite{AG13} develop a functional analytic framework for the Teichmüller flows which are also a class of dynamical systems with noncompact finite volume trapped set. They obtain a meromorphic continuation of the resolvent to a neighbourhood of zero (cf. \cite[Prop 3.3]{AG13}). It would probably be possible to adapt their method to geodesic flows on cusp manifolds in order to obtain a continuation to a small strip along the imaginary axis (instead of $\C$ in our case). However their functional analytic tools are quite different from ours.

Another series of related results have been obtained for the special case of surfaces of constant negative curvature with cusps. It has been shown in a series of articles, by Mayer, Morita and Pohl \cite{Mayer-91,Morita-97,Poh15,Poh16} that one can associate the geodesic flow with one dimensional expanding maps, using a carefully chosen discretization. Out of this discretization one can build transfer operators with discrete spectrum and these spectra have interesting relations to number theory and the theory of Maass cusp forms \cite{LZ01,MP13,BLZ15}. One should be able to recover these spectra as a subset\footnote{More precisely the connection should be to the so called first band of Ruelle-Pollicott resonances cf. \cite{Dyatlov-Faure-Guillarmou-15,GHW18,GHW18b}} of the resonances  defined from Theorem~\ref{thm:continuation-resolvent}. It will be subject to further research to establish this connection precisely.

As Ruelle-Pollicott resonances are an important tool to study decay of correlations, let us shortly mention that the question of mixing is not yet satisfactorily answered for our class of cusp manifolds: For constant curvature manifolds with cusps, exponential decay of correlations for the Liouville measure  was proved in \cite{Moore-87}, while for variable curvature only its mixing property is known \cite{Dalbo-Peigne-98}. Two other recent results on the mixing of Weil-Petersson geodesic flows on manifolds with cusp-like singularities\footnote{Note that their notion of cusp sigularities differs from ours: They consider singularities where the distance to the cusp is bounded, but the curvature is divergent} have been obtained in \cite{BMMW17a,BMMW17b}. We hope that the analytic tools that we develop in this article will prove to be helpful in the future for studying mixing properties of geodesic flows on manifolds with hyperbolic cusps. 

The meromorphic continuation of dynamical zeta functions is another important field where meromorphically continued resolvents of flow vector fields have successfully been applied. If $\mathcal P$ denotes the set of primitive periodic orbits of a hyperbolic flow and $\ell(\gamma)$ their lengths then the 
Ruelle zeta function is defined for $\Re(s)\gg 0$ by 
\[
 \zeta_R(s) := \prod_{\gamma \in \mathcal P}\left(1-e^{-s\ell(\gamma)}\right).
\]
Smale \cite{Sma67} raised the question if for Axiom A flows, the Ruelle zeta function\footnote{Actually Smale considered a different version of a dynamical zeta function which is rather an analog of Selberg's zeta function. The question of meromorphic continuation  is however trivially related to $\zeta_R$} has a meromorphic continuation to $\C$? This question has recently been affirmatively answered by Dyatlov and Guillarmou  \cite{DG18} following a long series of precedent works, that prove meromorphic continuation under additional assumptions \cite{Rue76,PP83, Fri95, GLP13, Dyatlov-Zworski-16, Dyatlov-Guillarmou-16} (We refer to \cite{Bal16, Zwo18} for a recent overview of the literature). In all the recent accounts \cite{GLP13, Dyatlov-Zworski-16, Dyatlov-Guillarmou-16, FT17, DG18} of  these meromorphically continued zeta functions, a meromorphically continued resolvent was the central ingredient. Consequently, Theorem~\ref{thm:continuation-resolvent} indicates\footnote{In fact the full statement of our result (Theorem~\ref{thm:full-theorem-resolvent}) already provides several further ingredients necessary for the meromorphic continuation of $\zeta_R$ such as the extension to differential forms and wavefront estimates} that the Smale conjecture could  hold for geodesic flows on cusp manifolds, i.e. beyond the class of Axiom A flows. So far such a result is only known in the particular case of constant negative curvature where it is a rather direct consequence of Selberg's trace formula.

Contrary to the ``classical'' Ruelle-Pollicott resonances, the definition of ``quantum'' resonances of the Laplace Beltrami operator $\Delta_N$ on a cusp manifold $N$ has been established for a long time starting with works of Maa{\ss}~\cite{Maass-1949} and Selberg~\cite{Sel69}. See the introduction of \cite{Lax-Phillips-Automorphic-76} for the constant curvature case, and \cite{CdV-81,Muller-83} for the variable curvature case. In fact the proof of our main result borrows ideas from the definition of quantum resonances (such as the compact Sobolev embedding, Lemma~\ref{lemma:compact-injection-H^1}).

Let us shortly sketch the further ingredients for proving the meromorphic continuation of the resolvent to the whole complex plane (Theorem~\ref{thm:continuation-resolvent}): As a first step we construct a family of \emph{anisotropic spaces} $H^{\escapeparam\mathbf{m}}$ that are adapted to the hyperbolic structure of the flow. These are Hilbert spaces of distributions on $M$, and $C^\infty_c(M)$ is dense in each $H^{\escapeparam\mathbf{m}}$. 
They are an adaptation of the spaces defined by Faure-Sj\"ostrand \cite{Faure-Sjostrand-10} and Dyatlov-Zworski \cite{Dyatlov-Zworski-16}. Using a mix of their techniques we obtain much in the same way a first parametrix, which inverts $X-s$ up to a smoothing remainder. However, this parametrix is --- contrary to the compact case  ---  not sufficient for a meromorphic resolvent. Therefore, it was necessary to introduce another technique. 
We chose to use ideas from Melrose's b-calculus to deal with the explicit form of the generator $X$ in the cusp. From the very nature of these techniques, they work independently of the dimension of $N$. 
It is not entirely clear whether the technique could be applied or not to the case that the curvature is not exactly equal to $-1$ in the cusps, only \emph{close} to $-1$. However, by analogy to the resonances of the Laplace operators we conjecture that the meromorphic continuation to the full complex plane will not hold true when assuming only pinched negative curvature in the cusps. 

Let us present the structure of the paper: 
In Section~\ref{sec:geometric-notations} we introduce the precise settings in which we are working and collect several properties of the geodesic flow on cusp manifolds, that will be crucial in the sequel. 
To prove our theorem, we then build a first parametrix in Section~\ref{sec:first-parametrix} following the arguments of \cite{Faure-Sjostrand-10,Dyatlov-Zworski-16}. 
The geometric construction of the escape function is presented; however, the technical microlocal lemmas are proved in Appendix~\ref{appendix:microlocal-tools}. 
Section~\ref{sec:continuation-indicial-resolvent} is devoted to introducing techniques adapted from b-calculus and proving the meromorphic continuation of the resolvent of a certain class of translation invariant operators. 
These operators show up precisely when restricting the geodesic flow to the zeroth Fourier mode in the cusp. In Section~\ref{sec:Black-Box} such a resolvent is used for the construction of a parametrix (up to compact remainder) of the geodesic flow vector field. Then, using analytic Fredholm theory we conclude on the meromorphic continuation announced in Theorem \ref{thm:continuation-resolvent}. 
In Section~\ref{sec:continuation-indicial-resolvent}  and 
\ref{sec:Black-Box} we work in a more general setting under a list of assumptions. This should allow for an easy generalization to more general settings (such as fibred or complex hyperbolic cusps) in the future. In Section~\ref{sec:explicit-computations} we finally compute explicitly the indicial roots for the b-operators associated to the geodesic flow on our class of cusp manifolds and we check that all the necessary assumptions in Section~\ref{sec:continuation-indicial-resolvent} and \ref{sec:Black-Box} are fulfilled. 

Note that in fact we prove more general and more precise versions of Theorem~\ref{thm:continuation-resolvent}. For 
example we continue the resolvent for a certain class of derivations on vector bundles (cf. Definition~\ref{def:admissible_vector_bundle}) including the geodesic vector field with smooth potential, Lie derivatives on perpendicular $k$-forms and general associated vector bundles over constant curvature manifolds (cf. Examples \ref{exmpl:potential}-\ref{exmpl:perp_k-forms}). Furthermore we give a precise description of the wavefrontset of the resolvent. For a full statement we refer the reader to Theorem~\ref{thm:full-theorem-resolvent} and \ref{thm:admissible-bundles}.

\textbf{Acknowledgement} We thank Fr\'ed\'eric Faure, 
Sébastien Gou\"ezel, Colin Guillarmou, Luis-Miguel Rodrigues, Gabriel Rivi\`ere and Viet Dang and the referee for helpful remarks and discussions. We acknowledge 
the hospitality of the Hausdorff Institute of Mathematics in Bonn where 
part of this work was done. TW acknowledges support by the Deutsche Forschungsgemeinschaft (DFG) through the Emmy Noether group ``Microlocal Methods for Hyperbolic Dynamics''(Grant No. WE 6173/1-1)

\section{Geometric preliminaries}\label{sec:geometric-notations}
\subsection{The geodesic flow on cusp manifolds}\label{sec:geodesic_flow_on_cusps}

Let us give a precise definition of the manifolds on which 
we are working.
\begin{definition}\label{def:admissible_cusp_manifold}
A manifold $N$ will be called an \emph{admissible cusp manifold} if the following assumptions hold. First, $(N,g)$ is a $(d+1)$-dimensional Riemannian manifold, connected and complete when endowed with the corresponding Riemannian distance. Second, it decomposes as the union $N_0\cup Z_1\cup\dots\cup Z_\kappa$. $N_0$ is a compact manifold whose boundary $\partial N_0$ is a finite disjoint union of $d$-dimensional torii. At each component $\ell=1\dots \kappa$ of $\partial N_0$ is glued the \emph{hyperbolic cusp} $Z_\ell$, which takes the form
\begin{equation}\label{eq:def-cusps}
Z_\ell = [a, + \infty[_y \times \left(\R^d/\Lambda_\ell\right)_{\theta}.
\end{equation}
(Here, $\Lambda_\ell$ is a lattice in $\R^d$, and we can impose the normalization condition that it is unimodular). We require that the metric $g$ has strictly negative curvature in the whole of $N$, and additionally, we fix for each $\ell=1\dots \kappa$
\begin{equation}\label{eq:def-cusp-metric}
g_{|Z_\ell} = \frac{dy^2 + d\theta^2}{y^2}.
\end{equation}
\end{definition}

\begin{figure}
\centering
\def\svgwidth{0.5\linewidth}
\begingroup%
  \makeatletter%
  \providecommand\color[2][]{%
    \errmessage{(Inkscape) Color is used for the text in Inkscape, but the package 'color.sty' is not loaded}%
    \renewcommand\color[2][]{}%
  }%
  \providecommand\transparent[1]{%
    \errmessage{(Inkscape) Transparency is used (non-zero) for the text in Inkscape, but the package 'transparent.sty' is not loaded}%
    \renewcommand\transparent[1]{}%
  }%
  \providecommand\rotatebox[2]{#2}%
  \newcommand*\fsize{\dimexpr\f@size pt\relax}%
  \newcommand*\lineheight[1]{\fontsize{\fsize}{#1\fsize}\selectfont}%
  \ifx\svgwidth\undefined%
    \setlength{\unitlength}{353.1551283bp}%
    \ifx\svgscale\undefined%
      \relax%
    \else%
      \setlength{\unitlength}{\unitlength * \real{\svgscale}}%
    \fi%
  \else%
    \setlength{\unitlength}{\svgwidth}%
  \fi%
  \global\let\svgwidth\undefined%
  \global\let\svgscale\undefined%
  \makeatother%
  \begin{picture}(1,0.46850246)%
    \lineheight{1}%
    \setlength\tabcolsep{0pt}%
    \put(0,0){\includegraphics[width=\unitlength,page=1]{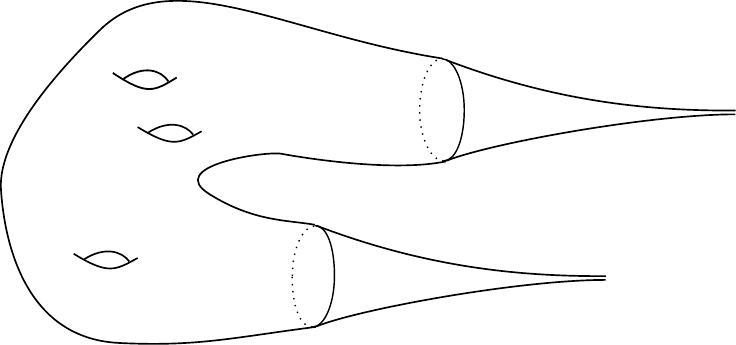}}%
    \put(0.21312973,0.09177468){\color[rgb]{0,0,0}\makebox(0,0)[lt]{\lineheight{1.25}\smash{\begin{tabular}[t]{l}$K<0$\end{tabular}}}}%
    \put(0.07812226,0.19340955){\color[rgb]{0,0,0}\makebox(0,0)[lt]{\lineheight{1.25}\smash{\begin{tabular}[t]{l}$N_0$\end{tabular}}}}%
    \put(0.4605946,0.07214408){\color[rgb]{0,0,0}\makebox(0,0)[lt]{\lineheight{1.25}\smash{\begin{tabular}[t]{l}$K=-1$\end{tabular}}}}%
    \put(0.63835021,0.29740667){\color[rgb]{0,0,0}\makebox(0,0)[lt]{\lineheight{1.25}\smash{\begin{tabular}[t]{l}$K=-1$\end{tabular}}}}%
    \put(0.79259982,0.35344646){\color[rgb]{0,0,0}\makebox(0,0)[lt]{\lineheight{1.25}\smash{\begin{tabular}[t]{l}$Z_2$\end{tabular}}}}%
    \put(0.61663501,0.02654636){\color[rgb]{0,0,0}\makebox(0,0)[lt]{\lineheight{1.25}\smash{\begin{tabular}[t]{l}$Z_1$\end{tabular}}}}%
    \put(-1.33701738,0.68675617){\color[rgb]{0,0,0}\makebox(0,0)[lt]{\begin{minipage}{1.61769514\unitlength}\raggedright \end{minipage}}}%
  \end{picture}%
\endgroup%

\caption{\label{fig:cusp_manifold} Schematic sketch of a cusp manifold}
\end{figure}
Then the sectional curvature is $-1$ in each cusp, and the volume of $N$ is finite. Since the sectional curvature of $N$ is pinched, we deduce that its geodesic flow $\varphi_t$ is a \emph{uniformly hyperbolic flow}\footnote{Often uniformly hyperbolic flows are also called \emph{Anosov flows}. Several authors use the term ``Anosov flow'' however only in the more narrow setting of hyperbolic flows on \emph{compact} manifolds. For this reason we refrain from using the term Anosov in our setting to avoid confusion.} on its cosphere bundle. More precisely we have:

\begin{proposition}\label{prop:invariant_bundles}
 Let $M=S^*N$ be the cosphere bundle of an admissible cusp manifold. There is a splitting
 \begin{equation}\label{eq:invariant_splitting}
  TM = E_0\oplus E_s\oplus E_u
 \end{equation}
 into $d\varphi_t$-invariant subbundles, which is Hölder continuous with uniform constants. Furthermore the angle between any pair of the invariant bundles is bounded from below by a uniform constant. Finally there are global constants $c, C, \beta, B > 0$ such that 
 \begin{align*}
  ce^{-Bt}\|v\| \leq \|(d\varphi_t) v\|\leq Ce^{-\beta t}\|v\|&\ \textup{for all }v\in E_s,\ t>0\\
  ce^{-Bt}\|v\| \leq \|(d\varphi_{-t}) v\|\leq Ce^{-\beta t}\|v\|&\ \textup{for all }v\in E_u,\ t>0.
 \end{align*}
\end{proposition}
 \begin{proof}
  Let $\tilde N$ be the universal cover of $N$. It is a simply connected, complete Riemannian manifold with pinched negative sectional curvature ($-k_{max}^2 < K <-k_{min}^2<0$), because the noncompact ends $Z_i$ are endowed with a constant negative curvature metric. For the same reason all derivatives of the sectional curvature are bounded. Thus Theorem 7.3 and Lemma 7.4 in \cite{PPS-12} apply to this situation and they provide the splitting into invariant bundles over $S^*\tilde N$ with the above properties. As the invariant bundles are invariant under isometries, taking the quotient we obtain the desired result. 
 \end{proof}

For the proof of Theorem \ref{thm:continuation-resolvent} it will be crucial 
to have a precise understanding of the geometry and the dynamics on 
the noncompact ends of $S^\ast N$. We therefore start by introducing explicit 
coordinates on $S^\ast Z_\ell$. In order to simplify the notation we will drop the 
indices $\ell = 1\dots \kappa$ that number the cusps.

Recall that a cusp is $Z=[a,\infty[\times \R^d/\Lambda$ and since we have assumed that $\Lambda$ is \emph{unimodular}, we have \emph{canonical coordinates} $y\in [a,\infty[, \theta\in \R^d/\Lambda$. In many cases it will be convenient to perform the change of variables $r=\log y\in [\log a,\infty[$ and the metric becomes 
\begin{equation}\label{eq:metric-r-theta}
g=dr^2+e^{-2r}d\theta^2.
\end{equation}
A single cusp has the \emph{local isometry pseudo-group} given by $\R\times \R^d$ which is 
realized by linear scaling and translations in the $y, \theta$ variables
\begin{equation}\label{eq:def-local-isometries}
 T_{\tau, \theta_0}(y,\theta) := (e^{\tau}y, e^{\tau}\theta +\theta_0),
\end{equation}
or in $r,\theta$-variables
\begin{equation}\label{eq:local-isometries-r-theta}
 T_{\tau, \theta_0}(r,\theta) := (r+\tau, e^{\tau}\theta +\theta_0).
\end{equation}
Using the $y,\theta$ variables we can write $\xi\in T^*_{y,\theta}Z$ as 
$\xi=Ydy+Jd\theta$ for $Y\in\R$ and $J\in \R^d$ and the Riemannian norm of such 
a cotangent vector is given by 
\begin{equation}\label{eq:metric-cotangent}
|\xi|_g = y\sqrt{Y^2+|J|_{\R^d}^2}.
\end{equation}
Elements $\xi\in S^*_{y,\theta}Z$ of a cosphere fibre are thus in bijection with 
$\zeta:= y(Y, J)\in \mathbb S^d\subset \R^{d+1}$. In particular the cosphere bundle over the cusp is trivializable 
$S^*Z\cong Z_{(y,\theta)}\times \mathbb S^d_\zeta$. The usual metric on $S^\ast Z$, the Sasaki metric (see e.g. \cite{Gudmundsson-Kappos} for an easily accessible introduction), is not a product metric. However, one can check (see the expression of the Sasaki metric in \cite[Section C.2]{Bonthonneau-2}) that it is equivalent to the \emph{product metric} $g_Z\otimes g_{\mathbb S^d}$ where $g_{\mathbb S^d}$ is the usual metric
on the sphere. We will use the product metric in the sequel.

For the study of the geodesic flow some more precise variables on the spheres are useful. We choose a orthonormal base of coordinates $\theta_1,\dots,\theta_d$ in $\R^d$. We fix $(yY=1,J=0) \simeq y^{-1}dy$ to be \emph{zenith}, and $y^{-1}d\theta_1$ the \emph{azimuthal reference}. With these conventions, a point $\zeta \in S^\ast_{y,\theta}Z$ is non-ambiguously determined by its \emph{inclination} $\varphi$ --- the angle it makes with the zenith --- and its azimuthal position, $u\in \Ss^{d-1}$ which is determined by the choice of base in $\R^d$. As a point in $\R^{d+1}$, $\zeta = (\cos\varphi, \sin\varphi u)\simeq y^{-1}\cos\varphi dy + y^{-1} \sin\varphi u\cdot d\theta$.

We single out two important points, the North pole $\mathcal N\in \mathbb S^d$ with $\varphi=0$ that corresponds to the cotangent element $y^{-1}dy=dr\in S^*Z$ pointing into the direction of the cusp and the South pole $\mathcal S$ corresponding to $-y^{-1}dy=-dr$ pointing perpendicularly to the bottom of the cusp. 

The geodesic flow is known to be the Hamiltonian flow with Hamiltonian 
\[
\mathscr h(x,\xi)=\tfrac{1}{2}g_x(\xi,\xi) = \tfrac{1}{2}y^2(Y^2+|J|_{\R^d}^2),
\]
and a straigthforward calculation
with the canonical symplectic structure on $T^*Z$ gives the associated 
Hamiltonian vector field
\[
 y^2Y\partial_y + y^2J\cdot\partial_\theta - y(Y^2+J^2)\partial_Y.
\]
Restricting this vector field to $S^*Z$ and using the spherical coordinates 
$\varphi,u$ we obtain an explicit expression for the geodesic vector 
field
\begin{align*}
 X&= y\cos(\varphi)\partial_y+y\sin(\varphi)u\cdot \partial_\theta + 
 \sin(\varphi)\partial_\varphi\\
 &=\cos(\varphi)\partial_r+e^r\sin(\varphi)u\cdot \partial_\theta + 
 \sin(\varphi)\partial_\varphi.
\end{align*}
Note that $u\cdot \partial_\theta$ is understood after identifying
$u\in\mathbb S^{d-1}\subset \R^d\cong T_\theta(\R^d/\Lambda)$.

The dynamics of the geodesic flow vectorfield is  illustrated in 
Figure~\ref{fig:cusp_geodesics}. Let us emphasize two important properties 
of the geodesic flow dynamics on $S^*Z$:

\begin{figure}
\centering
\def\svgwidth{0.7\linewidth}
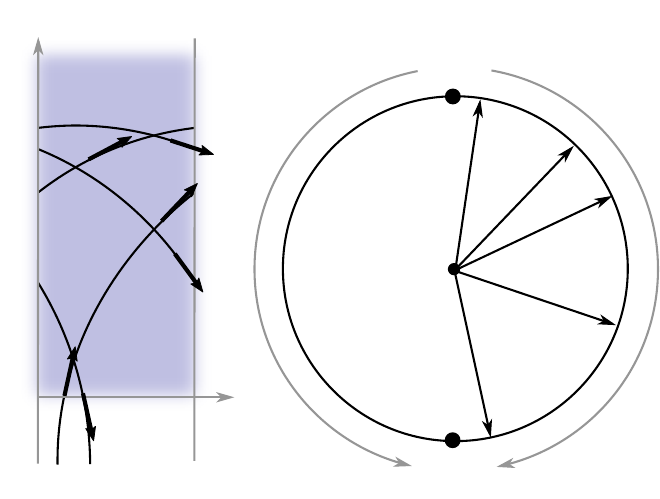
\caption{\label{fig:cusp_geodesics} This figure illustrates the dynamics of the geodesic flow on a cusp. The left part shows the fundamental domain 
of a cusp (for $d=1$) in the $y,\theta$ variable. The black solid line is the 
trace of a geodesic projected from $S^*Z$ to $Z$. The arrows indicate the 
direction of the flow and correspond to the cotangent vectors. On the right each of 
these arrows is represented by its $\zeta$ coordinate in $\mathbb S^d$, evidencing the transient dynamics from $\mathcal N$ to $\mathcal S$.}
\end{figure}
\begin{enumerate}[(A)]
	\item The Hamiltonian $\mathscr h$ is independent of the $\theta$ variable, which implies that the corresponding momentum variable $u$ is a constant of motion under the geodesic flow. 
	\item The dynamics of the variable $\zeta\in \mathbb S^d\cong S^*_{y,\theta}Z$ is decoupled from the dynamics on $Z$. By property (A) this dynamics is even rotationally invariant around the axis through $\mathcal N$ and $\mathcal S$ and it is precisely the gradient flow on the sphere $\mathbb S^d$ with the obvious height function.
\end{enumerate}

This has the following consequence for the dynamics of the geodesic flow on the 
cusp. Assume that trajectories stop when they reach the lower boundary $y=a$. Then the only \emph{wandering} trajectories are those with $\zeta = \mathcal N$ or $\zeta=\mathcal S$. 
They correspond to the geodesics that leave or enter the cusp, parallel to 
the $y$-axis. All other trajectories only rise up to a finite height into the cusp
and are thus ``trapped''. However, by chosing $\zeta$ arbitrary close to $\mathcal N$ this height can be made arbitrary large and the trapped set is noncompact.  
Since the non-compactness of the trapped set is the central problem in extending the techniques of \cite{Faure-Sjostrand-10,Dyatlov-Zworski-16}, these regions around $\mathcal N$ and $\mathcal S$ will become crucial in the analysis.

Finally let us add a third remark that is not directly related to the dynamics of
the geodesic flow, but rather to its action as a differential operator.
\begin{enumerate}[(A)]\setcounter{enumi}{2}
	\item As the geodesic vector field commutes with local isometries, it commutes in 
particular with the $\R^d$-action by translation and thus preserves the Fourier modes
in the $\theta$ variable. If $k$ is an element in the dual lattice 
$\Lambda^*\subset \R^d$ then, restricting the geodesic flow vector field
to the Fourier modes $e^{ik\theta}$ yields a differential operator
\begin{equation}\label{eq:def-X_k}
X_k = \cos(\varphi)\partial_r + \sin(\varphi)\partial_\varphi + ie^r\sin(\varphi)u\cdot k  \ .
\end{equation}
When $k=0$, this is a vector field with coefficients that do not depend on $r$.
\end{enumerate}

\begin{remark}
The structure of the flow restricted to the zeroth Fourier mode is essential to our proof. Indeed, since it is translation invariant, we can use techniques adapted from Melrose's b-calculus to find an exact inverse for the model flow on a ``full'' cusp (cf. Section~\ref{sec:continuation-indicial-resolvent}). The fact that in the other Fourier modes the flow does not have such a nice structure is compensated by the fact that we have a compact injection for functions in $H^1$ whose zeroth Fourier mode vanishes in each cusp (cf. Lemma~\ref{lemma:compact-injection-H^1}).
\end{remark}
\subsection{Admissible vector bundles}\label{sec:admissible_vb}
As mentioned in the introduction we want to prove the meromorphy of the resolvent not only for the 
geodesic vector field acting on functions but also for a larger class of admissible vector bundles. 
In order to precisely define these admissible vector bundles let us first recall how to write the 
noncompact ends $S^*Z_\ell$ as locally homogeneous spaces:

Given a cusp $Z_\ell=[a_\ell,\infty[\times \R^d/\Lambda_\ell$ we will consider the associated 
\emph{full cusp} to be the space $Z_{\ell,f}=(\R_+)_y\times(\R^d/\Lambda_\ell)_\theta$ with the metric 
$g$ defined in Equation \eqref{eq:def-cusp-metric} extended to $Z_{\ell,f}$ in the obvious way.
Let $\groupG = SO(d+1,1)$, then using the Iwasawa decomposition we can write  $\groupG=\groupN\groupA\groupK$,
where $\groupA=(\R_+,\cdot)$, $\groupN=(\R^d,+)$ are abelian groups
and $\groupK=SO(d+1)$ is the maximal compact subgroup in $\groupG$. 
Then a full cusp is simply the double quotient 
$Z_{\ell,f} = \Lambda_\ell \backslash\groupG/ \groupK$ where we consider $\Lambda_\ell\subset 
\groupN\cong\R^d$. The unit cosphere bundle can then be simply written as $S^*Z_{\ell,f} = 
\Lambda_\ell\backslash\groupG/\groupM$ where $\groupM=SO(d)$ (see e.g.
\cite{hilgert2005ergodic,GHW18} for more details). Recall furthermore that the 
Bruhat decomposition on the Lie algebra 
\begin{equation}\label{eq:Bruhat-decomposition}
\mathfrak g= \mathfrak m \oplus \mathfrak a \oplus \mathfrak n_+ \oplus \mathfrak n_-
\end{equation}
is $\textup{Ad}\groupM$ invariant. Accordingly $\groupG/\groupM$ is a reductive 
homogeneous space and for any orthogonal representation $(\tau,V)$ of $\groupM$ the associated vector
bundle $\groupG\times_\tau V$ is a homogeneous Riemannian vector bundle with a canonical compatible connection. 

Now we can define admissible vector bundles:
\begin{definition}\label{def:admissible_vector_bundle}
Let $N=N_0\cup_{\ell=1}^\kappa Z_\ell$ be an admissible cusp manifold in the sense of 
Definition~\ref{def:admissible_cusp_manifold} and $M=S^*N$. Let $L\to M$ be a Riemannian bundle endowed with a compatible connection $\nabla$. $L$ is an \emph{admissible vector bundle} if for each cusp $Z_\ell$, $\ell=1,\ldots,\kappa$ there is an orthogonal $\groupM$-representation 
$(\tau_\ell, V_\ell)$ such that $L|_{Z_\ell}$, coincides with the associated 
vector bundle 
\begin{equation}\label{eq:def-L-tau}
L_{\ell,\tau_\ell} = \Lambda_\ell\backslash \groupG  \times_{\tau_\ell}V_\ell.
\end{equation}
(the Riemannian bundle metric and connection of $L$ are also assumed to coincide with those of the 
associated bundle.)

Let $\mathcal{X}$ be a derivation on sections of $L$ that lifts the geodesic flow vector field $X$. That is to say that it satisfies the Leibnitz relation
\begin{equation}\label{eq:Leibnitz-relation}
\mathcal{X}(fs) =(Xf)s+f\mathcal{X}s \textup{ for }f\in C^\infty(M), s\in C^\infty(M,L).
\end{equation}
We say that $\mathcal{X}$ is an \emph{admissible lift} of $X$ if for each cusp $Z_\ell$, there is a fixed $A_\ell\in \End(V_\ell)^{\groupM}$ such that when restricted to $L|_{Z_\ell}$, $\mathcal{X}$ acts as
\begin{equation}\label{eq:def-X-ell}
\mathcal{X}_\ell := \nabla_X +A_\ell.
\end{equation}
\end{definition}

Let us mention three important examples of admissible vector bundles and differential operators:
\begin{example}\label{exmpl:potential}
Let $V\in C^\infty(M)$ be so that in each cusp, $V$ is just a constant. Then $X + V$ is an admissible operator on the trivial bundle.
\end{example}

\begin{example}\label{exmpl:associated_vb}
Let $\Gamma \subset \groupG$ be a non uniform torsion free lattice. Then $\Gamma \backslash \groupG/ \groupK$ is a non-compact manifold of constant curvature whose ends are cusps in the sense we have defined; it is thus an admissible cusp manifold. There is a finite number of ends. Given an orthogonal representation $\tau$ of $\groupM$ on $V$, we can then construct globally the bundle $L_\tau = \Gamma \backslash \groupG \times_\tau V$ and the corresponding connection. Then the operator $\nabla_X$ is an admissible lift of the geodesic flow on $M = \Gamma \backslash \groupG/ \groupM$.
\end{example}

\begin{example}\label{exmpl:perp_k-forms}
Let us take $N$ an admissible cusp manifold, $M=S^\ast N$ and $X$ the corresponding geodesic vector field. We can consider the Lie derivative $\mathcal{L}_X$ acting on $\Lambda(T^\ast M)$, the bundle of forms of arbitrary degree over $M$, it is an admissible lift of $X$. Further define
\begin{equation}\label{eq:def-Lambda-perp}
\Lambda^\perp(T^\ast M):= \{ \omega \in L(T^\ast M)\ |\ \imath_X\omega = 0\}.
\end{equation}
This sub-bundle of $\Lambda(T^\ast M)$ is invariant under $\mathcal{L}_X$. Also, $\mathcal{L}_X$ preserves the Liouville one form $\alpha$, which is a \emph{contact} one-form. In particular, we can identify the action of $\mathcal{L}_X$ on $\Lambda^\perp(T^\ast M)$ with the action of $\mathcal{L}_X$ on $\Lambda((\ker \alpha)^\ast)$. This is also an admissible lift of $X$.
\end{example}

\section{Anisotropic space and first parametrix}
\label{section:anisotropic-space}

The main idea that was presented in \cite{Faure-Sjostrand-10} was to resort to usual semi-classical techniques to prove the meromorphic continuation of the resolvent of the flow generator for Anosov flows on compact manifolds. This is not the only method available for compact manifolds --- see \cite{Butterley-Liverani-07} --- but it is the one we will extend to our case. Another paper \cite{Dyatlov-Zworski-16} used propagation of singularities to obtain the wavefront set of the resolvent, in order to simplify the proof of meromorphic continuation of the zeta functions. We will use a mixture of both, since we use the approach of \cite{Faure-Sjostrand-10} to continue the resolvent, and ideas from \cite{Dyatlov-Zworski-16} to obtain the wavefront set of the resolvent.

We consider $L\to M=S^\ast N$ an admissible bundle and $\mathcal{X}$ an admissible lift of $X$ the geodesic flow on $M$. Since we will use semi-classical techniques, we introduce a small parameter $0<h \leq h_0$, and we let $\mathbf{X} := h \mathcal{X}$. We refer to Appendix~\ref{appendix:microlocal-tools} where we collect the definition of the notions of microlocal and semiclassical analysis (pseudodifferential operators, symbol classes, \ldots) which we will use in the sequel.

The first result in this section is:
\begin{proposition}\label{prop:first-parametrix}
For each $\escapeparam >0$, and $h>0$, we can build a space of $L$-valued distributions $H^{\escapeparam\mathbf{m}}$ on $M$ that contains $C^\infty_c(M, L)$, and a pseudo-differential operator $Q$ microsupported in an arbitrarily small neighbourhood of the zero section in the fibers of $T^\ast M$, so that there exists $h_0>0$, so that for $0<h\leq h_0$, and for $|\Im s| < h^{-1/2}$ and $\Re(s) > -\escapeparam$,
\[
\mathbf{X} -Q - hs\text{ is invertible and } \| (\mathbf{X} - Q - h s )^{-1} \|_{H^{\escapeparam\mathbf{m}}} = \mathcal{O}(1/h).
\]
As $h$ varies, the spaces $H^{\escapeparam\mathbf{m}}$ remain the same as vector spaces, with equivalent norms.
\end{proposition}
The space $H^{\escapeparam\mathbf{m}}$ will take the form (see definition \ref{def:anisotropic-space})
\[
\Op(e^{-\escapeparam G})\cdot L^2( M, L).
\]
In this formula, $G$ denotes a so-called \emph{escape function}, and $\Op$ a semi-classical quantization that we define in Appendix \ref{appendix:microlocal-tools} (see equation \eqref{eq:def-quantization-global}). The construction of $G$ will be done first for $X$ acting on functions. Then the general case is obtained by tensorizing $\Op(e^{-\escapeparam G})$ with the identity $\mathbb{1}\in\End(L)$. 

\begin{remark}
As should be clear after reading the proof, the construction of the escape function is \emph{local} in the sense that it can be done in the universal cover. In particular, Proposition \ref{prop:first-parametrix} should hold in any geometrically finite negatively curved manifold whose universal cover has \emph{bounded geometry}. We do not prove this general result because that would require the construction of an explicit quantization with uniform bounds on these noncompact spaces. This seemed too much a detour considering that our aim is to study cusp manifolds and that a suitable quantization in this setting has already been developed by the first author in \cite{Bonthonneau-2}. 
\end{remark}

\subsection{Building the escape function}
\label{sec:escape-function}
In this subsection we want to construct an escape function in complete analogy to 
\cite[Lemma 1.2]{Faure-Sjostrand-10}. As we deal with a noncompact situation we however
have to take care that the required uniform bounds hold. 

The escape function $G$ will be a function on the cotangent bundle $T^\ast M$ and we introduce 
the decomposition
\begin{equation}\label{eq:Tstar_decomp}
T^\ast M = E^\ast_0 \oplus E^\ast_u \oplus E^\ast_s,
\end{equation}
so that $E^\ast_0 = \R \alpha$ where $\alpha$ is the Liouville one-form $-\xi\cdot dx$. Furthermore $E^\ast_u = (E^u \oplus E^0)^\perp$ and $E^\ast_s = (E^s \oplus E^0)^\perp$. 

We have to introduce some notations regarding the dynamics. We lift the geodesic flow $\varphi_t$ symplectically to the flow
\[
\Phi_t :(x,\xi) \mapsto  (\varphi_t(x), (d_x \varphi_t^\ast)^{-1}\cdot \xi).
\]
It is the Hamiltonian flow associated to the Hamiltonian 
\begin{equation}\label{eq:def_p}
p(x,\xi):= \xi\cdot X(x),
\end{equation}
which is the symbol of $-i X$
and we denote by $X_\Phi$ its hamiltonian vectorfield. The decompositions (\ref{eq:Tstar_decomp}) is preserved by the flow, and
\begin{equation}\label{eq:estimate-E-ast-s-exponential-decay}
(x,\xi) \in E^\ast_s \Rightarrow |\Phi_t(x,\xi) | \leq C e^{-\beta t} | (x,\xi)|\text{ for }t>0.
\end{equation}
(Likewise in negative time for $E^\ast_u$).
\begin{lemma}\label{lemma:escape-function}
For any sufficiently small uniform conical neighbourhoods $N_0, N_u, N_s$ of $E_0^*, E_u^*, E_s^*$ there are constants $C_G, R>0$ such that for any $\delta>0$, there is an escape function $G\in C^\infty(T^\ast M)$ with 
\begin{enumerate}[(i)]
	\item $X_\Phi G  > 1$ outside of $\{ | \xi | < R \delta\}\cup N_0$.
	\item $X_\Phi G \geq 0$ globally on $\{ |\xi| > \delta\}$. 
	\item For $|\xi| > R \delta$
\[
G(x,\xi) = \begin{cases} 	+ C_G \log |\xi| + \mathcal{O}(1)	&\xi\in N_u, \\
                        	- C_G \log |\xi| + \mathcal{O}(1)	&\xi\in N_s, \\ 
			 				0 									&\xi\in N_0.
           \end{cases}
\]
	\item $e^G \in S^{\mathbf{m}}_{\log}(M)$: it is an anisotropic symbol of order $\mathbf{m}(x,\xi)$, with $\mathbf{m} \in S^0_{cl}(M)$  being a
0-homogeneous classical symbol  (see Definition~\ref{def:C-infinity-structure-compactified} for a definition of classical symbol classes) with
\[
\mathbf{m}(x,\xi/|\xi|)=	\begin{cases}	+ C_G  	&\xi\in N_u, \\
                        					- C_G 	&\xi\in N_s, \\ 
			 								0 		&\xi\in N_0.
           					\end{cases}
\]
\end{enumerate}
\end{lemma}
In order to prove Lemma~\ref{lemma:escape-function} it will be helpful to restrict $\Phi_t$ to the unit sphere bundle $S^*M$. In order to do this, let us interpret $S^*M \cong (T^*M\setminus\{0\}) /\R$ where $\R$ acts on each fibre by linear multiplication. Then, by homogeneity, $\Phi_t$ factors to a flow $\tilde\Phi_t:S^*M\to S^*M$ with vector field $X_{\tilde\Phi}$. By an abuse of notations, we can see $E^\ast_0$, $E^\ast_u$ and $E^\ast_s$ as subsets of $S^\ast M$. From the uniform estimates in Proposition~\ref{prop:invariant_bundles} we obtain:
\begin{lemma}\label{lem:S_star_M_dynamic}
For $\epsilon >0$, let $U^\epsilon_u\subset S^\ast M$ be the  $\epsilon$ neighbourhood of $ E^*_u$ and likewise let $U^\epsilon_{0,s}\subset S^\ast M$ be the $\epsilon$ neighbourhood of $E^*_0\oplus E^*_s$. Then there exists $\epsilon >0$ such that $U^\epsilon_u$ and $U^\epsilon_{0,s}$ are disjoint. Furthermore, for any fixed $\epsilon$ as above, there is a finite maximal transition time $\tau_{\max} > 0$ such that for $t\geq \tau_{\max}$,
 \begin{enumerate}
	\item For all $(x,\xi)\in S^*M \setminus U^\epsilon_u$, $\ \tilde \Phi_{-t}(x,\xi)\in U^\epsilon_{0,s}$.
	\item For all $(x,\xi)\in S^*M \setminus U^\epsilon_{0,s}$, $\ \tilde \Phi_{t}(x,\xi)\in U^\epsilon_u$.
 \end{enumerate}
Finally for any $T>0$ there is $\epsilon'>0$ such that 
 $U_u^{\epsilon '}\subset \tilde \Phi_T (U_u^\epsilon)$ and $U_{0,s}^{\epsilon '}\subset \tilde \Phi_{-T} (U_{0,s}^\epsilon)$.
 
The same statement holds for $E^*_0\oplus E^*_u$ and $E^*_s$.
\end{lemma}

\begin{proof}[proof of lemma \ref{lemma:escape-function}]
Let us first construct the weight function $\mathbf{m}$. This decomposes into two symmetrical steps. 

Take an $\epsilon>0$ from Lemma~\ref{lem:S_star_M_dynamic} 
such that $U^{3\epsilon}_u\cap U^{3\epsilon}_{0,s} = \emptyset$. In a first step we want to smooth the characteristic functions $\mathbb{1}_u^{2\epsilon}, \mathbb 1_{0,s}^{2\epsilon} \in L^1_{\mathrm{loc}}(S^*M)$ on these two sets. 
Therefore, take $\tilde \chi_m \in C_c^\infty(]-\epsilon,\epsilon[)$, with $\tilde \chi_m\geq0$ and $\tilde \chi_m(0)>0$. Then we define a smoothing kernel $\chi_m\in C^\infty(S^*M\times S^*M)$ by
\[
\chi_m(x,x'):= \frac{\tilde{\chi}_m(d(x,x'))}{\int_{S^*M} \tilde{\chi}_m(d(x,x'')) dx''} ,
\]
and we denote by $K_m$ the corresponding smoothing operator. Now we define the function $m_{0,s}^u:= K_m(\mathbb{1}_u^{2\epsilon} - \mathbb 1_{0,s}^{2\epsilon})$ which, by the construction of the smoothing operator, fulfills the following assumptions:
\begin{enumerate}
	\item $m_{0,s}^u \in \mathscr{C}^\infty(S^\ast M)$ --- this means that all derivatives are bounded, see the discussion at the start of Appendix \ref{appendix:microlocal-tools}.
	\item $m_{0,s}^u$ equals $+1$ on $U_u^{\epsilon}$ and $-1$ on $U_{0,s}^{\epsilon}$.
	\item $m_{0,s}^u$ takes values in $[-1,1]$. 
\end{enumerate}
Now take the time $T=2 \tau_{\max}$ (with $\tau_{\max}$ being the transition time from Lemma~\ref{lem:S_star_M_dynamic}) and set
\[
m^+_T = \int_{-T}^T m_{0,s}^u \circ \tilde \Phi_t dt.
\]
We have
\[
X_{\tilde\Phi} m^+_T  = m_{0,s}^u \circ \tilde \Phi_T - m_{0,s}^u\circ \tilde \Phi_{-T}.
\]
By Lemma~\ref{lem:S_star_M_dynamic} for any $(x,\xi)\in S^*M$ either $\tilde \Phi_{T}\in U^\epsilon_u$ or 
$\tilde \Phi_{-T}\in U^\epsilon_{0,s}$.
Since $m_{0,s}^u$ takes values in $[-1,1]$, we deduce that everywhere
\[
X_{\tilde\Phi} m^+_T \geq 0.
\]
Now let us define $V_u:= \Phi_T(S^*M\setminus U^\epsilon_{0,s}) \subset U^\epsilon_u$ and $V_{0,s}:= \Phi_{-T}(S^*M\setminus U^\epsilon_u) \subset U^\epsilon_{0,s}$. Then, for $(x,\xi)\notin (V_u\cup V_{0,s})$, $\tilde \Phi_{T}(x,\xi) \in U^\epsilon_u$ and $\tilde \Phi_{-T}(x,\xi) \in U^\epsilon_{0,s}$ and consequently
\[
X_{\tilde\Phi} m^+_T(x,\xi) =2.
\]
On the other side, if $(x,\xi)\in V_u$, then $m_{0,s}^u( \tilde\Phi_t(x,\xi)) =1$ for $t>-\tau_{\max}$ and from the definition of $T$, we deduce that $m_T^+(x,\xi) \geq T$. For the same reason $m_T^+<-T$ on $V_{0,s}$. Finally with $\epsilon'>0$ from Lemma~\ref{lem:S_star_M_dynamic}, we deduce that $m^+_T$ is constant equal to $2T$ (resp. $-2T$) on $U_u^{\epsilon'}$ (resp. $U^{\epsilon'}_{0,s}$).

The second step is to build a similar function $m^-_T$ replacing $E^\ast_u$ by $E^\ast_s$, and going through the same procedure. Taking
\[
m = \frac{m^+_T + m^-_T }{2},
\]
we get
\begin{enumerate}[(a)]
	\item $m\in \mathscr{C}^\infty(S^\ast M)$.
	\item $X_{\tilde \Phi} m \geq 0$ in $S^\ast M$.
	\item $X_{\tilde\Phi}m \geq 1$ on $S^*M \setminus (V_u\cup V_s\cup U^\epsilon_0)$.
	\item On $U^{\epsilon'}_u$ (resp. $U^{\epsilon'}_s$, $U^{\epsilon'}_0$), $m$ equals $2T$ (resp. $-2T$, $0$).
	\item $m>T$ on $V_u$ and $m<-T$ on $V_s$.
\end{enumerate}
The actual weight function $\mathbf{m}$ will be $m$ multiplied by a constant, that we will determine at the end.

Now comes the second part of the proof: building the symbol $G$. Choose
$N_u$ (resp. $N_s$, $N_0$) to be the cone in $T^*M$ generated by $V_u\subset S^*M$ (resp. $V_s, U_0^\epsilon$). We want to choose a symbol $f\in S^1(M)$ to be a positive \emph{elliptic} symbol, so that outside of $|\xi| < \delta$, on $N_0$ it equals $|p|$. We also want that on $N_u$ (resp. $N_s$) it satisfies $X_\Phi \log f \geq \beta/2$ (resp $\leq - \beta/2$). We would like to set $f$ to be just the norm $|\xi|$ in a neighbourhood of $E^\ast_u \oplus E^\ast_s$, but this is not suitable because the constant $C$ in the estimate \eqref{eq:estimate-E-ast-s-exponential-decay} is not necessarily $1$. However, we find that for $(x,\xi)\in E^\ast_s$,
\[
X_\Phi\left(\frac{1}{2t}\int_{-t}^t \log |\Phi_s(x,\xi)| ds \right) \leq \frac{\log C}{2t} - \beta.
\]
This suggests to pick $T' > 2\log(C)/\beta$ and define for $(x,\xi)$ in a fixed conical neighbourhood of $E^\ast_u \oplus E^\ast_s$
\[
f_{us}(x,\xi):= \exp\left( \frac{1}{2T'}\int_{-T'}^{T'} \log |\Phi_t(x,\xi)| dt \right).
\]
This is not a norm anymore, but is still $1$-homogeneous and smooth --- except at $0$. On $E^\ast_s$, $X_\phi \log f_{us} \leq -3\beta/4$, so that if $\epsilon>0$ was chosen small enough, $X_\phi \log f_{us} \leq - \beta /2 $ in $N_s$. We also have the corresponding estimates in $N_u$. We can piece together $f_{us}$ and $|p|$ around $N_0$ to obtain a globally defined elliptic $1$-homogeneous symbol. Let $c_f$ be its infimum on $\{ |\xi| = 1 \}$.

We have all the pieces to define
\[
G(x,\xi) = C_G' \left[1-\chi_G(|\xi|/\delta)\right] m\left( x, \frac{\xi}{|\xi|} \right) \log \frac{2 f(x,\xi)}{c_f \delta}.
\]
$C_G'>0$ is a constant fixed later and $\chi_G$ is a $C^\infty_c(]-1,1[)$ function, that equals $1$ in $[-1/2,1/2]$ and takes values between $0$ and $1$. It is there to ensure that $G$ is smooth at $\xi = 0$. We can check that $G\geq 0$. By the properties of $m$ from above, we directly deduce that Lemma~\ref{lemma:escape-function}(iii) holds. 

Now, we can compute
\[
\begin{split}
X_\Phi G &= - C_G' (X_\Phi \chi_G) m \log \frac{2f}{c_f \delta} \\
	&+ C_G'(1-\chi_G)\left[ (X_{\tilde\Phi} m) \log\frac{2f}{c_f\delta} + m \frac{X_\Phi f}{f}\right].
\end{split}
\]
Let us discuss the different terms: 

$X_\Phi\chi_G$ vanishes outside $\{|\xi|<\delta\}$, thus the first line  is irrelevant for the properties (i) and (ii) of Lemma~\ref{lemma:escape-function}. Let us consider the second line case by case:
\begin{itemize}
	\item \textbf{If $(x,\xi) \notin (N_0\cup N_u\cup N_s)$:} Note that $|m|$ and $|\frac{X_\Phi f}{f}|$ are globally bounded by a constant $C_0$. By property (c) above $X_{\tilde\Phi}m>1$. By the fact that $f$ is elliptic, there is a constant $R>0$ such that when $\{|\xi|>R\}$, $\log(2f/c_f)> 1 + C_0^2$. Then for $|\xi|>R\delta$, we also have $\log (2f/c_f\delta) > 1 + C_0^2$, thus $X_\Phi G > C_G'$ for $|\xi|>R\delta$
	\item \textbf{If $(x,\xi) \in N_u$:} Now we only know that $X_{\tilde\Phi}m\geq 0$, so we need a uniform lower bound for the second term. But from the choice of $f$, it is precisely there that $X_\Phi f /f > \beta/2$. Together with the property (e) of $m$ above, we deduce $X_\Phi G > \beta T C_G'/2$ for $|\xi|>\delta$.
	\item \textbf{If $(x,\xi) \in N_s$:} As in the previous case, we obtain $X_\Phi G > \beta TC_G'/2$ for $|\xi|>\delta$.
	\item \textbf{If $(x,\xi) \in N_0$:} As $f$ is a function of $p$ on $N_0$ and $X_\Phi$ is the Hamiltonian flow of $p$, we have $X_\Phi f=0$. Since $X_{\tilde \Phi} m\geq 0$ we conclude $X_\Phi G \geq 0$ for $|\xi|>\delta$.
\end{itemize}

Let $N_u'$ (resp. $N_s'$) be the conical neighbourhood corresponding to $U^{\epsilon'}_u$ (resp $U^{\epsilon'}_s$). We have $N_u' \subset N_u$ and $N_s' \subset N_s$. On $N_u'$ (resp. $N_s'$), $G = 2 C_G' T \log |\xi| + \mathcal{O}(1)$ (resp. $- 2 C_G' T \log |\xi| + \mathcal{O}(1)$). So we choose $C_G' \geq \max(\frac{2}{\beta T},1)$. This gives $C_G = 2C'_GT$, and $\mathbf{m} = C_G' m$.

At last we have to verify that $\mathbf{m}$ and $G$ are symbols in the right class in the sense of Definition~\ref{def:C-infinity-structure-compactified}. The weight was constructed as a  $\mathscr{C}^\infty$ function on $S^\ast M$, and that is the definition of being in $S^0_{cl}(M)$. For $e^G$ to be elliptic in $S^\mathbf{m}_{\log}(M)$, it suffices then to check that $(1-\chi_G(|\xi|))f$ itself is elliptic in $S^1_{cl}(M)$. By definition, this means that $f/|\xi|$ is a $\mathscr{C}^\infty$ function on $S^\ast M$. That is also a direct consequence of the construction.
\end{proof}

Actually, in our case, we can say something a little better, that will simplify the rest of the proof.
\begin{lemma}\label{lemma:invariance-escape-function}
We can assume that for $y>\mathbf{a}$ with $\mathbf{a}$ large enough, both $G$ and $\mathbf{m}$ are invariant under the local isometries $T_{\tau,\theta_0}$ defined in equation \eqref{eq:def-local-isometries}.
\end{lemma}

\begin{proof}
Recall from the discussion in Section~\ref{sec:admissible_vb} that each cusp $Z_\ell$ can been seen as a subset of the full cusp $Z_{\ell, f}=\Lambda_\ell\backslash\groupG /\groupK$. 
The geodesic flow on the hyperbolic space $\groupG\backslash\groupK$ or rather on its sphere bundle $S(\groupG/\groupK) = \groupG/\groupM$ is known to be uniformly hyperbolic with  analytic stable and unstable bundles $\tilde E^{s/u}$ which are invariant under all isometries of the hyperbolic space $\groupG/ \groupK$ i.e. under the left $\groupG$ action. 
Consequently, these bundles descend to the full cusp $SZ_{\ell, f}$ and can thus be restricted to the cusps. We call them the stable and unstable bundles corresponding to constant curvature and denote them by $E_{u/s}^{c}$. 
By the invariance under isometries of hyperbolic space, the bundles $E_{u/s}^{c}$ are invariant under the local isometries $T_{\tau,\theta_0}$ defined in equation \eqref{eq:def-local-isometries}. 

Let us now explain that $E_{u/s}^c$ and $E_{u/s}$ become $\mathcal O(1/y)$ close high in the cusp. Let us do this for the example of $E_s$: First note that $E_u\oplus E_s=E_u^c\oplus E_s^c$ (this is because the contact form of both flows coincides. Now for trajectories whose past is included in the cusp, $E^u$ and $E^c_u$ have to coincide, so at the bottom of the cusp ($y=a$) directions  that are close to the South pole (i.e incoming trajectories), $E^s$ and $E^c_u$ are transverse (by continuity of the bundles). Now high in the cusp ($y\gg a$) in an arbitrary direction (except in $\mathcal{N}$), its trajectory, when it exits the cusp has to be almost vertical, i.e. in the neighbourhood of the south poles considered above; Now the uniform hyperbolicity of both splittings implies that $E^s$ and $E^c_s$ are $\mathcal{O}(1/y)$-close as $y\to +\infty$.

As a consequence to the fact that the bundles $E^c_{u/s}$ and $E_{u/s}$ become close, when building the functions $m_{0,s}^u$ and $m_{0,u}^s$, we can actually choose them to be invariant by $T_{\tau,\theta_0}$ high in the cusp --- say $y> y_0$.

Since it takes at least a time $\sim \log y$ to go from height $y$ in the cusp to the compact part $N_0$, and since all the constructions above make use only of propagation for a global finite time under the flow, we obtain that for $y> y_0 e^T$, $m$ is also invariant under $T_{\tau,\theta_0}$.

The last thing to check is the invariance of $f$. In the cusp, the vector field $X$ is also invariant under local isometries of the hyperbolic space, so that $f$ also can be chosen to be $T_{\tau,\theta_0}$ invariant for $y> y_0 e^{T'}$. 
\end{proof}

\begin{remark}\label{remark:determination-of-mathbf-a}
We can chose $\mathbf{a}$ so that it coincides with the $\mathbf{a}$ in Definition \ref{def:admissible-bundle}. It will be smaller than the $\mathsf{a}$ of point (7) of Proposition \ref{prop:properties-quantization}.
\end{remark}

\subsection{A first parametrix}
\label{sec:first-parametrix}

Now that we have built our escape function, we focus on building an approximate inverse for $\mathbf{X}-hs$. Recall that we use semiclassical analysis: We had defined $\mathbf X=h\mathcal X$ and we will work with the semi-classical quantization $\Op_{h,L}^w$ acting on sections of $L$ , see Appendix \ref{appendix:microlocal-tools} eq. \eqref{eq:def-quantization-global}. For a simpler notation we simply write $\Op$ in the sequel.
A priori for $\Op(\sigma)$ to make sense, we need that $\sigma$ is valued in $\End(L)$. If $\sigma$ is just a function, we can consider $\Op(\sigma \otimes \mathbb{1})$. This operator will be denoted by abuse of notations just as $\Op(\sigma)$.
\begin{definition}\label{def:anisotropic-space}
Let $\delta>0$ and $G_\delta$ the corresponding escape function given by Lemma \ref{lemma:escape-function}. Let $\escapeparam>0$. We denote by $H^{\escapeparam\mathbf{m}}_\delta$ the set of distributions
\begin{equation}\label{eq:def-H-cal-r}
H^{\escapeparam\mathbf{m}}_\delta  = \Op(e^{- \escapeparam G_\delta}) \cdot L^2(M, L).
\end{equation}
It is endowed with the norm
\[
\| f \|_{H^{\escapeparam\mathbf{m}}_\delta} = \| \Op(e^{-\escapeparam G_\delta})^{-1} f \|_{L^2(M, L)}.
\]
The space actually does not depend on $h$ or on $\delta$, but the norm does. As a convention, we denote $H^{0}_\delta = L^2(M, L)$.
\end{definition}
We will drop the $\delta$ indices in the notations, to lighten a bit the presentation, 
and just write $H^{\escapeparam\mathbf{m}}(=H^{\escapeparam\mathbf{m}}_\delta$). 
Only at the end of section \ref{sec:Black-Box} in the proof of Theorem \ref{thm:full-theorem-resolvent} will 
we let $\delta$ go $0$. 
From the properties of $G$, we directly obtain the following regularity properties, which 
show that $H^{\escapeparam\mathbf{m}}_\delta$ is an \emph{anisotropic space}
\footnote{The spaces that show up here are distributions that are regular in the $E_u^*$ direction and irregular in the $E_s^*$ direction. 
This is no contradiction to precedent works like e.g. \cite{Dyatlov-Guillarmou-16} where the authors continue the resolvent of the operator $-\mathbf X$ and thus obtain the converse regularity properties.}.
\begin{lemma}\label{lem:regularity_of_H_r}
 For any $\escapeparam>0, \delta>0$ we have the continuous inclusions $H^{+ C_G\escapeparam} \subset H^{\escapeparam\mathbf{m}}_\delta \subset H^{-C_G \escapeparam}$. Furthermore near $E_u^*$, $H^{\escapeparam\mathbf{m}}_\delta$ is microlocally equivalent to $H^{ C_G \escapeparam}$ and near $E_s^*$, 
 $H^{\escapeparam\mathbf{m}}_\delta$ is microlocally equivalent to $H^{- C_G \escapeparam}$. In particular, for $A\in S^0(M,L)$
 \begin{equation}\label{eq:regularity_of_H_r}
 \begin{split}
  &\WF_h(A)\in N_s \Rightarrow \|Au\|_{H^{\escapeparam\mathbf{m}}_\delta} 
  \leq C\|A u\|_{H^{-C_G\escapeparam }}\ , \\
    \text{ and }&\WF_h(A)\in N_u \Rightarrow \|Au\|_{H^{C_G\escapeparam }} \leq C\|A u\|_{H^{\escapeparam\mathbf{m}}_\delta}\ .
  \end{split}
 \end{equation}
\end{lemma}

The differential operator $\mathbf X$, which is a priori defined on $C_c^\infty(M,L)$ has a unique closed extension\cite[Lemma A.1]{Faure-Sjostrand-10} to the 
domain $D_{\escapeparam }:=\{u\in H^{\escapeparam\mathbf{m}}: \mathbf Xu\in D_{\escapeparam }\}$. The domain $D_{\escapeparam }$ is naturally a Hilbert space w.r.t. the scalar product $\langle\cdot,\cdot\rangle_{D_{\escapeparam }} := \langle\cdot,\cdot\rangle_{H^{\escapeparam\mathbf{m}}} + \langle \mathbf X\cdot,\mathbf X\cdot\rangle_{H^{\escapeparam\mathbf{m}}}$. 
The action of $\mathbf{X} -h s$ on $H^{\escapeparam\mathbf{m}}$, is equivalent to the action on $H^0=L^2$ of
\[
\begin{split}
\Op(e^{-\escapeparam G})^{-1} (\mathbf{X} -h s) &\Op(e^{-\escapeparam G}) = \\
			&\mathbf{X} - h (\escapeparam \Op(\{ p, G \}) + s) + \mathcal{O}(h^2 \Psi^{-1^+}_{\log}).
\end{split}
\]
Since $X_\Phi$ is the hamiltonian vector field of the Hamiltonian $p$ defined in \eqref{eq:def_p}, we have $\{ p, G\} = X_\Phi G$. We will need the following observation
\begin{lemma}\label{lem:invertible-on-the-right}
There are constants $C,C'>0$ such that for $\Re(s) > C(1 + \escapeparam)$, $\mathbf{X} - h s : D_{\escapeparam }\to H^{\escapeparam\mathbf{m}}$ is invertible. We denote the inverse by $\mathscr{R}(s)$ and its operator norm is bounded: $\|\mathscr{R} (s)\|_{H^{\escapeparam\mathbf{m}}\to H^{\escapeparam\mathbf{m}}}\leq C'h^{-1}$. 
\end{lemma}

\begin{proof}
From the sharp G\r{a}rding inequality Lemma \ref{lemma:sharp-Garding}, we conclude that there are $C,\varepsilon>0$ such that 
$\Re \langle (\mathbf{X}-hs) u ,u\rangle_{H^{\escapeparam\mathbf{m}}} < - \varepsilon h\|u\|_{H^{\escapeparam\mathbf{m}}}^2$ for $\Re(s)> C(1 + \escapeparam)$ and all $u\in C_c^\infty(M)$ ($C$ does not depend on $\escapeparam$).

We deduce that $\| (\mathbf{X}-hs)u \|_{H^{\escapeparam\mathbf{m}}} \geq \varepsilon h \|u\|_{H^{\escapeparam\mathbf{m}}}$. As a consequence, the image of $(\mathbf{X}-hs)$ is closed. We deduce that it is the orthogonal of the kernel of the adjoint. We also get that the kernel of $(\mathbf{X}-hs)$ is empty. Additionally, we observe that the adjoint of $\mathbf{X}-hs$ satisfies the same sharp G\r{a}rding estimate, so that it also is injective, and thus $(\mathbf{X}-hs)$ is surjective. We conclude that it is invertible.
\end{proof}

For each $\delta>0$, we pick $Q \ (=Q_\delta)$, a self-adjoint semi-classical pseudo-differential operator, of the form $\Op(q)$, with $q\in S^0$ equal to $1$ in $\{ |\xi| \leq 2R \delta \}$, everywhere positive, and supported in $\{ |\xi|< 3 R \delta\}$ --- the constant $R$ was given in Lemma \ref{lemma:escape-function}. This is an absorbing potential. Let us denote by
\[
\mathbf{X}_Q(s) = \mathbf{X} - Q - hs.
\]
Then we have the key estimate:
\begin{proposition}\label{prop:Inverse-up-to-smoothing-FS}
Let $\delta>0$, then there is a constant $C_\delta>0$. Assume that $s$ satisfies $\Re(s) > C_\delta - \escapeparam+1$, and $|\Im s | \leq h^{-1/2}$. Then for sufficiently small $h$, the operator $\mathbf{X}_Q(s)$ is invertible on $H^{\escapeparam\mathbf{m}}$. Denoting by $\mathscr{R}_Q(s)$ its inverse, we get $\| \mathscr{R}_Q(s)\| = \mathcal{O}(h^{-1})$.
\end{proposition}

\begin{proof}
We fix a tempered family of functions $u \in C^\infty_c(M,L)$. We consider the regions 
\begin{align*}
\Omega_{\Ell} &:= \Big\{(x,\xi)\ |\ |\xi| < 3 R \delta/2,\ \text{or}\ |p(x,\xi)| > \epsilon \langle\xi\rangle \Big\}.\\
\intertext{and}
\Omega_{\text{G\r{a}rding}} &:= \Big\{ (x,\xi)\ |\ |\xi| > R \delta\ \text{and}\ \xi/|\xi| \notin N_0 \Big\}.
\end{align*}
If $\epsilon>0$ is chosen small enough they overlap, so we can build a partition of unity $1 = A_{\Ell} + A_{\text{G\r{a}rding}}$, with $A_{\Ell}$ (resp. $A_{\text{G\r{a}rding}}$) microsupported in $\Omega_{\Ell}$ (resp. $\Omega_{\text{G\r{a}rding}}$), and both $A$'s in $\Psi^0$.

In the region $\Omega_{\Ell}$, the principal symbol of $\mathbf{X}_Q(s)$ is elliptic, so we deduce\footnote{Note that Proposition~\ref{prop:Elliptic-regularity}(2) is stated in terms of ordinary Sobolev spaces and not in terms of anisotropic Sobolev spaces. The statement on anisotropic spaces can however be deduced by applying Proposition~\ref{prop:Elliptic-regularity}(2) to the conjugated operators $\Op(e^{-\gamma G})^{-1} A_{\textup{ell}}\Op(e^{-\gamma G})$ and 
$\Op(e^{-\gamma G})^{-1} (\mathbf X-\mathbf Q-hs)\Op(e^{-\gamma G})$ respectively. Note therefore that the conjugation does not affect the ellipticity.} from Proposition~\ref{prop:Elliptic-regularity}.
\[
\| A_{\textup{ell}} u \|_{H^{\escapeparam\mathbf{m}}} \leq C \|\mathbf{X}_Q(s) u \|_{H^{\escapeparam\mathbf{m}}} + \mathcal{O}(h^\infty)\|u\|_{H^{\escapeparam\mathbf{m}}}.
\]
Now, we can concentrate on the region of interest $\Omega_{\text{G\r{a}rding}}$. By definition, the action of $\mathbf{X}_Q(s)$ on $H^{\escapeparam\mathbf{m}}$ is conjugated by $\Op(e^{-\escapeparam G})$ to the action on $L^2$ of 
\[
\widetilde{\mathbf{X}_Q(s)}=\mathbf{X} - Q - h( \escapeparam\Op( \{ p + i q, G\} + s)) + \mathcal{O}(h^2 \Psi^{-1^+}_{\log}).
\]
We denote by $\widetilde{A_{\text{G\r{a}rding}}}$ the operator obtained after conjugation by $\Op(e^{-\escapeparam G})$, and $\tilde{u}:= \Op(e^{-\escapeparam G})^{-1}u$ --- $\tilde{u}$ is in $L^2$.
We consider
\[\begin{split}
 -\Re \langle \widetilde{\mathbf{X}_Q(s)} \widetilde{A_{\text{G\r{a}rding}}}\tilde{u}, \widetilde{A_{\text{G\r{a}rding}}} \tilde{u} \rangle_{L^2}= \langle P A_2 \widetilde{A_{\text{G\r{a}rding}}} \tilde{u}, & \widetilde{A_{\text{G\r{a}rding}}}\tilde{u}\rangle_{L^2}\\
 &  +h(\Re(s)+\escapeparam )\|\widetilde{A_{\text{G\r{a}rding}}}\tilde{u}\|_{L^2}^2.
\end{split}
\]
where $A_2$ is a microlocal cutoff in a slightly bigger neighbourhood of $\WF_h(A_{\text{G\r{a}rding}})$ and
\[
 P := -\Re\mathbf{X} + Q + h \escapeparam\Op( \{ p, G \}- 1) + \mathcal{O}(h^2 \Psi^{-1^+}_{\log}).
\] 
(Here, $\Re \mathbf{X}$ is the real part of $\mathbf{X}$ acting on $L^2$, and it is an $\mathcal{O}(h)$ order $0$ operator). By Lemma~\ref{lemma:escape-function}(i) (recall that $\{p, G\}=X_\Phi G$) we conclude that $PA_2\in \Psi^{0+}$ has non-negative principal symbol 
and by the sharp G\r{a}rding inequality \ref{lemma:sharp-Garding}, we deduce that
\[
-\Re \langle \widetilde{\mathbf{X}_Q(s)} \widetilde{A_{\text{G\r{a}rding}}} \tilde{u}, \widetilde{A_{\text{G\r{a}rding}}} \tilde{u} \rangle_{L^2}  \geq h (-C_\delta + \Re(s) + \escapeparam)\|\widetilde{A_{\text{G\r{a}rding}}} \tilde{u}\|_{L^2}^2.
\]
The constant depends on $Q$, which depends itself on $\delta$. Using Cauchy-Schwarz and our assumption $\Re(s) > C_\delta - \escapeparam+1$ we get
\[
\|\widetilde{A_{\text{G\r{a}rding}}} \tilde{u}\|_{L^2} \leq Ch^{-1} \| \widetilde{\mathbf{X}_Q(s)} \widetilde{A_{\text{G\r{a}rding}}} \tilde{u}\|_{L^2},
\]
i.e.,
\[
\| A_{\text{G\r{a}rding}} u\|_{H^{\escapeparam\mathbf{m}}} 
		\leq Ch^{-1} \| \mathbf{X}_Q(s) A_{\text{G\r{a}rding}} u\|_{H^{\escapeparam\mathbf{m}}}.
\]
Gathering our estimates, we find that
\[
\| u\|_{H^{\escapeparam\mathbf{m}}} \leq Ch^{-1}  \| \mathbf{X}_Q(s) A_{\text{G\r{a}rding}} u \|_{H^{\escapeparam\mathbf{m}}} + C \| \mathbf{X}_Q(s) u \|_{H^{\escapeparam\mathbf{m}}}  + \mathcal{O}(h^\infty)\|u\|_{H^{\escapeparam\mathbf{m}}}.
\]
Now let us consider 
\[
\| \mathbf{X}_Q(s) A_{\text{G\r{a}rding}} u \|_{H^{\escapeparam\mathbf{m}}} \leq \| A_{\text{G\r{a}rding}} \mathbf{X}_Q(s) u\|_{H^{\escapeparam\mathbf{m}}} + \|[\mathbf X_Q(s), A_{\text{G\r{a}rding}}]u\|_{H^{\escapeparam\mathbf{m}}} 
\] 
We have $[\mathbf X_Q(s), A_{\text{G\r{a}rding}}]\in h\Psi^{0}$ and  as $\WF_h([\mathbf X_Q(s), A_{\text{G\r{a}rding}}]) \subset \Omega_\text{ell}\cap\Omega_\text{G\r{a}rding}$ we get by elliptic regularity 
$\|[\mathbf X_Q(s), A_{\text{G\r{a}rding}}]u\|_{H^{\escapeparam\mathbf{m}}} \leq C\|\mathbf X_Q(s)u\|_{H^{\escapeparam\mathbf{m}}} +\mathcal O(h^\infty)\|u\|_{H^{\escapeparam\mathbf{m}}}$. By continuity of $A_\text{G\r{a}rding}$ we deduce 
$\| A_{\text{G\r{a}rding}} \mathbf{X}_Q(s) u\|_{H^{\escapeparam\mathbf{m}}} \leq C \| \mathbf{X}_Q(s) u\|_{H^{\escapeparam\mathbf{m}}}$, so alltogether we get $\| \mathbf{X}_Q(s) A_{\text{G\r{a}rding}} u \|_{H^{\escapeparam\mathbf{m}}}\leq C \| \mathbf{X}_Q(s) u \|_{H^{\escapeparam\mathbf{m}}}$ and consequently
\[
\| u \|_{H^{\escapeparam\mathbf{m}}} \leq  \frac{C}{h} \| \mathbf{X}_Q(s) u \|_{H^{\escapeparam\mathbf{m}}} +\mathcal O(h^\infty)\|u\|_{H^{\escapeparam\mathbf{m}}}.
\]
This estimate implies that for sufficiently small $h$, the operator $X_Q(s)$ is injective and has closed range. Performing exactly the same estimates for the adjoint operator, we deduce that $X_Q(s)$ is surjective.
\end{proof}
In the case of compact manifolds, the end of the proof of the equivalent of Theorem \ref{thm:continuation-resolvent} is based on the fact that by writing
\[
(\mathbf{X} - hs)\mathscr{R}_Q(s) = \mathbb{1}  + Q \mathscr{R}_Q,
\]
we have that $\mathbf{X} -hs$ is invertible modulo a smoothing operator, and smoothing operators are compact on compact manifolds, so $\mathbf{X} - h s$ is invertible modulo \emph{compact} operator. Hence it is Fredholm, of index $0$, and its inverse is a meromorphic family of operators in the $s$ parameter.

However, in our case, smoothing operators are \emph{not} compact. We will present a special ingredient in the next section to overcome this problem. Before that, we consider wavefront sets.
\begin{proposition}\label{prop:wavefront-R_Q}
Let $\Omega_+$ be the subset of phase space
\begin{equation}\label{eq:def-Omega}
\Omega_+ := \left\{((x,\xi);\Phi_t(x,\xi))\ |\ p(x,\xi) = 0,\ t\geq 0 \right\} \subset T^\ast M \times T^\ast M.
\end{equation}
Recall that $\Delta(T^\ast M)$ is the diagonal in $T^\ast M$. The wave front set of $\mathscr{R}_Q(s)$ satisfies
\[
\WF_h'(\mathscr{R}_Q(s)) \cap (T^*M\times T^*M) \subset \Delta(T^\ast M) \cup \Omega_+.
\]
\end{proposition}
\begin{proof}
First, by ellipticity in $\{ p(x,\xi) \neq 0\}\cup \{ |\xi|\leq 2R\delta\}$, the wavefront of $\mathscr{R}_Q(s)$ is contained in $\Delta(T^\ast M) \cup \{ p(x,\xi) = p(x',\xi') = 0,\ |\xi|,|\xi'|> 2R\delta \}$.

Next, note that by Lemma~\ref{lemma:property-WF'} we have to prove that for $((x,\xi),(x',\xi'))
\in T^*M \times T^*M$ fulfilling 
\[
p(x,\xi) = p(x',\xi')=0,\quad |\xi|,|\xi'|>2R\delta,\quad \text{and} \quad ((x,\xi),(x',\xi'))\notin \Omega_+,
\]
there are $A, A'\in S^0$, elliptic in $(x,\xi)$ (resp. $(x',\xi')$) such that $A\mathscr{R}_Q(s) A'$ is $\mathcal O_{H^{-\infty}\to H^{\infty}}(h^ \infty)$. 

In order to achieve this, let $((x,\xi),(x',\xi'))
\in T^*M \times T^*M$ be such a point. Recall that as $t\to+\infty$,  $|\Phi_t(x',\xi')|$ either goes to $0$ or to $\infty$. Hence, we can chose two relatively compact open sets $U, U' \subset T^*M$ such that $\Phi_t(U) \cap U' = \emptyset$ for all $t\geq 0$, and $(x,\xi)\in U$, $(x',\xi')\in U'$. Fix $A,A'\in \Psi^0$ microsupported in respectively $U$ and $U'$. 

Let us prove that $A \mathscr{R}_Q(s) A' = \mathcal{O}_{H^{-\infty}\to H^{\infty}}(h^\infty)$. Let $u$ be a tempered family of distributions. Let $T>0$, and $B, B_1$ elliptic on respectively $\Phi_{T}(U)$ and $\cup_{0\leq t \leq T} \Phi_{t}(U)$. Observe that $A\mathscr{R}_Q(s) A' u$ is in all Sobolev spaces because $A,A'$ are compactly microsupported. Then we get by Propagation of Singularities \ref{lemma:Propagation-of-singularities} that for $k\in\R$
\[
\| A \mathscr{R}_Q(s) A' u \|_{H^k} \leq C \| B  \mathscr{R}_Q(s) A' u \|_{H^k} + \frac{C}{h} \| B_1 A' u \|_{H^k} + \mathcal{O}_{k,u}(h^\infty).
\]
By the assumption on the microsupport of $A$ and $A'$, by taking the microsupport of $B_1$ small enough, we can ensure that $B_1 A' u = \mathcal{O}(h^\infty)$, hence
\begin{equation}\label{eq:wf_resolvent_prop_est}
\| A \mathscr{R}_Q(s) A' u \|_{H^k} \leq C \| B  \mathscr{R}_Q(s) A' u \|_{H^k} + \mathcal{O}_{k,u}(h^\infty).
\end{equation}
Now, we just have to consider what happens when the time $T$ becomes larger. For $(x,\xi)\in \{p=0\}\subset   T^*M$ there are only two possibilities: Either there is $T>0$, such that $\Phi_T(x,\xi) \subset \textup{ell}_1(Q)=\{|\xi|\leq R\delta\}$ or $\Phi_t(x,\xi)$ converges to $E^\ast_u\cap\partial \overline{T^\ast} M$ (see Definition~\ref{def:C-infinity-structure-compactified} for the radial compactification). 

In the first case take $U$ sufficiently small such that $\Phi_T(U)\subset \textup{ell}_1(Q)$. Thus we can assume that $B$ in the propagation estimate \eqref{eq:wf_resolvent_prop_est} is microsupported in $\textup{ell}_1(Q)$. Taking $B'\in \Psi^0$ elliptic on the microsupport of $B$, the elliptic estimate (Proposition~\ref{prop:Elliptic-regularity}) gives,
\[
\| B \mathscr{R}_Q(s) A' u \|_{H^k} \leq C\| B' A' u \|_{H^k} + \mathcal{O}_{m,u}(h^\infty).
\]
Since we can choose $B'$ such that $\WF(B') \cap \WF(A') = \emptyset$, the RHS is $\mathcal{O}(h^\infty)$.

Now, we turn to second case which we will treat using the high regularity radial estimate
(Proposition~\ref{prop:Sink-estimate}): Note that $E^*_u\cap\partial\overline{T^*}M$ is a sink in the sense of Definition~\ref{def:sources-sinks}. Next let us choose $C\in \Psi^0$ such that 
$E^*_u\cap\partial\overline{T^*}M \subset \textup{ell}_1(C)$ and such that $\WF_h(C)\cap \WF_h(A') = \emptyset$. Then Proposition~\ref{prop:Sink-estimate} provides us with an order $0$ operator $C_1$ which is elliptic in a neighbourhood of $E^*_u\cap\partial\overline{T^*}M$. Furthermore we can assume $\WF_h(C_1) \subset N_u$. 

Since $C_1 \mathscr{R}_Q A' u\in H^{\escapeparam\mathbf{m}}$ and is microsupported in $N_u$, by
Lemma~\ref{lem:regularity_of_H_r}, we know that $C_1 \mathscr{R}_Q(s) A'u \in H^{\escapeparam  C_G}$ and
taking $\escapeparam C_G>k_0$ we have the necessary regularity for the
sink estimate. We get for any $k>k_0$
\begin{equation}\label{eq:wf_resolvent_sink_estimate}
\| C_1 \mathscr{R}_Q(s) A' u \|_{H^k} \leq \frac{C}{h}\| C A' u \|_{H^k} + \mathcal{O}(h^\infty) = \mathcal O(h^\infty) .
\end{equation}
Finally for $U$ sufficiently small and by propagation of singularity for a long enough but finite time $T$ we can assume that $\Phi_T(U) \subset \textup{ell}_1(C)$. Combining \eqref{eq:wf_resolvent_prop_est} and \eqref{eq:wf_resolvent_sink_estimate} we obtain as desired
$\|A\mathscr{R}_Q(s)A'u\|_{H^k} = \mathcal O(h^\infty)$.
\end{proof}

We have a final remark for this section
\begin{defprop}\label{prop:invertibility-higher-regularity}
If $\escapeparam\geq 0$ and $N\in \R$, we say that $\mathbf{k}=\escapeparam \mathbf{m} + N$ is a \emph{weight}. Such a weight is said to be large if $\escapeparam$ is large, and $N/\escapeparam$ is small. We define 
\[
H^{\mathbf{k}}_{(\delta)}(M,L) := \Op(e^{-\escapeparam G_\delta}) H^{N}(M,L).
\]
We get that the conclusion of Proposition \ref{prop:Inverse-up-to-smoothing-FS} holds on the space $H^{\mathbf{k}}$ when $|\Im s|<h^{-1/2}$, $\Re s \geq C_\delta - \escapeparam + CN + 1$ for some constant $C$ independent of $\escapeparam,N$, and for $h>0$ small enough. 
\end{defprop}

The proof is completely analogous to the proof of proposition \ref{prop:Inverse-up-to-smoothing-FS}.

\section{Continuation of the resolvent for translation invariant operators}
\label{sec:continuation-indicial-resolvent}

In this section, we will be considering a vector bundle over some compact Riemannian manifold $\FibreL \to \FibreM$, endowed with a bundle metric and a compatible connection. We will always see the space $\R \times \FibreL$ as a fiber bundle over $(\R)_r \times (F)_{\PointM}$, endowed with the product structure. We will also use the natural measure $dr d\PointM$, and $L^2(\R\times\FibreL)$ will be understood as the space of square-integrable sections with respect to this measure.

Let us first see how bundles of this type can be naturally obtained from admissible vector bundles in the sense of Definition~\ref{def:admissible_vector_bundle}.
\begin{example}\label{exmpl:admissible_bundle_reduction}
 Let $L\to M=S^*N$ be an admissible vector bundle and fix a cusp
 $Z_\ell$. Then over this cusp the bundle takes the form
 $L= \Lambda_\ell\backslash \groupG \times_{\tau_\ell} V_\ell$. Using the Iwasawa decomposition $\groupG=\groupN\groupA\groupK$  and identifying $\groupA\cong(\R,+), \groupN\cong(\R^d,+)$ we obtain $L=(\R^d/\Lambda_\ell) \times \R \times (\groupK\times_{\tau_\ell}V_\ell)$. In Section~\ref{sec:Black-Box}
 we will study sections of these bundles that are independent 
 on the variable $\theta\in (\R^d/\Lambda_\ell)$ and these sections are naturally identified with sections of $\R\times(\groupK\times_{\tau_\ell}V_\ell)$. This shows that 
 studying $\theta$-independent sections of admissible vector bundles $L_{|S^*Z_\ell}$ leads to the study of sections of $\R\times\FibreL_\ell \to \R\times \FibreM$ with $\FibreL_\ell =\groupK\times_{\tau_\ell}V_\ell\to \groupK/\groupM\cong\mathbb S^d=\FibreM$. 
\end{example}
\begin{remark}\label{rem:general_bundle}
For the proof of Theorem \ref{thm:full-theorem-resolvent} on vector bundles one could restrict the discussion of the whole section to the special case in the example above. As all arguments, however, hold without any further complications in the general case of vector bundles $\FibreL\to\FibreM$ over general compact manifolds $\FibreM$ we announce and prove all results in this section in this setting. Additionally, we expect that this wider class is likely to show up when studying uniformly hyperbolic flows on fibred cusps.  
\end{remark}
\subsection{b-Operators}\label{sec:b-operators}
We will consider a particular class of operators on $\R\times \FibreL\to\R\times\FibreM$: Recall that by the Schwartz kernel theorem any continuous linear Operator $A:C_c^\infty(\R\times \FibreL) \to \mathcal D'(\R\times\FibreL)$ is represented by its kernel $K_A\in\mathcal D'(\R\times\R\times\FibreL\boxtimes\FibreL)$. We call such an operator $A$ a convolution operator if there is $\tilde K_A\in\mathcal D'(\R\times\FibreL\boxtimes\FibreL)$ such that $K_A=p^*\tilde K_A$ where $p:\R\times\R\times\FibreL\boxtimes \FibreL \ni(r,r',l\boxtimes l')\mapsto (r-r',l\boxtimes l')\in \R\times\FibreL\boxtimes\FibreL$.
\begin{definition}\label{def:free-b-operators}
The set of semiclassical pseudo-differential operators acting on sections of $\R\times \FibreL$ that are convolution operators in the $r$ variable will be denoted by $\Psi_b(\R\times \FibreL)$. It is the set of \emph{b-operators}.

Such operators that additionally are supported in $\{|r-r'|\leq \log C\}$ will be denoted $\Psi_{b,C}(\R\times \FibreL)$. We say that they are b-operators with \emph{precision} $C$. When $C=1$, the kernels are supported on $\{r=r'\}$.
\end{definition}
\begin{remark}
Our notion of b-operators is, as its name suggests, strongly inspired by Melrose's b-calculus (see e.g. \cite{Melrose-APS-93}). However in this article we use a much more restrictive class of operators. Let us shortly explain the relation of our b-operators to the usual class of b-differential operators in the sense of Melrose. Let $[0,\infty[_x \times \R_\zeta$ be the simplest model of a manifold with boundary. Then the b-differential operators are those in the algebra of operators generated by b-vectorfields that take the form $a(x,\zeta)x\partial_x + b(x,\zeta)\partial_\zeta$ with $a,b\in C^\infty([0,\infty[_x\times \R_\zeta)$. Using a Taylor expansion, the leading order near the boundary of these operators takes the form $a_0(\zeta)x\partial_x+b_0(\zeta)\partial_\zeta$. After a variable transformation $r=\log(x)$ these are in the form $a_0(\zeta)\partial_r + b_0(\zeta)\partial_\zeta$. Such operators are then translation invariant in the $r$ variable, i.e. are convolution operators. Their kernels take the form
\[
a_0(\zeta)\delta(r-r') + b_0(\zeta)\delta(\zeta-\zeta').
\]
In some sense our class of b-operators contains just those which are equal to their leading part in the asymptotic expansion near the boundary of the usual class of b-(pseudo)-differential operators. For our purpose this is sufficient and the restriction to this class allows us to concentrate on the difficulties that arise from the fact that we have to construct a parametrix for an operator that is not elliptic (even in a b-calculus sense).
 
\end{remark}

\begin{example}\label{exmpl:b-op-of_geodesic_flow}
The generator of the geodesic flow acting on functions supported in a cusp and not depending on $\theta$ is a differential operator acting on the trivial bundle, i.e. on $L^2(\R\times\Ss^d, e^{-rd}dr d\zeta)$ given by (cf. equation~\eqref{eq:def-X_k})
\[
 X_b^0 = \cos\varphi\partial_r +\sin\varphi\partial_\varphi.
\]
In order to make it a b-operator acting on $L^2(\R\times\Ss^d, dr d\zeta)$ we conjugate it with $e^{-rd/2}$ and get:
\begin{equation}\label{eq:def-X_b}
X_b= \cos\varphi\partial_r + \frac{d}{2}\cos\varphi+\sin\varphi\partial_\varphi.
\end{equation}
In order to work in the semiclassical calculus we write $\mathbf X_b:=hX_b$.
\end{example}
The aim of Section~\ref{sec:continuation-indicial-resolvent} is to show that the resolvent of $\mathbf{X}_b$ can be continued meromorphically from $\Re(s) > 0$ to $\C$. In fact, for the reasons discussed in Remark~\ref{rem:general_bundle}, we will treat a more general class of operators $\mathbf X_b\in \Psi_{b,0}(\R\times\FibreL)$  whose precise assumptions will be formulated in Section~\ref{sec:translation_invariant_approximate_inverse}

Next, let us introduce symbols and quantizations that lead to b-operators
\begin{definition}
Denote by $g$ the metric on $\FibreM$ and by $T(T^*\FibreM)=H\oplus V$ the splitting into vertical and horizontal directions w.r.t. the Levi-Civita connection. We endow $T^\ast (\R\times\FibreM)$ with the metric described in Definition~\ref{def:Kohn-Nirenberg-metric}. Consider its restriction $\overline{g}_b$ to $(T^\ast_0 \R)_\lambda \times (T^\ast \FibreM)_{(\PointM,\eta)}$. It can be expressed as
\[
\begin{split}
\overline{g}_{b,(\PointM;\eta,\lambda)}&(X^v + Y^h + \mu\partial_\lambda, W^v + Z^h + \mu'\partial_\lambda) \\
	&= g_{\PointM}(Y, Z) + \frac{1}{1 + g_{\PointM}(\eta,\eta) + \lambda^2}\left[ g_{\PointM}(X, W) + \mu \mu' \right].
\end{split}
\]
By Lemma~\ref{lemma:bounded-curvature-Kohn-Nirenberg}, $\overline{g}_b$ has bounded geometry.
\end{definition}

\begin{defprop}\label{defprop:free-b-symbols}
We denote by $S^0_b(\R\times \FibreL)$ the set of $\mathscr{C}^\infty$ sections $T^*(\R\times \FibreM) \to \End(\FibreL)$ with uniformly bounded derivatives with respect fo $\overline g_b$ which additionally are independent of the $r$ variable. 
They are the translation invariant elements of $S^0(\R\times \FibreM, \R\times \FibreL)$ from Definition~\ref{def:symbol-classes}. Similarly, we define $S^0_{b,\epsilon}$, $S^0_{b,\epsilon,\xi}$ and $S^0_{b,\log}$. We call them  order $0$ \emph{b-symbols}. Given $m_b\in S^0_b(\R\times \FibreM)$, we can also define $S^{m_b}_{b,\log}(\R \times \FibreL)$ as $\langle \xi \rangle^{m_b} S^0_{b,\log}(\R \times\FibreL)$. It is the set of anisotropic symbols of order $m_b$.

These symbol classes are stable by all the usual symbolic manipulations (because $\overline{g}_b$ has bounded geometry).
\end{defprop}

Consider a semi-classical Weyl quantization $\Op^w_h$ for sections of $\FibreL \to \FibreM$ (see e.g. \cite[Theorem 14.1]{Zworski-book} or Appendix~\ref{app:quantization_on_cups}): Given a finite open cover $U_k$ of $\FibreM$ and a trivialisation $t_k:\text{pr}^{-1}_{\FibreL\to\FibreM}(U) \to V\times \R^{\dim(L_x)}$ as well as a quadratic partition of unity $\sum_k\chi_k^2=1$, $\chi_k\in C_c^\infty(U_k, \R_{\geq0})$ such a quantization can be written for $\sigma\in S^m(\FibreL)$ by
\begin{equation}\label{eq:Op^w_hL}
 \Op_{h, \FibreL}^w(\sigma) := \sum_k \chi_k t_k^*\Op_{h, \R^{\dim\FibreM}}^w((t_k^{-1})^*\sigma) (t_k^{-1})^*\chi_k
\end{equation}
where $\Op_{h, \R^{\dim\FibreM}}^w$ is the usual Weyl quantization on $\R^{\dim\FibreM}$.

Now we can use $\Op_{h, \FibreL}^w$ to define a quantization of b-symbols $\sigma_b\in S^m_b(\R\times\FibreL)$ on $\R\times\FibreL$ by 
\begin{equation}\label{eq:def-free-Op}
(\Op^b( \sigma_b)f) (r,\PointM) := \frac{1}{2\pi h}\int e^{\frac{i}{h}\lambda(r-r')} \Big[\Op^w_{h,\FibreL}(\sigma_b(\cdot, \cdot;\lambda))f(r',\cdot)\Big](\PointM) d\lambda dr'.
\end{equation}
which yields an element of $\Psi_b(\R\times L)$. If additionally, we choose a smooth cutoff $\chi_C$ supported in $]-\log C,\log C[$, equal to $1$ in $]-\log C^{1/2},\log C^{1/2}[$, we can multiply the kernel of $\Op^b(\sigma_b)$ by $\chi_C(r-r')$, and obtain an operator $\Op^b(\sigma)_C$ in $\Psi_{b,C}(\R\times\FibreL)$.

It should be noted that plugging in \eqref{eq:Op^w_hL} into \eqref{eq:def-free-Op} and writing 
$\tilde t_k: \R\times \text{pr}^{-1}_{\FibreL\to\FibreM}(U)\ni (r,l) \mapsto (r,t_k(l))\in \R\times V\times \R^{\dim(L_x)}$ we get
\[
 (\Op^b( \sigma_b)f) (r,\PointM) := \sum_k \chi_k \tilde t_k^*\Op_{h, \R^{\dim\FibreM+1}}^w((\tilde t_k^{-1})^*\sigma_b) (\tilde t_k^{-1})^*\chi_k.
\]
From this expression we see that all usual properties of quantizations, such as composition formulas, $L^2$ estimates, sharp G\r{a}rding inequalities etc that hold for the quantization on $\R^{\dim\FibreM+1}$ (see e.g. \cite[Appendix E]{Dyatlov-Zworski-book}) directly transfer to $\Op^b(\sigma_b)$. The same holds for $\Op^b(\sigma_b)_C$ because the cutoff away from the diagonal modifies the operator only be an element of $h^\infty\Psi^{-\infty}$.

\begin{remark}\label{rem:Op-to-Op-b}
We will see in Proposition~\ref{prop:S_b_gives_Psi_b} that there will be a method to construct b-symbols from any symbol $\sigma\in S(L \to S^\ast Z)$ which is invariant by the local isometries of the cusp $T_{\tau,\theta}$. (Recall that e.g. the escape function had this property).
\end{remark}

\subsection{Approximate inverse}\label{sec:translation_invariant_approximate_inverse}

\begin{definition}\label{def:free-admissible-triple}
Let $\mathbf{X}_b\in \Psi^1_{b,1}(\R\times\FibreL)$, $G_b\in S^{0+}_b(\R\times\FibreM)$ and $Q_b\in \Psi^{-\infty}_{b,C}(\R\times\FibreM)$. We will say that this triple is admissible if
\begin{itemize}
	\item $-i\mathbf{X}_b$ and $Q_b$ have scalar, real principal symbols.
	\item $e^{\escapeparam G}$ is elliptic in $S^{m_b}_{b,\log}$ for some $m_b\in S^0_b$.
	\item $\mathbf{X}_b = h X_b$ where $X_b$ is a differential operator independent of $h$.
	\item Let $i p_b$ be the principal symbol of $\mathbf{X}_b$. There is a $\delta'>0$ such that
\[
|p_b| \leq \delta'|\xi|\text{ and } |\xi|> \delta' \Longrightarrow \{p_b, G_b \} > 1.
\]
	\item For the same $\delta'>0$, $Q_b$ is elliptic on $|\xi|<2\delta'$ and microsupported in $|\xi|< 3\delta'$. 
\end{itemize}
\end{definition}

\begin{example}
We will see in Section~\ref{sec:Black-Box} that $\mathbf{X}$, $G$ and $Q$ defined in the previous section \ref{sec:first-parametrix} will give rise to a an admissible triple after restricting to $\theta$-invariant sections. The constant $\delta'$ is just $R \delta$, when $\delta>0$ is small enough.
\end{example}

\begin{definition}\label{def:anisotropic_b_spaces}
As in Def-Proposition \ref{prop:invertibility-higher-regularity}, we say that $k_b\in S^0_b$ is a \emph{weight} if it is of the form $\escapeparam m_b + N$. When we say that a weight is \emph{large}, it means that $\escapeparam >0$ is large, and that $N/\escapeparam $ is arbitrarily small.

Given a weight $k_b$ and $\rho\in \R$, we will work with the space of $\FibreL$-valued distributions on $\R\times \FibreM$
\begin{equation}\label{eq:def-H-cal-b}
\mathcal{H}^{k_b}_{b,\rho} :=e^{\rho r} \Op^b( e^{-\escapeparam G_b}\langle \xi \rangle^{- N})_C L^2(\R\times L),
\end{equation}
endowed with the corresponding norm $\|u\|_{\mathcal{H}^{k_b}_{b,\rho}} := \|\Op^b( e^{-\escapeparam G_b}\langle \xi \rangle^{- N})_C^{-1}e^{-\rho r}u\|_{L^2}^2$ (Note that for $h>0$ small enough, $\Op^b( e^{-\escapeparam G_b}\langle \xi \rangle^{- N})_C^{-1}$ exists because of the ellipticity of $e^{-\escapeparam G_b}\langle \xi \rangle^{- N}$ in $S_{b,\log}^{-\escapeparam m_b-N}(\R\times\FibreM)$). 
\end{definition}

The main result in this subsection is the following:
\begin{lemma}\label{lemma:Inversion-up-to-smoothing-indicial-operator}
Assume that $(\mathbf{X}_b, G_b, Q_b)$ is an admissible triple. Then there is a constant $C>0$ such that for $\Re(s) > 1 + C_\delta +C(|\rho | + |N|)-\escapeparam $ and $|\Im s |\leq h^{-1/2}$, and for small enough $h>0$, $\mathbf{X}_b - Q_b - hs$ is invertible on 
$\mathcal{H}_{b,\rho}^{\escapeparam m_b + N}$. 
\end{lemma}

\begin{proof}
We can apply the same arguments as in the proof of Proposition \ref{prop:Inverse-up-to-smoothing-FS}. Note that in the positive commutator part, which uses the sharp G\r{a}rding inequality, it is important that the real part of $\mathbf{X}_b-Q_b -hs$ on $\mathcal{H}_{b,\rho}^{\escapeparam m_b + N}$ is unitarily equivalent to the action on $L^2$ of 
\begin{equation}\label{eq:conjugation-action-real-part-Xb-Qb}
\begin{split}
\Re \mathbf{X}_b - Q_b& - h\Big( \Re(s) + \escapeparam \Op^b(\{ p_b, G_b\})_C \\
&+ N\Op^b(\{p_b,\log\langle\xi\rangle\})_C -  e^{-\rho r}[i\Im X_b, e^{\rho r}]\Big) + \mathcal{O}_{L^2\to L^2}(h^2).
\end{split}
\end{equation}
Here it is crucial that the absolute value of the second line in \eqref{eq:conjugation-action-real-part-Xb-Qb} is bounded by $C(|\rho|+|N|)$ --- it would not be the case a priori replacing $\log \langle\xi\rangle$ by $m'_b\log \langle\xi\rangle$ where $m'_b\in S^0_b$.
\end{proof}

\subsection{The indicial family}\label{sec:indicial_family}

For the following constructions it is useful to bear in mind the elementary method of invertion of convolution operator on $\R$. Consider some $f\in \mathcal{D}'(\R)$ compactly supported and the operator $T_f :g \mapsto f\ast g$. Obviously, the Fourier transform of $T_f g$ is just $\hat{f}\hat{g}$. To invert $T_f$, is suffices then to invert $\hat{g}\mapsto\hat{f}\hat{g}$. Our aim is to invert the b-operators introduced in section  section~\ref{sec:b-operators} which motivates us to introduce an analogon to the above appearing Fourier transform:

Let $A\in \Psi_{b,C}(\R\times\FibreL)$ and $f\in C^\infty(\FibreL)$. For $\lambda\in \C$, we consider $e^{\lambda r/h}f(\zeta)\in C^\infty(\R\times\FibreL)$.
By the support properties of the kernel, $A$ is a properly supported pseudodifferential operator and thus defines a continuous operator on $C^\infty(\R\times \FibreL)$. Moreover, by the fact that $A$ is a convolution operator,
$(r,\zeta)\mapsto e^{-\lambda r/h}(Ae^{\lambda \bullet/h}f(\bullet))(r,\zeta)$ is independent of $r$, thus $\zeta\mapsto e^{-\lambda r_0/h}(Ae^{\lambda \bullet/h}f(\bullet))(r_0,\zeta)$ is independent of $r_0$ and is a well defined smooth section of $\FibreL$.
\begin{definition}\label{def:indicial_family}
Given $A\in \Psi_{b,C}(\R\times\FibreL)$ and $\lambda\in\C$, we define the \emph{indicial family associated to A} as the family of operators
$I(A,\lambda):C^\infty(\FibreL)\to C^\infty(\FibreL)$ given by 
\[
(I(A,\lambda)f)(\PointM) := e^{-\lambda r_0/h}( Ae^{\lambda\bullet/h}f(\bullet))(r_0,\PointM)
\]
\end{definition}
Note that given a second operator $B\in \Psi_{b,C}(\R\times\FibreL)$, it follows from the definition that $I(AB,\lambda) = I(A,\lambda)I(B,\lambda)$.
\begin{example}
If $\mathbf{X}_b$ is obtained from the geodesic flow on a cusp, i.e. is the operator in equation \eqref{eq:def-X_b}, the corresponding indicial family is
\[
I(\mathbf X_b,\lambda)= \lambda \cos\varphi + h\left( \frac{d}{2}\cos\varphi + \sin\varphi\partial_\varphi\right).
\]
\end{example}
Note that an equivalent description of the indicial family is the following: Fix $\chi\in C_c^\infty(\R)$ with $\int\chi(r) dr=1$, then the indicial family is the family of operators $I(A,\lambda): C^\infty(\FibreL)\to\mathcal D'(\FibreL)$ such that for any $f_1,f_2\in C^\infty(\FibreL)$:
\begin{equation}\label{eq:indicial_alternative_def}
 \big\langle f_2,I(A,\lambda)f_1\big\rangle_{C^\infty(\FibreL),\mathcal D'(\FibreL)} = \int \chi(r)e^{-\lambda r/h}\big\langle f_2(\zeta), \big(Ae^{\lambda\bullet/h}f_1(\bullet)\big)(r,\zeta)\big\rangle_{\FibreL_\zeta}dr.
\end{equation}
This expression is helpful for two purposes: First, by taking complex derivatives of the right hand side with respect to $\lambda$ we conclude:
\begin{lemma}
 For $A\in \Psi_{b,C}(\R\times\FibreL)$, $I(A,\lambda)$ is holomorphic in $\lambda$ as a family of operators $C^\infty(\FibreL)\to\mathcal D'(\FibreL)$.
\end{lemma}
Secondly it allows to extend the definition of the indicial families to convolution operators on $\R\times\FibreL$ that fail to be in $\Psi_{b,C}$. Note that we will work with nonelliptic problems, thus the appearing inverse operators (like for example $(\mathbf X_b-Q_b-hs)^{-1}$ from Lemma~\ref{lemma:Inversion-up-to-smoothing-indicial-operator}) will not be pseudodifferential operators, so it will be crucial to have the following extended definition. 

\begin{lemma}\label{def:extension-I(A)}
Let $A$ be a convolution operator $C^\infty_c(\R\times \FibreL)\to \mathcal{D}'(\R\times\FibreL)$, such that for some $N_1,N_2\in\R$ and $\rho_0<\rho_1$
\begin{equation}\label{eq:extension-I(A)-bound}
\| A\|_{\mathcal{H}_{b,\rho_0}^{N_1} \to \mathcal{H}_{b,\rho_0}^{N_2}} < \infty \text{ and }\| A\|_{\mathcal{H}_{b,\rho_1}^{N_1} \to \mathcal{H}_{b,\rho_1}^{N_2}} < \infty.
\end{equation}
Then \eqref{eq:indicial_alternative_def} defines for $\Re(\lambda)\in]\rho_0,\rho_1[$ a holomorphic family of linear operators $C^\infty(\FibreL)\to \mathcal D'(\FibreL)$. Furthermore, for $\rho\in]\rho_0,\rho_1[$ we have $\|I(A,\rho+iw)\|_{H^{N_1}(\FibreL)\to H^{N_2}(\FibreL)}\leq C\langle w\rangle^{|N_1|+|N_2|}$.
Given a second convolution operator $B$ fulfilling $\|B\|_{\mathcal{H}_{b,\rho_{0/1}}^{N_2} \to \mathcal{H}_{b,\rho_{0/1}}^{N_3}}<\infty$ we have for any $\Re(\lambda)\in]\rho_0,\rho_1[$
\begin{equation}\label{eq:indicial_algebra_hom}
I(AB,\lambda) = I(A,\lambda)I(B,\lambda).
\end{equation}

\end{lemma}
\begin{proof}
We want to show that for $f_1\in H^{N_1}(\FibreL)$ and $\Re(\lambda)\in]\rho_0,\rho_1[$, $Ae^{\lambda \bullet/h}f_1(\bullet)$ is well defined:
let us choose a partition of unity $\Psi_1,\Psi_2\in C^\infty(\R)$, 
$\supp(\Psi_1)\in]-\infty, 1]$, $\supp(\Psi_2)\in ]-1,\infty]$, $\Psi_1+\Psi_2=1$. Choose 
$\rho\in ]\rho_0,\rho_1[$ and set $\lambda=\rho+iw$. Then the maps
\[
\begin{cases} H^{N_1}(\FibreL)\ni f_1(\zeta) \mapsto \Psi_1(r)e^{\lambda r/h}f_1(\zeta) \in \mathcal{H}_{b,\rho_0}^{N_1}, \\
H^{N_1}(\FibreL)\ni f_1(\zeta) \mapsto \Psi_2(r)e^{\lambda r/h}f_1(\zeta) \in \mathcal{H}_{b,\rho_1}^{N_1},\end{cases}
\]
are continuous with operator norm bounded by $C\langle w\rangle^{|N_1|}$. By the compact support of the cutoff function $\chi$ appearing in \eqref{eq:indicial_alternative_def}, for any $\rho'\in\R$, the linear operator
\[
 H^{-N_2}(\FibreL)\ni f_2(\zeta) \mapsto \chi(r)e^{-\lambda r/h}f_2(\zeta) \in \mathcal{H}_{b,\rho'}^{-N_2},
\]
is well defined and bounded by $C\langle w\rangle^{|N_2|}$. 
Using the continuity of $A: \mathcal H_{b,\rho_0}^{N_1}\to\mathcal H^{N_2}_{b,\rho_0}$ and $A: \mathcal H_{b,\rho_1}^{N_1}\to\mathcal H^{N_2}_{b,\rho_1}$ respectively, yields that 
\begin{align*}
 \big\langle f_2,I(A,\lambda)f_1\big\rangle_{C^\infty(\FibreL),\mathcal D'(\FibreL)} = &\Big\langle \chi(r)e^{-\lambda r/h} f_2(\zeta), \big(A\Psi_1(\bullet) e^{\lambda\bullet/h}f_1(\bullet)\big)\Big\rangle_{C_c^\infty(\R\times\FibreL),\mathcal D'(\R\times\FibreL)},\\
 &+\Big\langle \chi(r)e^{-\lambda r/h} f_2(\zeta), \big(A\Psi_2(\bullet) e^{\lambda\bullet/h}f_1(\bullet)\big)\Big\rangle_{C_c^\infty(\R\times\FibreL),\mathcal D'(\R\times\FibreL)},\\
 \leq& C\langle w\rangle^{|N_1|+|N_2|}\|f_2\|_{H^{-N_2}(\FibreL)}\|f_1\|_{H^{N_1}(\FibreL)}.
\end{align*}
This shows the well definedness of $I(A,\lambda)$ for $\Re(\lambda)\in]\rho_0,\rho_1[$ and the bounds on the operator norm. The holomorphicity is again deduced from the fact that the right hand side of \eqref{eq:indicial_alternative_def} is holomorphic in $\lambda$. The above calculations also show that
\[
 A(e^{\lambda\bullet/h}f_1(\bullet)) = e^{\lambda\bullet/h}I(A,\lambda)f_1,
\]
and from this equation the composition property \eqref{eq:indicial_algebra_hom} follows directly. 
\end{proof}

\begin{lemma}\label{lemma:Indicial-symbol}
Let $\sigma_b\in S_b(\R\times\FibreL)$. Then there is a holomorphic family $\lambda\to \sigma_{b,\lambda}\in S_b(\FibreL)$
such that
\[
I(\Op^b(\sigma_b)_C, \lambda) = \Op^w_{h\FibreL}(\sigma_{b,\lambda}).
\]
It is given by
\begin{equation}
 \label{eq:expression-sigma-b-lambda}
\sigma_{b,\lambda}(\zeta,\eta) := \frac{1}{2\pi h}\int_\R\sigma_b(\zeta,\eta,\lambda')\hat\chi_C\left(\frac{-i\lambda-\lambda'}{h}\right) d\lambda'.
\end{equation}

Furthermore if $\sigma_b\in S_{b,\log}^{m_b}(\R\times \FibreL)$ then $\sigma_{b,\lambda}\in S^{m_{b,0}}_{\log}(\FibreL)$ and the leading asymptotics in the high frequency limit is given by $\sigma_b(\cdot,\cdot,0)$, i.e. 
\begin{equation}\label{eq:leading_asym_sigma_lambda}
 \sigma_{b,\lambda} - \sigma_b(\cdot,\cdot,0) \in (1+\log\langle \xi\rangle )S^{m_{b,0} - 1}_{\log}(\FibreL).
\end{equation}
In particular the leading asymptotics of $\sigma_{b,\lambda}$ in the high frequency regime is independent of $\lambda$.
\end{lemma}

\begin{proof}
We use Definition~\ref{def:indicial_family} and choose $r_0=0$ for simplicity. Then using \eqref{eq:def-free-Op} we get 
\begin{align*}
 I(\Op^b( \sigma_b)_C,\lambda)f  &= \frac{1}{2\pi h}\int e^{\frac{i}{h}\lambda'(-r')}\chi_C(-r')e^{\lambda r'/h} \Big[\Op^w_{h,\FibreL}(\sigma_b(\cdot, \cdot;\lambda'))f\Big] d\lambda' dr'.\\
 &=\frac{1}{2\pi h}\int \hat \chi_C\left(\frac{-i\lambda'-\lambda}{h}\right)\Big[\Op^w_{h,\FibreL}(\sigma_b(\cdot, \cdot;\lambda'))f\Big]d\lambda'
\end{align*}
Now the fact that the $d\lambda'$ integral can be interchanged with $\Op^w_{h,\FibreL}$ is justified by the fact that $\Op^w_{h,\FibreL}$ is defined in a finite number of charts (see\eqref{eq:Op^w_hL}). 

It remains to study the leading asymptotics of $\sigma_{b,\lambda}$. Let $\sigma_b\in S^{m_b}_{\log}(\R\times\FibreL)$ and choose $N>0$ such that $-N\leq m_b\leq N$. By the symbol estimates we deduce
\begin{equation}\label{eq:sigma_b_la_symbol_estimate}
 |\partial_\zeta^\alpha\partial_\eta^\beta\partial_\lambda^k \sigma_b(\zeta,\eta,\lambda)| \leq C(1+\log\langle\xi\rangle)^{|\alpha|+|\beta|+k}(\langle\eta\rangle\langle\lambda\rangle)^{m_b(\zeta,\eta,\lambda)-|\beta|-k}.
\end{equation}
In particular we have $\partial_\lambda^k \sigma_b(\zeta,\eta,0)\in \log(2+\xi^2)^{k} S^{m_b(\zeta,\eta,0)-k}_{\log}(\FibreL)$. Now by remainder estimates on the Taylor series in $\lambda$, for any $(\zeta,\eta)\in T^*\FibreM$ and $\lambda\in \R$, there is $|p_{\zeta,\eta,\lambda}|\leq|\lambda|$ such that 
\[
 \sigma_b(\zeta,\eta,\lambda) = \sum_{k=0}^{2N}\frac{1}{k!}\partial_\lambda^k\sigma_b(\zeta,\eta,0)\lambda^k + 
 \frac{\lambda^{2N+1}}{(2N+1)!}(\partial_\lambda^{2N+1}\sigma_b)(\zeta,\eta, p_{\zeta,\eta,\lambda}).
\]
Plugging this into the formula for $\sigma_{b,\lambda}$ yields
\[\begin{split}
\sigma_{b,\lambda}(\zeta,\eta) := \sigma_b(\zeta,\eta,0) &+\sum_{k=1}^{2N}c_k\cdot \partial_\lambda^k\sigma_b(\zeta,\eta,0) \\
	&+ \frac{1}{2\pi h}\int_\R\frac{\lambda^{2N+1}}{(2N+1)!}(\partial_\lambda^{2N+1}\sigma_b)(\zeta,\eta, p_{\zeta,\eta,\lambda})\hat\chi_C\left(\frac{-i\lambda-\lambda'}{h}\right) d\lambda'.
\end{split}\]
Now \eqref{eq:sigma_b_la_symbol_estimate} assures that the last term is in $S^{-N-1}(\FibreL)$. Putting everything together we conclude $\sigma_{b,\lambda} - \sigma_b(\cdot,\cdot,0) \in \log(2+\xi^2)S^{m_{b,0} - 1}_{\log}(\FibreL)$.
\end{proof}
Now, we can define spaces on $\FibreL\to\FibreM$:
\begin{definition}\label{def:anisotropic_indicial_spaces}
Let $k_b = \escapeparam m_b+N$ be a weight. We denote by $\mathsf{H}^{\escapeparam m_b + N}_\lambda$ the space 
\begin{equation}\label{eq:def-indicial-spaces}
I(\Op^b(e^{-\escapeparam G_b}\langle \xi \rangle^{-N})_C, \lambda) L^2(\FibreL),
\end{equation}
endowed with the corresponding norm.
\end{definition}

\begin{remark}
The different letters are associated to functional spaces on different objects. First, $H^{\mathbf{k}}$, or $H^{\mathbf{k}}(M,L)$ is a space on the whole manifold. Then, $\mathcal{H}^{k_b}$ is the corresponding space ``restricted'' to the zeroth Fourier mode in a cusp. Finally, $\mathsf{H}^{k_b}_\lambda$ is the ``Fourier Transform'' of $\mathcal{H}^{k_b}$.
\end{remark}

Let us discuss the $\lambda$ subscript in the notation of the spaces $\mathsf H^{\escapeparam m_b + N}_\lambda$. It may seem that these spaces depend on the parameter $\lambda$, and since we want to consider analytic families of operators depending on the parameter $\lambda$, this may be problematic --- recall that for the theory in Kato \cite{Kato-80} to apply, we need that operators are of type (A), which basically means that they all act on the same domain. To address this problem, we start with
\begin{lemma}\label{lemma:equivalence-indicial-spaces}
For any weight $k_b$, the space $\mathsf{H}^{k_b}_\lambda(\FibreL)$ does 
not depend on the $\lambda$ parameter. Only the norm does, and
it varies continuously with $\lambda$.
\end{lemma}

\begin{proof}
Recall from Definition~\ref{def:anisotropic_b_spaces} that $\Op^b(e^{-\escapeparam G_b}\langle \xi \rangle^{-N})_C^{-1}$ exists for small enough $h>0$ and by the fact that $I(A,\lambda)$ is an algebra homomorphism we get 
\[
I(\Op^b(e^{-\escapeparam G_b}\langle \xi \rangle^{-N})_C, \lambda)^{-1}= I(\Op^b(e^{-\escapeparam G_b}\langle \xi \rangle^{-N})_C^{-1}, \lambda').
\]
It suffices to check that the operators
\[
I(\Op^b(e^{-\escapeparam G_b}\langle \xi \rangle^{-N})_C, \lambda)I(\Op^b(e^{-\escapeparam G_b}\langle \xi \rangle^{-N})_C^{-1}, \lambda')
\]
and
\[
I(\Op^b(e^{-\escapeparam G_b}\langle \xi \rangle^{-N})_C^{-1}, \lambda)I(\Op^b(e^{-\escapeparam G_b}\langle \xi \rangle^{-N})_C, \lambda')
\]
are bounded on $L^2$ for $\lambda,\lambda'\in \C$ (and depend continuously on $\lambda$, $\lambda'$). Since they are pseudo-differential, and their symbols have the same asymptotics for large $\eta$ (see \eqref{eq:leading_asym_sigma_lambda}), this is a consequence of usual pseudo-differential arguments.
\end{proof}

Besides the indicial family we will need the following inverse construction:
\begin{defprop}\label{defprop:A(I)}
Let $\rho_0<\rho_1$ and for $\Re \lambda \in ]\rho_0, \rho_1[$ let $\lambda\mapsto I(\lambda)$ be a holomorphic family of continuous operators  $I(\lambda):C^\infty(\FibreL)\to\mathcal D'(\FibreL)$. Also assume that it is tempered, i.e $\|I(\lambda)\|_{H^N(\FibreL)\to H^{-N}(\FibreL)} \leq C\langle \Im \lambda\rangle^N$, with $C,N>0$ depending continuously on $\Re \lambda$.

Then, for $\rho \in ]\rho_0,\rho_1[$, there is a continuous operator
$A(I,\rho):C_c^\infty(\R\times\FibreL) \to\mathcal D'(\R\times\FibreL)$ with kernel given by
\[
e^{\rho (r-r')/h} \mathscr{F}_h^{-1}( I( \rho + i\cdot))(r-r').
\]
The resulting operator does not depend on $\rho$, so we denote it by $A(I)$. In the case that $I(\lambda) = I(A, \lambda)$ for $A \in \Psi_{b,C}$ or some $A$ as in Lemma~\ref{def:extension-I(A)}, we get that $A(I) = A$.

Furthermore for two families $I_1(\lambda), I_2(\lambda)$ of operators holomorphic on $\Re(\lambda)\in]\rho_1,\rho_2[$ fulfilling $\|I_1(\rho+iw)\|_{H^{k_1}(\FibreL)\to H^{k_2}(\FibreL)} \leq C\langle w\rangle^{N_1}$ and $\|I_2(\rho+iw)\|_{H^{k_2}(\FibreL)\to H^{k_3}(\FibreL)} \leq C\langle w\rangle^{N_2}$ one has
$A(I_2 I_1) = A(I_2)A(I_1)$. 
\end{defprop}

\begin{proof}
Let us first check that the given kernel defines a well defined continuous operator $A(I,\rho):C_c^\infty(\R\times\FibreL) \to\mathcal D'(\R\times\FibreL)$: The expression of the kernel means that for $f_1,f_2\in C_c^\infty(\FibreL)$, $g_1,g_2\in C_c^\infty(\R)$ one has
\begin{align}
 \big\langle f_1g_1, &A(I,\rho)f_2g_2 \big\rangle_{C_c^\infty(\R\times\FibreL),\mathcal D'(\R\times\FibreL)} \nonumber\\
 & := \frac{1}{2\pi h} \int e^{iw(r-r')/h}e^{\rho(r-r')/h}g_1(r)g_2(r') \langle f_1, I(\rho+iw)f_2 \rangle_{C_c^\infty(\FibreL),\mathcal D'(\FibreL)} drdr'dw \nonumber,\\
 &= \int \mathscr F_h^{-1}(e^{\rho\bullet/h}g_1)(w)\mathscr F_h(e^{-\rho\bullet/h} g_2)(w)
 \langle f_1, I(\rho+iw)f_2 \rangle_{C_c^\infty(\FibreL),\mathcal D'(\FibreL)}dw.
 \label{eq:A(I)-details}
\end{align}
As the Fourier transform of compactly supported functions extend holomorphically to $\C$ the independence from $\rho$ follows from Cauchy's theorem. The fact that $A(I,\rho)$ can be extended continuously to arbitrary (nonproduct) elements of $C_c^\infty(\R\times\FibreL)$ can be seen by letting any of the $f_1,f_2\in C^\infty(\FibreL)$ or $g_1, g_2\in C_c^\infty(\R)$ to zero in the corresponding topologies. Then the temperedness assumption of $I(\lambda)$ implies that \eqref{eq:A(I)-details} goes to zero.
As to why $A(I(A))=A$, this follows by Fourier inversion after plugging in the definitions \eqref{eq:indicial_alternative_def} and \eqref{eq:A(I)-details} and a few lines of straightforward calculations. Also the multiplicativity $A(I_2 I_1) = A(I_2)A(I_1)$ follows from a straightforward calculation which is completely analogous to the calculations needed to show that the Fourier transform of a product is the convolution of Fourier transforms. 
\end{proof}

Next, we have the lemma on boundedness:
\begin{lemma}\label{lemma:equivalence-boundedness}
Consider $I(\lambda)$ as in Definition-Proposition \ref{defprop:A(I)}. We have the following identities. For $h\rho \in ]\rho_0, \rho_1[$, and two weights $k_b=\escapeparam m_b+N$ and $\ell_b=\escapeparam'm_b + N$,
\begin{equation}
 \label{eq:AI_operator_formula}
\| A(I, h\rho) \|_{\mathcal{H}_{b,\rho}^{k_b} \to \mathcal{H}_{b,\rho}^{\ell_b}} = \sup_{\Re \lambda = h\rho} \| I(\lambda) \|_{ \mathsf{H}^{k_b}_\lambda \to \mathsf{H}^{\ell_b}_\lambda}.
\end{equation}
\end{lemma}

\begin{proof}
The first step is to reduce to the case $\rho=0$: Since
\[
\| A(I) \|_{\mathcal{H}_{b,\rho}^{k_b} \to \mathcal{H}_{b,\rho}^{\ell_b}} =
	 \left\|e^{-\rho r} A(I) e^{\rho r} \right\|_{\mathcal{H}_{b,0}^{k_b} \to \mathcal{H}_{b,0}^{\ell_b}},
\]
the action of $A(I)$  is equivalent to the action of $A_\rho$ on $\mathcal{H}_{b,0}^{k_b} \to \mathcal{H}_{b,0}^{\ell_b}$, where $A_\rho$ is an operator whose kernel is that of $A$ multiplied by $e^{\rho(r'-r)}$, i.e it is
\[
\mathscr{F}_h^{-1}( I( h\rho + i\cdot))(r-r').
\]
Let $I_\rho(\lambda) = I(\lambda + h\rho)$. We deduce that the action of $A(I,h\rho)$  is equivalent to the action of $A(I_\rho, 0)$ on $\mathcal{H}_{b,0}^{k_b} \to \mathcal{H}_{b,0}^{\ell_b}$. 
Next we conjugate to an operator $L^2\to L^2$. By definition,
\[
\begin{split}
\| A(I_\rho) &\|_{\mathcal{H}_{b,0}^{k_b} \to \mathcal{H}_{b,0}^{\ell_b}} = \\
	& \left\|\Op\left[e^{-\escapeparam 'G_b}\langle \xi\rangle^{-N'}\right]_C^{-1} A(I_\rho)  \Op\left[e^{-\escapeparam G_b}\langle \xi\rangle^{-N}\right]_C \right\|_{L^2 \to L^2}.
\end{split}
\]
and
\[
\begin{split}
\| I_\rho(\lambda)& \|_{ \mathsf{H}^{k_b}_\lambda \to \mathsf{H}^{\ell_b}_\lambda} = \\
&\left\|I(\Op(e^{-\escapeparam 'G_b}\langle \xi\rangle^{-N'}), \lambda)^{-1} I(\lambda) I(\Op(e^{-\escapeparam G_b}\langle \xi\rangle^{-N}), \lambda) \right\|_{L^2 \to L^2}.
\end{split}
\]
Now, both maps $A\to I(A)$ and $I \to A(I)$ are multiplicative. We deduce that it suffices to prove the lemma in the case that $k_b = \ell_b = 0$. 

After this additional reduction, we are left to prove that 
\[
\|A(I)\|_{L^2 \to L^2} = \sup_{\Re \lambda = 0} \| I(\lambda)\|_{L^2\to L^2}.
\]
This is just an avatar of the Plancherel formula: By the definition of $A(I)$ (see \eqref{eq:A(I)-details}) one has for $f_1,f_2\in C_c^\infty(\FibreL)$, $g_1,g_2\in C_c^\infty(\R)$:
\begin{align}
 \big\langle f_1g_1, &A(I,\rho)f_2g_2 \big\rangle_{L^2(\R\times\FibreL)} \nonumber\\
 &= \frac{1}{2\pi h}\int \overline{\mathscr F_h(g_1)(w)}\langle f_1, I(\rho+iw)f_2 \rangle_{L^2(\FibreL)}\mathscr F_h(g_2)(w)\,dw, \nonumber
\end{align}
and from this formula \eqref{eq:AI_operator_formula} can be read off directly.
\end{proof}

Finally, we get
\begin{proposition}\label{prop:constant-domain-order1-operators}
Let $\mathbf{X}_b\in \Psi^1_C(\R\times\FibreL)$ and let $k_b$ be a weight. Then each $I(\mathbf{X}_b, \lambda)$ has a unique extension as a closed operator on $\mathsf{H}^{k_b}_\lambda(\FibreL)$. The domain, as a subset of $\mathsf{H}^{k_b}_\lambda(\FibreL)= \mathsf{H}^{k_b}_0(\FibreL)\subset \mathcal D'(\FibreL)$ does not depend on $\lambda$.
\end{proposition}

\begin{proof}
Since $\mathbf{X}_b$ is of order $1$, $I(\mathbf{X}_b, \lambda)$ and $I(\mathbf{X}_b,\lambda')$ differ by an order $0$ operator, which acts boundedly on each $\mathsf{H}^{k_b}_\lambda(\FibreL)$. So it suffices to check the case $\lambda = 0$. The operator $I(\mathbf{X}_b,0):C^\infty(\FibreL)\subset \mathsf{H}^{k_b}_0(\FibreL)\to \mathsf{H}^{k_b}_0(\FibreL)$
 is unitarily equivalent to the operator 
\[
\begin{split}
  W = I(\Op^b(e^{-\escapeparam G_b}\langle \xi \rangle^{-N})_C, 0) ^{-1} I(\mathbf{X}_b, 0) I(&\Op^b(e^{-\escapeparam G_b}\langle \xi \rangle^{-N})_C, 0): \\
  	& C^\infty(\FibreL)
  \subset L^2(\FibreL) \to L^2(\FibreL)
\end{split}
\]
 and $W$ is a PDO of order one in $\FibreL$. Now the uniqueness of the closed extensions
 follows from the proof of \cite[Lemma A.1]{Faure-Sjostrand-10}
\end{proof}

As a consequence, the family $I(\mathbf{X}_b, \lambda)$ is a type (A) family, so that we can apply the results from \cite{Kato-80}.

\subsection{Fredholm Indicial families}

We now come back to admissible triples and will prove that their indicial families are Fredholm
\begin{lemma}\label{lemma:Indicial-invertible-high-lambda}
Assume that $\Re(s) > 1 + C_\delta + C (h^{-1}|\Re \lambda| + |N|) - \escapeparam $ and $|\Im s|\leq h^{-1/2}$. Then $I(\mathbf{X}_b - Q_b - hs, \lambda)$ is invertible with norm $\mathcal{O}(1/h)$ on $\mathsf{H}^{\escapeparam m_b + N}_\lambda$, uniformly in $\lambda$ .

Additionally, if either 
\begin{equation}\label{eq:Indicial-invertibility}
|\Im s|\leq h^{-1/2},\ \text{and}\ \Re(s) > 1 + C_\delta + C(h^{-1}|\Re \lambda| + |N|) - \escapeparam ,\ \text{and}\ |\Im \lambda|>4\delta',
\end{equation}
or
\begin{equation*}
\Re(s) > C(1 + \escapeparam  + h^{-1}|\Re \lambda| + |N|),
\end{equation*} 
then $I(\mathbf{X}_b - hs, \lambda)$ is also invertible with norm $\mathcal{O}(1/h)$ on $\mathsf{H}^{\escapeparam m_b + N}_\lambda$, uniformly in $\lambda$.
\end{lemma}

\begin{proof}
We start with the invertibility of $I(\mathbf X_b-Q_b-hs,\lambda)$: by Lemma \ref{lemma:Inversion-up-to-smoothing-indicial-operator} we deduce that $(\mathbf{X}_b-Q_b-hs)^{-1}$ is a well defined convolution operator for $s$ in the announced domain. It furthermore fulfills all requirements of Definition~\ref{def:extension-I(A)} and thus $I((\mathbf{X}_b-Q_b-hs)^{-1} , \lambda)$ is well defined. By the multiplicativity of $I(\cdot,\lambda)$ we conclude $I((\mathbf{X}_b-Q_b-hs), \lambda)^{-1} =I((\mathbf{X}_b-Q_b-hs)^{-1} , \lambda)$ and Lemma~\ref{lemma:equivalence-boundedness} implies that it is $\mathcal{O}(1/h)$ uniformly in $\lambda$.

Now, we turn to the case of $I(\mathbf{X}_b - hs, \lambda)$, in the region that $I(\mathbf{X}_b-Q_b-hs,\lambda)$ is invertible. First we study $I(Q_b,\lambda)$. Since $Q_b$ is microsupported for $|\xi|<3\delta'$, we can use Lemma \ref{lemma:Indicial-symbol} and Equation \eqref{eq:expression-sigma-b-lambda} to deduce that when $|\Im \lambda| > 4\delta'$, $I(Q_b, \lambda) = \mathcal{O}(h^\infty)$ in $\Psi^{-\infty}$ uniformly in $\Im\lambda$ and locally uniformly in $\Re \lambda$. This implies that $I(\mathbf{X}_b-hs,\lambda)$ is invertible because
\begin{equation}\label{eq:Indicial-op-in-Fredholm-form}
\begin{split}
I(\mathbf{X}_b-hs,\lambda&)I(\mathbf{X}_b-Q_b-hs,\lambda)^{-1} = \\
	&\mathbb{1}  + I(Q_b,\lambda)I(\mathbf{X}_b-Q_b -h s,\lambda)^{-1}= \mathbb{1} + \mathcal{O}(h^\infty).
\end{split}
\end{equation}
(the remainder being bounded on the relevant spaces). 

Finally, when $\Re(s) > C( 1 + \escapeparam  + h^{-1}|\Re \lambda| + |N|)$, recall formula \eqref{eq:conjugation-action-real-part-Xb-Qb} (removing the $Q_b$ part). We deduce that the sharp G\r{a}rding inequality applies to show that $I(\mathbf{X}-hs,\lambda)$ is invertible with norm $\mathcal{O}(1/h)$ uniformly in $\lambda$, provided $C$ is large enough (as in the proof of Lemma \ref{lem:invertible-on-the-right}).
\end{proof}

Now we get to the aim of this section: Recall from the Definition~\ref{def:free-admissible-triple} of 
$\mathbf X_b$ and Definition\ref{def:indicial_family} that $I(\mathbf X_b-hs, h\lambda) = h(P_\lambda-s)$
where $P_\lambda$ is an $h$-independent holomorphic family of differential operators on $\FibreL$. 
By Lemma~\ref{lemma:Indicial-invertible-high-lambda} we know that $I(\mathbf{X}_b-hs,h\lambda)^{-1} : L^2(\FibreL)\to L^2(\FibreL)$ is well defined and holomorphic on $\{\Re(s)>C(1+|\Re \lambda|)\}\subset \C^2$:
\begin{proposition}\label{prop:meromorphic indicial_resolvent}
 We have meromorphic extension of $I(\mathbf{X}_b-hs,h\lambda)^{-1}: L^2(\FibreL)\to L^2(\FibreL)$ to $\C^2$ as operators $I(\mathbf{X}_b-hs,h\lambda)^{-1}:C^\infty(\FibreL)\to\mathcal D'(\FibreL)$
\end{proposition}

\begin{proof}
All the work has already been done in some sense, since Formula \eqref{eq:Indicial-op-in-Fredholm-form} shows that up to an invertible operator, $I(\mathbf{X}_b-hs,\lambda)$ can be written as $\mathbb{1} + K(\lambda,s)$, where $K$ is a holomorphic family of compact operators (recall that $I(Q_b,\lambda) \in \Psi^{-\infty}(\FibreM)$). The statement then follows from analytic Fredholm theory.
\end{proof}

\subsection{Effective continuation}

In this last subsection of Section~\ref{sec:continuation-indicial-resolvent}, we want to establish a meromorphic continuation of $(\mathbf X_b-hs)^{-1}$.

Before going on with the proof, let us come back to the convolution operator on the real line $T_f: g \mapsto f\ast g$, with $f\in\mathcal{D}'(\R)$ compactly supported. In the language above $I(T_f,h\lambda) = \hat f(-i\lambda)$. Since $f$ is compactly supported, $\hat{f}$ is an entire function, and acts by multiplication on the whole of $\C$. Given $s\in \C$, the function $(\hat{f}-s)^{-1}$ is a meromorphic function. So one can define for $g\in C^\infty_c(\R)$ and $\rho_0\in\R$,
\[
 R_f(\rho_0,s)g(x) := \frac{1}{2\pi i} \int_{\Re \lambda = \rho_0} e^{\lambda x} \frac{\hat{g}(-i\lambda)}{\hat{f}(-i\lambda) - s }d\lambda.
\]
Note that in the general notation from Definition-Proposition~\ref{defprop:A(I)} we can identify after setting $h=1$,  $R_f(\rho_0,s) = A((I(T_f,\lambda) - s)^{-1}, \rho)$. One finds that $(T_f -s)R_f(0, s) = \mathbb{1}$ when $s$ is not in the closure of $\hat{f}(\R)$. By Cauchy's theorem, $R_f(\rho_0,s) = R_f(\rho_1, s)$ when $\hat{f}(-i\lambda)$ does not take the value $s$ in the region $\Re \lambda \in [ \rho_0, \rho_1]$. Now, consider $\lambda_1\in \C$ such that $\hat{f}(-i\lambda_1)= s$, $\hat{f}'(-i\lambda_1)\neq 0$, and $\hat{f}(-i\cdot)$ does not take the value $s$ another time in a region $\Re \lambda \in ]\Re \lambda_1 - \epsilon, \Re\lambda_1 + \epsilon[$ for some $\epsilon>0$. Another application of Cauchy's theorem gives
\begin{equation}\label{eq:contour-deformation-resolvent-real-line-convolution}
(R_f(\Re(\lambda_1) + \epsilon, s) - R_f(\Re(\lambda_1) - \epsilon, s))g(x) = -i e^{\lambda_1 x}\frac{\hat{g}(-i\lambda_1)}{\hat{f}'(-i\lambda_1)}.
\end{equation}

Using this argument, one can hope to obtain a meromorphic continuation of the resolvent of a translation invariant operator $\mathbf X_b$ from the meromorphicity of the resolvent of its indicial family $I(\mathbf X_b,\lambda)$. This is done by replacing ``multiplication'' by ``action in the $\zeta$ variable''. This heuristics is at the core of Melrose's b-calculus and will be pursued here. As one can expect, just as it is crucial to follow the solutions of $\hat{f}(i\lambda) =s$ for the convolution by $f$, we have to follow the $(\lambda,s)$'s such that $I(\mathbf{X}_b - hs, h\lambda)$ is not invertible. 
\begin{definition}\label{def:roots-are-affine}
Given an admissible triple $\mathbf{X}_b$, $G_b$, $Q_b$, let us consider the meromorphically continued family of indical operators $I(\mathbf{X}_b-hs,h\lambda)^{-1}$ from Proposition~\ref{prop:meromorphic indicial_resolvent}. 
\begin{enumerate}
 \item For fixed $s\in\C$, the set of $\lambda\in\C$ such that $I(\mathbf{X}_b-hs,h\lambda)^{-1}$ is singular is the set of ($s$-)\emph{indicial roots} of $\mathbf{X}_b$. It will be denoted by $\specb(s)$ and by construction it is independent of $h$.
\item If there are $a_k\in \R \setminus \{0\}$ $|a_k|\leq C$ uniformly in $k$, $b_k\in\C$ such that $\specb(s) = \{a_ks+b_k\}$ then we say that the \emph{roots are affine}.
\item For affine roots say that a root $\lambda_k(s) = a_k s + b_k$ is positive if $a_k>0$ (resp. negative if $a_k<0$) and we denote the set of positive/negative roots by $\specb^\pm(s)$.
\item For any $-\infty\leq\rho<\rho'\leq\infty$ we define 
\[
\specb^{(\pm)}(s,\rho,\rho'):= \{ \lambda\in\specb^{(\pm)}(s)\ |\ \rho < \Re \lambda < \rho' \}.
\]
In particular, we call elements of $\specb^{+}(s,-\infty,0)$ (resp. $\specb^{-}(s,0,+\infty)$) the positive (resp. negative) visible roots.
\end{enumerate}
\end{definition}

By the analytic Fredholm theorem, the set 
\[
\mathfrak{C}:=\{(\lambda,s)\ |\ I(\mathbf{X}_b-hs,h\lambda) \text{ is not invertible}\},
\]
is a complex analytic submanifold of $\C^2$, possibly with algebraic singularities --- corresponding to intersection of indicial roots. 
The set $\specb(s)$ is the intersection of $\mathfrak{C}$ with $\{(\lambda,s)\ |\ \lambda\in\C\}$. 
From Proposition \ref{prop:meromorphic indicial_resolvent}, we deduce that the set of roots depends neither on the choice of $Q_b$ nor on that of $G_b$. 
From now on, we work under the assumption that all indicial roots are affine. (This implies in particular that there are no algebraic singularities in $\mathfrak{C}$).
\begin{example}\label{exmpl:indicial_roots}
In Section~\ref{sec:explicit-computations} we will be able to explicitly compute the indicial roots for the geodesic flow vector field (and even for admissible lifts in the sense of Definition \ref{def:admissible_vector_bundle}). In the scalar case we get (see Proposition~\ref{prop:indicial_roots}) 
\[
 \specb(s) = \{\pm (s+(d/2 + n)), n\in \mathbf N\}
\]
\end{example}
\begin{figure}
\centering
\def\svgwidth{1\linewidth}
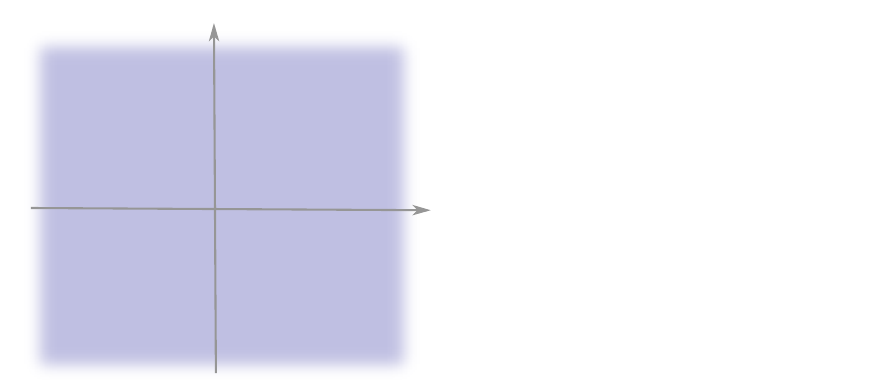
\caption{\label{fig:indicial_roots} Indicial roots for the geodesic flow on a $d+1$ dimensional cusp. On the left the Situation is depicted for $\Re(s)=0$ and on the right for negative $Re(s) = -(d/2+4.5)$. The visible positive and negative roots that have crossed the imaginary axis are marked in red.}
\end{figure}
\begin{conjecture}
The roots of an admissible triple are always affine, and for $\Re s = 0$, no root is on the imaginary axis.
\end{conjecture}
The next theorem is the technical heart of our article. In order to formulate it we introduce
\begin{equation}\label{eq:def-rho-max}
\begin{split}
\rho_{\max}(s):= &\max\{0\}\cup\\
	&\left\{ |\Re \lambda|\ |\ \lambda \in \specb^{+}(s,-\infty, 0)\text{ or }\lambda\in\specb^{-}(s,0,\infty) \right\}
\end{split}
\end{equation}
which encodes the maximal real part of the visible indicial roots. Note that under the assumption that the roots are affine we deduce that 
$\rho_{\max}(s)$ is continuous and depends only on $\Re(s)$. Furthermore, $\{\tau\ |\ \rho(\tau)\neq 0\}$ is a semi-bounded interval $]-\infty,\tau_0[$ on which $\rho_{\max}$ it is strictly decreasing. 

\begin{theorem}\label{thm:Continuation-Indicial-Resolvent}
Assume that for some h-differential operator $\mathbf X_b$ there exist $Q_b$ and $G_b$ such that $(\mathbf{X}_b, G_b, Q_b)$ form an admissible triple with affine roots. Then the inverse $(\mathbf{X}_b-hs)^{-1}$, defined as a bounded operator on $L^2(\R\times\FibreL)$ for $\Re(s) > C$ for some constant $C>0$, has a meromorphic extension $\IndicialRes^{\mathbf{X}_b}(s)$ to $\C$, as an operator mapping $C^\infty_c(\R\times\FibreL)$ to $\mathcal{D}'(\R\times\FibreL)$. 

Additionally, given any $\tau,N\in \R$, $\escapeparam > 1 + C_\delta +C(|\rho_{\max}(\tau)| + |N|)- \tau$ (with the constants of Lemma \ref{lemma:Inversion-up-to-smoothing-indicial-operator})  then $\IndicialRes^{\mathbf{X}_b}(s)$ is a meromorphic family of bounded operators 
\begin{equation}\label{eq:bounded_indicial_resolvent}
\IndicialRes^{\mathbf{X}_b}(s):e^{-\rho_{\max}(\tau) \langle r \rangle}\mathcal{H}^{\escapeparam  m_b + N}_{b,0} \to e^{\rho_{\max} (\tau) \langle r \rangle}\mathcal{H}^{\escapeparam  m_b + N}_{b,0}.
\end{equation}
on the domain $\Re(s)>\tau$  and $|\Im s |\leq h^{-1/2}$. At the eventual poles, the order is finite, and the rank of the Laurent expansion is also finite.
\end{theorem}

The remainder of this section is devoted to the proof of Theorem \ref{thm:Continuation-Indicial-Resolvent}. We start with some observations. As a direct consequence of Lemma~\ref{lemma:Indicial-invertible-high-lambda}, we get:
\begin{lemma}
 Let $\specb^\pm(s)=\{a_{k,\pm} s + b_{k,\pm}\}$, then there is a constant $C\in \R$ such that $\pm \Re(b_{k,\pm})>C$. Furthermore for all $R>0$:
 \[
  \sup\{|\Im b_k|: |\Re(b_k)|<R\} <\infty
 \]
\end{lemma}
\begin{proof}
 For the first statement we recall from Lemma~\ref{lemma:Indicial-invertible-high-lambda} that there is a constant $C$ such that for $\Re(s)>C$ there are no indicial roots with $\Re(\lambda)=0$. Consequently
 all positive indicial roots satisfy $\Re(a_k s + b_k)>0$ if $\Re(s)>C$. By the assumption that $a_k$ are uniformly bounded the assertion follows for positive roots and is completely analogous for negative ones. 
 
The second statement follows directly from considering \eqref{eq:Indicial-invertibility} in the case $s=0$.
\end{proof}
We call $\rho\in\R$ s-regular if $\specb(s) \cap (\rho+i\R) = \emptyset$. The above bounds on $a_k$ and $b_k$ imply that for any $s\in \C$, the set of $s$-regular $\rho\in\R$ 
is open and dense. Furthermore for a $s$-regular $\rho$, Lemma~ \ref{lemma:Indicial-invertible-high-lambda} implies that 
\[
I(\mathbf{X}_b-hs,h\rho+i w)^{-1}:\mathsf{H}^{\escapeparam m_b + N}_{h\rho+i w} \to \mathsf{H}^{\escapeparam m_b + N}_{h\rho+i w} 
\]
is uniformly bounded in $w\in\R$ with norm $O(1/h)$ provided that 
$\escapeparam>1+C_\delta+C(\rho+|N|)-\Re(s)$. Thus we 
can define
\[
\IndicialRes^{\mathbf{X}_b}_{\rho}(s) := A( I(\mathbf{X}_b-hs,\lambda)^{-1},h\rho): 
\mathcal H^{\escapeparam m_b + N}_{b,\rho}\to \mathcal H^{\escapeparam m_b + N}_{b,\rho}
\]
which is again bounded with norm $O(1/h)$. Indeed one directly checks that 
\[
(\mathbf X_b-hs)\IndicialRes^{\mathbf{X}_b}_{\rho}(s) = 
\IndicialRes^{\mathbf{X}_b}_{\rho}(s)(\mathbf X_b-hs) = \mathbb{1}.
\]
However, $\IndicialRes^{\mathbf{X}_b}_{\rho}(s)$ depends strongly on the choice of $\rho$ due to the fact that $I(\mathbf{X}_b-hs,\lambda)^{-1}$ has singularities, i.e. that there exist indicial roots. In order to understand the meromorphic continuation one has examine what happens if indicial roots cross the integration contours. 

In order to shorten the notation in the sequel it is convenient to define 
\[
F:(s,\lambda)\mapsto  h e^{\lambda (r-r')}I(\mathbf{X}_b-hs, h\lambda)^{-1},
\]
seen as a meromorphic function on $\C^2$ taking values in convolution 
operators on $\R\times\FibreL$. The $h$ factor is actually chosen such 
that it becomes $h$-independent and using the Definition of $A(I)$ (Definition-Proposition\ref{defprop:A(I)}) we write
\begin{equation}\label{eq:R_as_F_integral}
 \mathbf R_{\rho}^{\mathbf X_b}(s)=
 \frac{1}{2\pi i h} \int_{\Re(\lambda)=\rho} F(s,\lambda)d\lambda.
\end{equation}
When freezing the $s$ variable, the 
poles of $F(s,\cdot)$ are precisely $\specb(s)$. We can integrate $F$ over a small closed 
curve $\gamma$ around a pole $\lambda_0$, enclosing only $\lambda_0$, and 
obtain its residue $\Res(F(s,\cdot),\lambda_0)$ in the $\lambda$ variable.
With this notation, we can state an equivalent to 
equation~\eqref{eq:contour-deformation-resolvent-real-line-convolution}:
\begin{lemma}\label{lemma:contour-deformation-indicial-resolvent}
Let $-\infty<\rho<\rho'<\infty$ be $s$-regular for some $s\in \C$. 
Then we have $\specb(s,\rho,\rho')$ is finite and
\[
\IndicialRes^{\mathbf{X}_b}_{\rho'}(s) - \IndicialRes^{\mathbf{X}_b}_{\rho}(s) = h^{-1}\sum_{\lambda\in\specb(s,\rho,\rho')} \Res(F(s,\cdot),\lambda)
\]
\end{lemma}

\begin{proof}
That the sets are finite follows from the uniform estimates on $a_k, b_k$. The identity is a consequence of Cauchy's theorem and \eqref{eq:R_as_F_integral}.
\end{proof}

The following lemma is crucial in the proof:
\begin{lemma}\label{lemma:perturbation-residues}
Define $B(s,\lambda):=h^{-1}\Res( F(s, \cdot), \lambda)$ which is by definition a convolution operator on $\R\times L$. Consider a parametrized indicial root $\lambda_k(s)=a_k s + b_k$ and  Then the map
\[
s \mapsto B(s,\lambda_k(s))
\]
is a meromorphic function of $s$, and the set of poles is contained in the set of $s$'s such that $\lambda_k(s)$ crosses another root.
\end{lemma}

\begin{proof}
Since we already know that we can parametrize the roots without algebraic singularities --- in the words of Kato, there is no branching point --- this is a direct consequence of Theorem 1.8 in \cite[p.70]{Kato-80}
\end{proof}

Now we can come back to the proof of our theorem.
\begin{proof}[Proof of Theorem \ref{thm:Continuation-Indicial-Resolvent}]
First, we focus on the meromorphic continuation of the Schwartz kernel of the resolvent. Recall from Lemma~\ref{lemma:Indicial-invertible-high-lambda} that there is $C$ such that $I(\mathbf X_b - hs,i\xi)$ is invertible for $\Re(s)>C$
and in this half plane, we define
\[
\IndicialRes^{\mathbf{X}_b}(s) = A( I(\mathbf{X}_b - hs, \lambda)^{-1}, 0).
\]
If $\rho_1<0<\rho_2$ are such that $\{\Re \lambda\in[\rho_1,\rho_2]\}$ does not intersect $\specb(s)$ then we deduce that $\IndicialRes^{\mathbf{X}_b}(s)$ is bounded on all spaces $\mathcal{H}_{b,\rho}^{k_b}$ for $\rho\in [\rho_1,\rho_2]$, given that the weight $k_b$ is large enough. 

We want to construct a meromorphic continuation of $\IndicialRes^{\mathbf X_b}(s)$ to all $\C$ and therefore we have to take care of the indicial roots that cross the contour at $\Re(\lambda)=0$. We define the set of \emph{positive (resp. negative) visible roots at $s$} as
 $\specb^{+}(s,-\infty,0) $ and $\specb^{-}(s,0,\infty)$, respectively (see Figure~\ref{fig:indicial_roots} for the case of the geodesic flow for cusps). 

By the uniform bounds on $a_k, b_k$, we deduce that for any $s\in\C$, there are finitely many visible roots. 

Let $\mathscr{U}$ be the set of $s\in \C$ such that $0$ is $s$-regular, i.e $\specb(s) \cap i\R=\emptyset$. For $s\in\mathscr{U}$, we set
\begin{equation}\label{eq:form-continuation-indicial-resolvent}
\IndicialRes_{\mathscr U}^{\mathbf{X}_b} (s) := \IndicialRes^{\mathbf{X}_b}_0(s) - \sum_{\lambda \in \specb^{+}(s,-\infty,0)} B(s,\lambda) + \sum_{\lambda \in \specb^{-}(s,0,\infty)} B(s,\lambda).
\end{equation}
As $\IndicialRes^{\mathbf X_b}_0(s)$ is holomorphic on any connected component of $\mathscr U$ and as $B(s,\lambda_k(s))$ are meromorphic by Lemma~\ref{lemma:perturbation-residues} this defines a meromorphic family on $\mathscr U$. It remains to prove that we can patch the different connected components of $\mathscr U$ (which are vertical strips because the roots are affine) together:

Therefore take $s_0$ such that $0$ is not $s_0$-regular, we consider $\rho< 0 < \rho'$ small enough such that $\specb(s_0,\rho,\rho') \subset i\R$. Then, for $s$ in a small vertical strip $D$ around $s_0$, $\rho$ and $\rho'$ are still $s$-regular. For $s\in D$ we define
\begin{eqnarray*}
 \IndicialRes_D^{\mathbf X_b}(s) &:= &\IndicialRes^{\mathbf X_b}_{\rho}(s) + \sum_{\lambda\in\specb^-(s,\rho,\infty)}B(s,\lambda)  - \sum_{\lambda \in \specb^{+}(s,-\infty,\rho)}B(s,\lambda)\\
 &=&\IndicialRes^{\mathbf X_b}_{\rho'}(s) - \sum_{\lambda\in\specb^{+}(s,-\infty,\rho')}B(s,\lambda)+  \sum_{\lambda \in \specb^{-}(s,\rho',\infty)}B(s,\lambda)
\end{eqnarray*}
The equality between the two expressions follows from Lemma~\ref{lemma:contour-deformation-indicial-resolvent}. By construction and by Lemma~\ref{lemma:perturbation-residues} $\IndicialRes_D^{\mathbf X_b}(s)$ defines a meromorphic operator on the strip $D$. It only remains to check that on $\mathscr U\cap D$ both definitions of $\IndicialRes_{\mathscr U}^{\mathbf X_b}(s)$ and $\IndicialRes_{\mathscr D}^{\mathbf X_b}(s)$ coincide. But this is again a direct consequence of Lemma~\ref{lemma:contour-deformation-indicial-resolvent}. We can thus patch the definitions to a globally meromorphic operator which we denote by $\IndicialRes^{\mathbf X_b}(s)$.

Now we will determine on which functional spaces this meromorphic continuation acts.  Let us focus on the structure of the residues $B$ of $F$. 
If we assume that $\lambda_0$ is an $s$-indicial root, and that for $\epsilon>0$, there are no other indicial roots in $\{ \lambda,\ |\lambda-\lambda_0|\leq \epsilon\}$. In that case, 
\[
B(s,\lambda_0) = \frac{1}{2i\pi h}\int_{|\lambda-\lambda_0|=\epsilon} F(s,\lambda)d\lambda.
\]
We will need the lemma:
\begin{lemma}
For $\epsilon>0$ and $\rho \in \R$, we have the equality of spaces
\[
e^{\rho r +\epsilon \langle r \rangle} \Op^b(e^{-\escapeparam G}) H^N(\R\times\FibreL) =  \Op^b(e^{-\escapeparam G}) e^{\rho r +\epsilon \langle r \rangle} H^N(\R\times\FibreL).
\]
The corresponding norms are equivalent with $\mathcal{O}(1)$ constants as $h\to 0$.
\end{lemma}

\begin{proof}
It suffices to prove that both
\[
e^{-\rho \langle r \rangle} \Op^b(e^{-\escapeparam G})^{-1} e^{\rho \langle r \rangle} \Op^b(e^{-\escapeparam G}),\ \Op^b(e^{-\escapeparam G})^{-1} e^{-\rho \langle r \rangle} \Op^b(e^{-\escapeparam G}) e^{\rho \langle r \rangle} 
\]
are bounded on $L^2(\R \times \FibreL)$. However since the quantization is properly supported, these operators are pseudo-differential with symbols in $1 + \mathcal{O}(h S^{-1^+})$. Hence they give rise to bounded operators on $L^2(\R \times \FibreL)$.
\end{proof}

With $\lambda_0,\epsilon$ as above we deduce
\[
\begin{split}
\| B(s,\lambda_0) & \|_{e^{-2\epsilon \langle r \rangle} \mathcal{H}^{k_b}_{b,\Re \lambda_0} \to e^{2\epsilon \langle r \rangle} \mathcal{H}^{k_b}_{b,\Re \lambda_0} } \leq C_\epsilon/h\sup_{|\lambda- \lambda_0|=\epsilon} \| e^{- \Re\lambda_0 r -2\epsilon \langle r\rangle} \\
& \Op( e^{-\escapeparam G-N\log\langle \xi\rangle})^{-1} F(s,\lambda) \Op( e^{-\escapeparam G-N\log\langle \xi\rangle}) e^{\Re \lambda_0 r -2\epsilon \langle r \rangle} \|_{L^2\to L^2}.
\end{split}
\]
If $W$ is the multiplication by $e^{-2\epsilon\langle r\rangle}$, the operator in the norm is the composition $W S_\lambda W$, so that $S_\lambda$ is a convolution operator whose kernel takes the form
\[
h e^{(\lambda-\Re \lambda_0)(r-r')} I(\Op( e^{-\escapeparam G-N\log\langle \xi\rangle})^{-1} (\mathbf{X}_b - hs) \Op( e^{-\escapeparam G-N\log\langle \xi\rangle}),h\lambda)^{-1}.
\]
Recall that $\Re s > 1 + C_\delta +C(|\rho_{\max}(s)| + |N|)-\escapeparam $ and $|\Im s |\leq h^{-1/2}$, so we can apply Lemma \ref{lemma:Indicial-invertible-high-lambda}. In particular, the indicial operator in the last line is bounded on $L^2$ with norm $C(\epsilon)$. Since the kernel of $S_\lambda$ decomposes as a product, we see directly that it is bounded from $e^{\Re(\lambda-\lambda_0)r - \epsilon\langle r\rangle} L^2$ to $e^{\Re(\lambda-\lambda_0)r + \epsilon\langle r\rangle} L^2$. But since $|\lambda-\lambda_0|=\epsilon$, it is thus bounded from $e^{-2\epsilon\langle r\rangle} L^2$ to $e^{2\epsilon\langle r\rangle} L^2$ uniformly in $\lambda$. Finally, since $W$ maps $L^2$ to $e^{-2\epsilon\langle r\rangle} L^2$ and $e^{2\epsilon\langle r\rangle} L^2$ to $L^2$, we obtain the desired result
\begin{equation}\label{eq:boundedness-residue}
\| B(s,\lambda_0) \|_{e^{-2\epsilon \langle r \rangle} \mathcal{H}^{k_b}_{b,\Re \lambda_0} \to e^{2\epsilon \langle r \rangle} \mathcal{H}^{k_b}_{b,\Re \lambda_0} } \leq C(s,\epsilon)/h,
\end{equation}
for some $C(s,\epsilon)>0$ locally uniform. On the other hand, using Lemma \ref{lemma:equivalence-boundedness}, we obtain that when $\rho$ is $s$-regular,
\begin{equation}\label{eq:boundedness-imaginary-axis}
\left\| \IndicialRes^{\mathbf{X}_b}_{\rho}(s) \right\|_{\mathcal{H}^{\escapeparam m_b}_{b,\rho} \to \mathcal{H}^{\escapeparam m_b}_{b,\rho}} \leq C_{s,\rho}.
\end{equation}
If $\Re(s)$ is such that there are no visible roots (i.e. $\specb^+(s,-\infty, 0)\cup\specb^-(s,0,\infty) =\emptyset$), then the boundedness estimate \eqref{eq:bounded_indicial_resolvent} follows directly from \eqref{eq:form-continuation-indicial-resolvent} and \eqref{eq:boundedness-imaginary-axis}.

Else, if $s\in\C$ with $\Re(s)>\tau$ such that there are visible roots, let us choose $\varepsilon >0$ such that 
\[
\max_{\lambda\in \specb^+(s,-\infty, 0)\cup\specb^-(s,0,\infty)}  |\Re(\lambda)| + 2\varepsilon <\rho_{\max}(\tau).
\]
Note that this is possible because we are in the case $\rho(\tau)>0$ and thus, as was discussed after \eqref{eq:def-rho-max}, $\rho$  is strictly monotonous. Now, combining Equations \eqref{eq:form-continuation-indicial-resolvent}, \eqref{eq:boundedness-residue} and \eqref{eq:boundedness-imaginary-axis}, we deduce that 
\begin{equation}\label{eq:boundedness-indicial-res}
\left\| \IndicialRes^{\mathbf{X}_b}(s) \right\|_{e^{- \rho_{\max}(\tau)\langle r\rangle}\mathcal{H}^{k_b}_{b,-\rho_{\max}} \to e^{\rho_{\max}(\tau)\langle r\rangle}\mathcal{H}^{k_b}_{b,\rho_{\max}}} \leq C_{s,\rho}
\end{equation}
To obtain the boundedness for $s\in \C\setminus \mathscr{U}$, one can use similar arguments.

Consider a pole $s$ of $\IndicialRes^{\mathbf{X}_b}(s)$ corresponding to an indicial root crossing $\lambda_0$. From the considerations above, it follows that the Laurent expansion has its image contained in the direct sum of 
\begin{equation}\label{eq:image-Laurent-expansion-pole-indicial-resolvent}
e^{\lambda_0 r}H_0 \oplus \dots \oplus r^k e^{\lambda_0 r}H_k,
\end{equation}
where $H_0, \dots,H_k$ are finite dimensional subspaces of $\mathsf{H}^{k_b}_{\lambda_0}(\FibreL)$, related to the images of the Laurent expansion of $I(\mathbf{X}_b - hs,\lambda)^{-1}$ around $\lambda_0$. In particular, this is finite dimensional.
\end{proof}
Note that in the case of a geodesic flow we will see in Section~\ref{sec:explicit-computations} that the resonant states of $\mathbf X_b$ coming from the indicial resolvent can be explicitely expressed by dirac distributions and homogeneous distributions on the North and South pole of $\FibreM=\mathbb S^d$.

\section{Black box formalism and main theorem}
\label{sec:Black-Box}

In this section, we introduce a black box formalism in the spirit of 
\cite{Sjostrand-Zworski-91}. For the same reason as in Section~\ref{sec:continuation-indicial-resolvent} we work in a geometric setting that is more general then the admissible bundles $L\to S^*N$ from Definition~\ref{def:admissible_vector_bundle}. Again, this bigger generality comes without any additional effort in the proofs. Let us define the geometric setting of this section.

\begin{definition}\label{def:fibred_cusp}
Let us consider a cusp $Z=[a,+\infty)\times\R^d/\Lambda$ and a product $Z\times\FibreM$ with $(\FibreM, g_{\FibreM})$ a compact connected Riemannian manifold. The product $(Z\times\FibreM, g_Z + g_{\FibreM})$ is a \emph{trivial fibred cusp}.
\end{definition}
\begin{definition}\label{def:admissible-extension}
Let $(M,g')$ be a complete connected Riemannian manifold. Assume that it can be decomposed as the union of a compact manifold $M_0$, and several ends $M_1, \dots,M_\kappa$ that are trivial fibred cusps. Then we say that $M$ is an \emph{admissible manifold}. 
\end{definition}
Observe that if $M$ is admissible, then its curvature tensor is $\mathscr{C}^\infty$ bounded.

\begin{definition}\label{def:admissible-bundle}
Let $(M,g')$ be an admissible manifold. Let $L\to M$ be a vector bundle with Riemannian bundle metric $\|\cdot\|_L$ and compatible connection $\nabla$.
We say that $L$ is a \emph{general admissible bundle} if over each cusp $Z_\ell \times \FibreM$, for $y>\mathbf{a}$, $L$ has a product structure $L_{|Z_\ell} \simeq Z_\ell\times \FibreL_\ell$, where $\FibreL_\ell$ is a Riemannian bundle over $\FibreM_\ell$.
\end{definition}
Again, if $L$ is general admissible, its curvature and derivatives are bounded.

\begin{example}
Let $(N,g)$ be an admissible cusp manifold, and let $L\to M=SN\to N$ be an admissible bundle. Then $L\to M$ is a general admissible bundle and the fibre $\FibreM$ is just the sphere $\mathbb S^d$.
\end{example}

Let $L\to M$ be a general admissible bundle, with $\kappa$ 
cusps $Z_1,\dots,Z_\kappa$. Take $\mathsf{a}>\mathbf{a}$, and let 
\begin{equation}\label{eq:def-L2-a}
L^2_\mathsf{a}(M,L) = \left\{f\in L^2(M,L)\ \middle|\ \int f|_{y>\mathsf{a}} d\theta = 0 \right\}.
\end{equation}
We have the orthogonal decomposition
\begin{equation}\label{eq:Black-Box-decomposition}
L^2(M,L) = L^2_\mathsf{a}(M,L) \oplus_{\ell=1}^\kappa L^2\Big(]\log \mathsf{a}, + \infty[\times \FibreL_\ell, e^{-r d} dr d\PointM\Big).
\end{equation}
In Section \ref{sec:continuation-indicial-resolvent}, we used the 
measure $dr d\PointM$ instead of $e^{-r d}drd\PointM$. In particular, 
\[
L^2( e^{-r d}dr d\PointM) = e^{rd/2}L^2(drd\PointM).
\]
In Equation \eqref{eq:Black-Box-decomposition}, the first term will be 
regarded as a \emph{black box} and the second one as the \emph{free 
space}. In the black box, we will use the variable $y$ (more appropriate for geometric purposes), and in the free space the $r$ variable (more appropriate for analysis). 
In the case of elliptic operators, one can really isolate the 
black box, because it can be embedded in another space where the 
relevant operator --- mostly the Laplacian --- has compact resolvent. 
However in our case, since being uniformly hyperbolic is a global property, 
such surgery cannot be performed a priori. It is the fact that the flow is \emph{exactly}
translation invariant that will save us.

We can define extension and restriction operators. Let $\phi \in 
C^\infty(M,L)$. We let $\mathscr{P}_\ell^\mathsf{a} \phi$ be the 
function in $C^\infty([\log \mathsf{a},+\infty[_r\times \FibreM_\PointM,\FibreL)$ 
obtained by restriction to the cusp $Z_\ell$ and averaging in the 
$\theta$ variable. Conversely, let $\phi \in C^\infty_c(]\log\mathsf{a},
+\infty[_r\times\FibreM_\PointM, \FibreL)$. We consider it as a function 
$\mathscr{E}_\ell^\mathsf{a}\phi$ supported in cusp $Z_\ell$, not 
depending on $\theta$. We have $\mathscr{P}_\ell^\mathsf{a} 
\mathscr{E}_\ell^\mathsf{a} = \mathbb{1}$. We extend these definitions 
to distributions by duality: for distributions $v\in\mathcal{D}'(M,L)$ and $u\in 
\mathcal{D}'([\log\mathsf{a}, +\infty[\times\FibreM,\FibreL)$,
\[
\langle \mathscr{E}_\ell^\mathsf{a} u, \phi\rangle : = \langle u, \mathscr{P}_\ell^\mathsf{a} \phi \rangle, \text{ and } \langle \mathscr{P}_\ell^\mathsf{a} v, \phi\rangle : = \langle v, \mathscr{E}_\ell^\mathsf{a} \phi \rangle.
\]
Note that after this extension we can apply $\mathscr E^\mathsf{a}_\ell$
equally to $C^\infty([\log\mathsf{a},\infty[)$ and the compostion $\mathscr{E}_\ell^\mathsf{a}\mathscr{P}_\ell^\mathsf{a}$ is well defined. Given a function 
$\chi \in C^\infty([\log \mathsf{a}, + \infty[)$ that is constant near $\mathsf{a}$, we can define the associated \emph{black box multiplication operator} as the operator
\begin{equation}\label{eq:def-chi-multiplication}
\mathbf{B}(\chi) = \chi(\log \mathsf{a}) +  \sum_{\ell} \mathscr{E}_\ell^\mathsf{a}( \chi(r) - \chi(\log\mathsf{a})) \mathscr{P}_\ell^\mathsf{a}.
\end{equation}
(here $\chi(r)$ is the multiplication operator.) Note that with this definition the operator $\mathbf B(\chi)$ acts on $L^2_{\mathsf a}(M,L)$ simply by multiplication with the constant $\chi(\log\mathsf a)$ and on the free spaces $L^2(]\log \mathsf a,\infty[ \times \FibreL_\ell)$ as a multiplication operator with $\chi(r)$.

In this section, we will define a class of operators that preserve
this structure, and review some of their properties. Then, we will 
conclude on the meromorphic extension of the resolvent of admissible 
such operators.

\subsection{The class of cusp-b-pseudors}

Now that we have added some structure to our space $L^2(M,L)$, we need to determine a reasonnable class of operators that will preserve the structure. First consider a differential operator  $P$ that commutes with $y\partial_\theta$ and $y\partial_y$ in each cusp, for $y>\mathsf{a}$. It thus acts on the space of smooth functions supported in a cusp that do not depend $\theta$. We denote by $P^0_{b,\ell}$ that restriction for each cusp $Z_\ell$. Then we find that for $i=1,\dots ,\kappa$, acting on $\mathcal{D}'(M,L)$,
\[
\mathscr{P}_\ell^\mathsf{a}  P  = P^0_{b,\ell} \mathscr{P}_\ell^\mathsf{a} .
\]
We also have the dual statement, acting on $C^\infty_c(]\mathsf{a}, +\infty[\times\FibreM_\ell, \FibreL_\ell)$:
\[
\mathscr{E}_\ell^\mathsf{a} P^0_{b,\ell} = P \mathscr{E}_\ell^\mathsf{a}.
\]
Since we want to use anistropic spaces that can only be defined using \emph{pseudo-dif\-fer\-en\-tial} operators, we have to accept slightly different relations. Indeed, pseudo-differential operators cannot be \emph{exactly} supported on the diagonal. 
\begin{definition}[cusp-b-operators]\label{def:b-operator}
Let $A$ be an operator $C^\infty_c(M,L)\to \mathcal{D}'(M,L)$. We say that $A$ is a \emph{black box operator} with precision $C\geq 1$ at height $\mathsf{a}$ if acting on $C^\infty_c(M,L)$,
\begin{equation}\label{eq:compact-stays-compact}
\mathscr{P}_\ell^{C\mathsf{a}} A ( \mathbb{1} - \mathscr{E}_\ell^{\mathsf{a}} \mathscr{P}_\ell^{\mathsf{a}}) = 0\text{, for all }\ell=1,\ldots,\kappa
\end{equation}
and acting on $C^\infty_c( ] \log(C \mathsf{a}), +\infty[\times \FibreM,\FibreL)$,
\begin{equation}\label{eq:incusp-stays-incusp}
(\mathbb 1- \mathscr{E}_\ell^\mathsf{a} \mathscr{P}_\ell^{\mathsf{a}}) A \mathscr{E}_\ell^{C\mathsf{a}} = 0\text{, for all }\ell=1,\ldots,\kappa.
\end{equation}
If additionally for each $\ell=1,\dots ,\kappa$, $\mathscr{P}_\ell^{\mathsf{a}} A \mathscr{E}_\ell^{C\mathsf{a}}$ acts on $C^\infty_c(]\log(C\mathsf{a}),+\infty[\times\FibreM,\FibreL)$ as the restriction of a translation invariant operator $A_{b,\ell}^0$ on sections of $\R \times \FibreL$, that is supported for $|r-r'|\leq \log C$, we say that $A$ is a \emph{cusp-b-operator}. 

We define
\[
A_{b,\ell}:= e^{-rd/2}A_{b,\ell}^0 e^{rd/2},
\]
which is again translation invariant. In this way, while $A_{b,\ell}^0$ acts
naturally on $L^2(\R\times\FibreL, e^{-rd}drd\PointM)$, $A_{b,\ell}$ acts on $L^2(\R\times\FibreL, drd\PointM)$. 

Finally, if $A\in \Psi(M,L)$ is also a pseudo-differential operator, we say that $A$ is
a \emph{cusp-b-pseudor}, and write $A\in \Psi_{b,C}(M,L)$.
\end{definition}

\begin{example}
In the case of the geodesic flow $M=S^\ast N$, the vector field of the geodesic flow $X$ is a cusp-b-operator with precision $1$. We also have
\[
X_{b,\ell}^0 = \cos\varphi\partial_r + \sin\varphi\partial_\varphi,\text{ and } X_{b,\ell} = \cos\varphi\partial_r + \frac{d}{2}\cos\varphi + \sin\varphi\partial_\varphi.
\]
\end{example}

In what follows the constant $\mathsf{a}$ will be fixed a priori, it is a geometric data of the problem, and we will mostly not mention it. Let us give a word of explanation. Condition \eqref{eq:compact-stays-compact} implies that if $f$ has zero mean value in the $\theta$ variable in each cusp for $y> \mathsf{a}$, then the mean value of $Af$ in the $\theta$ variable vanishes when $y>C\mathsf{a}$. The condition \eqref{eq:incusp-stays-incusp} is the dual version of the assumption: it means that if $f$ was supported only in cusps for $y>C\mathsf{a}$, and did not depend on $\theta$, then $Af$ would be supported in $y>\mathsf{a}$, and also not depend on $\theta$.

\begin{proposition}
Let $A\in\Psi_{b,C}(M,L)$. Then for $\ell=1\dots\kappa$, the operator $A_{b,\ell}$ defined from $A$ by Definition~\ref{def:b-operator} is an element in $\Psi_{b,C}(\R\times\FibreL_\ell)$ --- see Definition \ref{def:free-b-operators}.
\end{proposition}
This follows directly from the definition. Recall from section~\ref{sec:geodesic_flow_on_cusps} that covectors decompose as $\xi= Ydy+Jd\theta+\eta d\PointM$.

\begin{definition}
Let $\sigma\in S^0(M,L)$. Assume that in each cusp, $\partial_\theta \sigma = 0$ for $y>\mathsf{a}$, and for $r>\log \mathsf{a}$, $\ell=1\dots \kappa$, let
\begin{equation}\label{eq:def-sigma-b}
\sigma_{b,\ell}(r, \PointM;\lambda, \eta) := \sigma|_{Z_\ell}( e^r, \theta, \PointM; e^{-r}\lambda, J=0, \eta).
\end{equation}
Assume that $\sigma_{b,\ell}$ does not depend on $r$ for $\ell=1\dots \kappa$. Then we say that $\sigma$ is a \emph{b-symbol} of order $0$ and write $\sigma\in S^0_b(M,L)$. Given a cusp-b-symbol $m$ of order $0$, we correspondingly define $S^m_b(M,L)$ the set of cusp-b-symbols of order $m$.
\end{definition}

By a direct computation, one gets:
\begin{proposition}
For a general admissible bundle $L\to M\to N$ any $\sigma\in S(M,L)$ that is invariant under the action of local isometries $T_{\tau,\theta}$ (see equation \eqref{eq:local-isometries-r-theta}) in each cusp is a cusp-b-symbol. 
\end{proposition}

We also get
\begin{lemma}\label{lemma:global-b-symbols-give-free-b-symbols}
Let $\sigma\in S_b(M,L)$. Then for $\ell=1\dots \kappa$, $\sigma_{b,\ell} \in S_b(\R\times\FibreL)$ --- see Definition-Proposition \ref{defprop:free-b-symbols}.
\end{lemma}

\begin{proof}
From the considerations in Section \ref{sec:symbols-on-cusps}, we deduce that $\sigma_{b,\ell}$ satisfies usual symbol estimates on $\R \times \FibreL$. The $r$-invariance follows from the definition.
\end{proof}

Let us consider $\sigma\in S_b(M,L)$ and the corresponding operator $\Op(\sigma)$. According to Proposition \ref{prop:properties-quantization}, by adjusting the parameter $\mathsf{a}\geq \mathbf{a}$, we get that $\Op(\sigma)$ satisfies Equation \eqref{eq:incusp-stays-incusp}. 
It is not difficult to check that it also satisfies Equation \eqref{eq:compact-stays-compact} for similar reasons. We now consider its restriction to functions $f$ supported in the cusp $Z_\ell$ and not depending on $\theta$. 
Actually, we want to compute directly $\{\Op(\sigma)\}_{b,\ell}$ instead of $\{\Op(\sigma)\}_{b,\ell}^0$. 
Thus we take a function of the form $e^{rd/2}f(r,\PointM)$, so that the action of $\Op(\sigma)$ on $L^2(L)$ --- with the measure $e^{-r d}drd\PointM$ --- will correspond to the action on $L^2(\R\times\FibreL,drd\PointM)$. 
By definition of the quantization --- see equations \eqref{eq:def-loc-quantization} and \eqref{eq:def-loc-quantization-cutoff} --- and already replacing $(2\pi h)^{-d}\int e^{i\langle\theta-\theta',J\rangle/h} d\theta'$ by $\delta_{J=0}$, we get for $e^{-rd/2} \Op(\sigma) e^{rd/2}f$:
\[
\begin{split}
\frac{1}{(2\pi h)^{1+k}}&\int \chi^{\Op}\left(\log\frac{y}{y'}\right)e^{\frac{i}{h}(\langle y- y',Y\rangle + \langle \zeta - \zeta', \eta \rangle)}  \\
	\sigma&\bigg|_{Z_\ell}\left( \frac{y+y'}{2}, \frac{\zeta+\zeta'}{2}, Y, J =0, \eta\right) f(y',\zeta') \sqrt{\frac{y}{y'}}dy'd\zeta' dY d\eta.
\end{split}
\]
(recall $k$ is the dimension of $\FibreM$). We take the coordinate change $r=\log y$ and $\lambda = (y+y')Y/2$. The volume form becomes
\[
\frac{2 e^{ r' +(r-r')/2}}{e^r + e^{r'}} dr' d\lambda d\zeta' d\eta.
\]
The phase:
\[
\Phi(r,\lambda, \zeta,\eta) = \langle \zeta - \zeta', \eta \rangle + 2\lambda \tanh\frac{r-r'}{2}
\]
The symbol under the integral giving $\Op(\sigma)f$ is now in the form
\[
\chi^{\Op}(r-r') \tilde{f}(r',\zeta')\sigma_{b,\ell}\left(r + \log\frac{1 + e^{r'-r}}{2}, \frac{\zeta + \zeta'}{2}, \lambda, \eta\right).
\]
where $\tilde{f} = \mathscr{P}_\ell f$. Since $\sigma_{b,\ell}$ does not depend on $r$, we deduce that
\[
\begin{split}
\{\Op(\sigma)\}_{b,\ell} \tilde{f}(r, \zeta) := & \int e^{\frac{i}{h}\Phi(r,\lambda, \zeta,\eta)}\chi^{\Op}(r-r') \tilde{f}(r',\zeta') \\
&\sigma_{b,\ell}\left(\frac{\zeta + \zeta'}{2}, \lambda, \eta\right)\frac{2 e^{ (r+r')/2}}{e^r + e^{r'}} \frac{dr' d\lambda d\zeta' d\eta}{(2\pi h)^{1 + k}}
\end{split}
\]
Provided that the support of the cutoff $\chi_C$ chosen after Equation \eqref{eq:def-free-Op} is slightly larger than the support of $\chi^{\Op}$, we can find a symbol $\widetilde{\sigma_{b,\ell}}\in S_b(\R\times \FibreL)$ such that 
\[
\{\Op(\sigma)\}_{b,\ell} = \Op^b( \widetilde{\sigma_{b,\ell}} )_C.
\]
We have proved
\begin{proposition}\label{prop:S_b_gives_Psi_b}
Let $\sigma\in S_b(M,L)$. Then, $\Op(\sigma)\in \Psi_{b,C}(M,L)$, where $C>1$ is a constant chosen in the construction of the quantization and there is a symbol $\widetilde \sigma_{b,\ell}\in S_b(M,L)$ such that $\{\Op(\sigma)\}_{b,\ell} = \Op^b( \widetilde{\sigma_{b,\ell}} )_C$. When the height $\mathsf{a}$ at which $\sigma$ starts being invariant varies, we can change the quantization and keep the same constant $C$.
\end{proposition}

\subsection{Meromorphic continuation of resolvents of admissible b-operators}

In this section, we will need the crucial compactness lemma:
\begin{lemma}\label{lemma:compact-injection-H^1}
Let $L\to M$ be a general admissible bundle as in Definition \ref{def:admissible-bundle}. Let $\mathsf{a}> \mathbf{a}$ and
\[
H^1_\mathsf{a}(M, L):= \left\{ f \in H^1(M, L)\ \middle|\ \mathscr{P}^\mathsf{a} f = 0 \right\}.
\]
This is a closed subspace of $H^1(M,L)$ and the injection $H^1_\mathsf{a}(M,L)\hookrightarrow L^2(M,L)$ is compact.
\end{lemma}

\begin{proof}
We can adapt the argument of Lax-Phillips \cite{Lax-Phillips-Scattering-67}. Let $\chi\in C_c^\infty(\R, [0,1])$ be equal to one in a neighbourhood of $0$ and
set $\chi_n(y):=\chi(y/n)$. Consider the multiplication operator $\chi_n(y)$ of in each cusp $Z_\ell$ of $M$. From Rellich's theorem, the multiplication operator $\chi_n$ is compact for all $n$ from $H^1_\mathsf{a}(M,L)$ to $L^2(M,L)$. Now, assume that as $n\to +\infty$, $\chi_n$ restricted to $H^1_\mathsf{a}(M,L)$ has the injection $H^1_\mathsf{a}(M,L)\hookrightarrow L^2(M,L)$ as norm limit. Then that injection has to be compact also. 

To show that it is a norm limit we have to show that for $f \in H^1_\mathsf{a}(M,L)$,
\[
\| f \|_{L^2(M,L),y>n} \leq C_n \|f \|_{H^1(M,L)},
\]
with a constant $C_n\to 0$ as $n\to +\infty$. We use the Poincar\'e inequality: consider a unimodular lattice $\Lambda\subset \R^d$ and $\mathbb{T}_\Lambda = \R^d/\Lambda$. For $\tilde{f}\in H^1( \mathbb{T}_\Lambda)$ with $\int \tilde{f}= 0$, we have
\[
\| \tilde{f} \|_{L^2}  \leq C_\Lambda \| \nabla \tilde{f} \|_{L^2}.
\]
Now, with $\kappa$ the number of cusps,
\begin{align*}
\| f \|_{L^2(M,L),y>n}^2 	&= \sum_{\ell = 1}^\kappa\int_{y>n} \frac{d y d \PointM}{y^{d+1}} \| f(y,\PointM,\cdot) \|^2_{L^2\big(\mathbb{T}_{\Lambda_\ell}\big)} \\
							& \leq C \sum_{\ell = 1}^\kappa\int_{y>n} \frac{d y d \PointM}{y^{d+1}} \| \partial_\theta f(y,\PointM,\cdot) \|^2_{L^2\big(\mathbb{T}_{\Lambda_\ell}\big)}\\
							&\leq C \frac{1}{n^2} \| f \|_{H^1(M,L)}^2.
\end{align*}
\end{proof}
We have a statement for general weights. 
\begin{defprop}\label{lemma:compact-injection-general}
We pick some smooth function $r'(r)$ equal to $\log\mathsf{a}$ for $r \leq \log C\mathsf{a}$, and equal to $r$ when $r> \log C^2\mathsf{a}$. Then, given $\escapeparam,N,\rho \in \R$ and the corresponding black box multiplication operator $\mathbf B(e^{\rho r'})$ from \eqref{eq:def-chi-multiplication}, we define
\[
H^{\escapeparam\mathbf{m}+N}_{\rho}(M,L) := \mathbf{B}(e^{\rho r'}) H^{\escapeparam\mathbf{m}+N}(M,L).
\]
Let $\rho < \rho'$. Then the injection $H^{\escapeparam\mathbf{m} + N + 1}_{\rho}(M,L) \hookrightarrow H^{\escapeparam\mathbf{m} + N}_{\rho'}(M,L)$ is compact.
\end{defprop}
\begin{proof}
From the choice of $r'$, and Pseudodifferential operator symbol calculus we can reduce directly to the case of $\escapeparam=0$ and $N=0$. Applying $\mathbf{B}(e^{(\rho -\rho')r'})$, we can also reduce to the case $\rho <0=\rho'$. Then we can adapt the argument from before, adding a contribution from the 0-th Fourier mode that decays as $e^{\rho \log n}\|f\|_{L^2}$.
\end{proof}

\begin{definition}\label{def:general_admissible_op}
Let $L\to M$ be a general admissible bundle. Let $\mathcal{X}$ be a derivation on sections of $L$ extending a vector field $X$ on $M$. 
Also assume that $\mathbf{X} : = h\mathcal{X} \in \Psi_b(M,L)$. Assume that the flow generated by $X$ is uniformly hyperbolic, and that we can construct escape functions $G\in S_b(M,L)$ for any $\delta>0$ satisfying the conclusions (i)-(iv) of Lemma~\ref{lemma:escape-function} as well as the invariance properties from Lemma~\ref{lemma:invariance-escape-function}. Then we say that $\mathbf{X}$ is a \emph{general admissible operator}. We denote by $E^{u,s}$ and $E^\ast_{u,s}$ the corresponding stable and unstable bundles.
\end{definition}
Given a general admissible operator the proof of Propositions \ref{prop:Inverse-up-to-smoothing-FS} and \ref{prop:wavefront-R_Q} apply, so we get a first parametrix $\mathscr{R}_Q(s):= (\mathbf{X} - Q - hs)^{-1}$ with norm $\mathcal{O}(h^{-1})$. Furthermore from Definition~\ref{def:b-operator} and Proposition\ref{prop:S_b_gives_Psi_b} we deduce straightforwardly:
\begin{proposition}
Let $L\to M$ be a general admissible bundle and let $\mathbf{X}$ a general admissible operator. Let $\delta>0$. Let $G$ be a corresponding escape function. Let $Q\in \Psi^{-\infty}_b(M,L)$ be microsupported in $|\xi|<3R\delta$, and elliptic in $|\xi|<2R\delta$ --- as the $Q$ used in Proposition \ref{prop:Inverse-up-to-smoothing-FS}. Then for each $\ell$, $(\mathbf{X}_{b,\ell}, G_{b,\ell}, Q_{b,\ell})$ is an admissible triple in the sense of Definition \ref{def:free-admissible-triple}. 
\end{proposition}
Consequently from Lemma~\ref{lemma:Inversion-up-to-smoothing-indicial-operator} and Theorem~\ref{thm:Continuation-Indicial-Resolvent} we deduce that $\IndicialRes_{Q,\ell}(s) := (\mathbf{X}_{b,l}  - Q_{b,\ell} -hs)^{-1}$ and that $\IndicialRes^{\mathbf{X}_{b,\ell}}(s)$ are analytic, respectively meromorphic families of operators on the appropriate anisotropic spaces. 
We now choose $\chi \in C^\infty(\R)$ such that $\chi(r) = 0$ for $r< \log(C\mathsf{a})$, and $\chi(r) = 1$ for $r> \log(C^2 \mathsf{a})$ and define
\[
\mathscr{R}_{Q}'(s): = \mathscr{R}_Q(s) + \sum_\ell \mathscr{E}_\ell \chi\left[ \IndicialRes^{\mathbf{X}_{b,\ell}}(s) - \IndicialRes_{Q,\ell}(s) \right] \chi \mathscr{P}_\ell.
\]
Next, let us define for $\tau\in\R$
\begin{equation}\label{eq:def-rho-max'}
\rho_{\max}(\tau) :=  \sup_\ell \rho_{\max,\ell}(\tau).
\end{equation}
Recall that $\rho_{\max,\ell}$ was defined in Equation \eqref{eq:def-rho-max}. Also keep in mind that weights are functions of the form $\mathbf{k}=\escapeparam \mathbf{m} + N$, and they are large when $\escapeparam$ is large and so is $\escapeparam/|N|$.
\begin{lemma}\label{lemma:parametrix-Fredholm-form}
Let $\tau<0$, and let $\mathbf{k}$ be sufficiently large. Then for $\Re (s)> \tau$ and $|\Im(s)|\leq h^{-1/2}$ the operator family $\mathscr{R}_{Q}'(s)$ is a meromorphic family of bounded operators from $H^{\mathbf{k}}_{-\rho_{\max}'(\tau)}$ to $H^{\mathbf{k}}_{\rho_{\max}'(\tau)}$. Additionally, we can write
\[
(\mathbf{X} - h s)\mathscr{R}_{Q}'(s) = \mathbb{1} + K(s),
\]
where $K(s)$ is a meromorphic family of compact bounded operators on $H^{\mathbf{k}}_{-\rho_{\max}'(\tau)}$. Additionally, $\mathbb{1}+K(s)$ is invertible for 
$\Re(s)$ large enough.
\end{lemma}

As a consequence, we get the main theorem of this article:
\begin{theorem}\label{thm:full-theorem-resolvent}
Let $\mathbf{X}=h\mathcal{X}$ be a general admissible operator (see Definition~\ref{def:general_admissible_op}) on a general admissible bundle $L\to M$ (see Definition~\ref{def:admissible-bundle}) and assume that the indicial roots are affine in the sense of Definition~\ref{def:roots-are-affine}. The Schwartz kernel of $\mathscr{R}(s):=(\mathcal{X}-s)^{-1}$ has a meromorphic continuation to $\C$. 
The corresponding poles are finite order, finite rank. We also have the wavefront set statements:
\begin{equation}\label{eq:wavefront-set-resolvent}
\WF'(\mathscr{R}(s))=\WF'_h(\mathscr{R}(s)) \cap T^\ast (M\times M)  \subset \Delta(T^\ast M) \cup \Omega_+ \cup E^\ast_s \times E^\ast_u.
\end{equation}
Furthermore, if $s_0$ is a pole and
\[
 \mathscr{R}(s) = \sum_{j=1}^J \frac{A_j}{(s-s_0)^j} + \mathscr R_H(s)
\]
is the Laurent expansion, with holomorphic part $
\mathscr R_H(s)$ then,
\begin{equation}\label{eq:wavefront-set-Laurent-expansion}
\WF'(A_j) \subset E^\ast_s \times E^\ast_u
\text{ and } \WF'(\mathscr{R}_H(s_0)) \subset 
\Delta(T^\ast M) \cup \Omega_+ \cup E^\ast_s \times E^\ast_u.
\end{equation}
\end{theorem}

\begin{proof}[Proof of Lemma \ref{lemma:parametrix-Fredholm-form}]
Recall that $K(s)= (\mathbf{X} - h s)\mathscr{R}_{Q}'(s) - \mathbb{1}$. The meromorphy of $\mathscr{R}'_Q(s)$ and $K$ has already been proved, and so has the invertibility for large $\Re s>0$ of $1+K(s)$. It suffices now to show that $K(s)$ is compact on the appropriate space. We will use the fact that if $\mathbf{k}$ is large, so it $\mathbf{k}\pm 1$.

The first observation is that from a standard resolvent identity, we have for 
$\ell = 1 \dots \kappa$,
\begin{equation}\label{eq:Resolvent-formula}
\IndicialRes^{\mathbf{X}_{b,\ell}}(s) - \IndicialRes_{Q,\ell}(s) = \IndicialRes^{\mathbf{X}_{b,\ell}}(s) Q_{b,\ell} \IndicialRes_{Q,\ell}(s).
\end{equation}
This is a bounded operator from $e^{-\rho_{\max,\ell}(\tau)\langle r \rangle} \mathcal{H}^{\mathbf{k}-1}_b$ to $e^{\rho_{\max,\ell}(\tau)\langle r \rangle} \mathcal{H}^{\mathbf{k}+1}_b$ (it is smoothing).

Now, we compute $(\mathbf{X}-hs)\mathscr{R}_{Q}'(s)$, and we find that the operator $K(s)=K_1(s)+K_2(s)$ writes as the sum of two terms. The first one is
\begin{equation}\label{eq:first-term-K(s)}
K_1(s):=\sum_{\ell} \mathscr{E}_\ell [\mathbf{X}_{b,\ell}, \chi] \IndicialRes^{\mathbf{X}_{b,\ell}}(s) Q_{b,\ell} \IndicialRes_{Q,\ell}(s) \chi \mathscr{P}_\ell.
\end{equation}
This operator is compact on $H^{\mathbf{k}}_{-\rho_{\max}(\tau)}$ since it maps it continuously to $\mathbb{1}_{y< C} H^{\mathbf{k}+1}$ --- here we are applying Theorem~\ref{thm:Continuation-Indicial-Resolvent} crucially.

The other term in $K(s)$ is
\begin{equation}\label{eq:second-term-K(s)}
K_2(s):=Q \mathscr{R}_Q(s) - \sum_{\ell} \mathscr{E}_\ell \chi Q_{b,\ell} \IndicialRes_{Q,\ell}(s) \chi \mathscr{P}_\ell.
\end{equation}
Applying $(\mathbf{X} - Q - hs )$ on the right, we obtain
\[
\underset{:=K_3}{\underbrace{Q - \sum_{\ell} \mathscr{E}_\ell \chi Q_{b,\ell} \chi \mathscr{P}_\ell}} + \underset{:=K_4}{\underbrace{\sum_\ell \mathscr{E}_\ell \chi Q_{b,\ell} \IndicialRes_{Q,\ell}(s) [\chi, \mathbf{X}_{b,\ell} - Q_{b,\ell}] \mathscr{P}_\ell}}.
\]
Since $\chi(y) = 1$ when $y> C^2 \mathsf{a}$, we get that
\[
\mathscr{P}^{C^3\mathsf{a}} \left[ Q - \sum_{\ell} \mathscr{E}_\ell \chi Q_{b,\ell} \chi \mathscr{P}_\ell \right] = 0
\]
Using that $Q$ is smoothing together with Lemma~\ref{lemma:compact-injection-H^1}, we deduce that $K_3(\mathbf{X} - Q - hs )^{-1}$ is a compact operator on $H^{\mathbf{k}}_{-\rho_{\max}'(\tau)}$. According to Lemma~\ref{lemma:Inversion-up-to-smoothing-indicial-operator}, provided $\mathbf{k}$ is large enough, $\IndicialRes_{Q,\ell}(s)$ is bounded on spaces $\mathcal H^{\mathbf{k_b}}_{b,\rho}$ with $\rho < -\rho_{\max}(\tau)$. Recall that $\chi$ was chosen to be constant outside a compact set so $[\chi, \mathbf{X}_{b,\ell} - Q_{b,\ell}]: \mathcal H^{\mathbf{k_b}}_{b,\rho_1} \to \mathcal H^{\mathbf{k_b}}_{b,\rho_2}$ is bounded for arbitrary $\rho_1,\rho_2$. We deduce that $K_4(\mathbf{X} - Q - hs )^{-1}$  maps $H^{\mathbf{k}}_{-\rho'_{\max}(\tau)}$ to $H^{\mathbf{k}+1}_{\rho}$ for some $\rho < -\rho_{\max}(\tau)$ and by Lemma~\ref{lemma:compact-injection-general} it is compact. This concludes the proof.
\end{proof}

\begin{proof}[Proof of Theorem \ref{thm:full-theorem-resolvent}]
From Lemma \ref{lemma:parametrix-Fredholm-form}, using the Gohberg Sigal theorem \cite{Gohberg-Sigal-70} --- see theorem C.7 in \cite{Dyatlov-Zworski-book} for a version in english --- we deduce that $\mathscr R'_Q(s)(1+K(s))^{-1}$ is a meromorphic
right inverse to $(\mathbf X-hs)$, bounded on $H^{\mathbf{k}}_{-\rho_{\max}(\tau)}\to H^{\mathbf{k}}_{\rho_{\max}(\tau)}$ for $\Re s > \tau$ and $\mathbf{k}$ large enough. As for $\Re(s)>0$, $\mathbf X - hs$ is invertible,
it has to coincide with the inverse there and we deduce it is a meromorphic continuation of $(\mathbf X-hs)^{-1}$. Since $C^\infty_c(M,L)$ is contained and dense in all spaces $H^{\escapeparam\mathbf{m} + N}_{\rho}$, we deduce the meromorphic extension of the Schwartz kernel. In particular, the poles do not depend on the choice of space. 

It remains to show the announced property on the wavefront set. We can use the arguments from \cite[page 18 of the ArXiV version]{Dyatlov-Zworski-16} again as in the end of the proof of theorem \ref{thm:Continuation-Indicial-Resolvent}. We reproduce the argument here. We have
by the second resolvent identity
\begin{equation}\label{eq:Resolvent-formula-twice-again}
\mathscr{R}(s) = h\mathscr{R}_Q(s) - h\mathscr{R}_Q(s)Q\mathscr{R}_Q(s) + \mathscr{R}_Q(s)Q \mathscr{R}(s) Q \mathscr{R}_Q(s).
\end{equation}
(one can check that all the terms in the equation are well defined). The wave front set of the first term in the RHS is contained in the announced wavefront set for $\mathscr{R}(s)$, so we concentrate in the second and third term. For both of them, their $\WF'_h\cap T^\ast(M\times M)$ is a subset of 
\[
\left\{ (x,\xi, x',\xi')\ \middle|\ \exists (x_1,\xi_1,x_1',\xi_1'),\ \begin{array}{ll} (x,\xi,x_1,\xi_1) & \in \WF'_h(\mathscr{R}_Q(s)Q),\\ (x_1',\xi_1',x',\xi')&\in \WF'_h(Q \mathscr{R}_Q(s))\end{array}\right\}.
\]
This is contained in $E^+_\delta\times E^-_\delta$, where
\[
E^\pm_\delta = \{ (x,\xi)\in T^\ast M\ |\ \exists\ T>0,\ |\Phi_{\pm T}(x,\xi)| \leq 3R\delta \}.
\]
Since the wavefront set of $\mathscr{R}(s)$ does not depend on $\delta$, we can let it go to $0$. The intersection of the $E^+_\delta \times E^-_\delta$ for $\delta \geq 0$, is exactly $E^\ast_s \times E^\ast_u$.

For the Wavefront set at a pole $s_0$ we consider 
\eqref{eq:Resolvent-formula-twice-again}. Comparing the Laurent coefficients we obtain 
\[
 A_J = \mathscr{R}_Q(s)QA_JQ\mathscr{R}_Q(s).
\]
We can apply the same argument as above and obtain
$\WF'(A_J) \subset E^\ast_s \times E^\ast_u$. For the other coefficients as well as $\mathscr R_H(s_0)$ we can argue inductively. Indeed, formula \eqref{eq:Resolvent-formula-twice-again} will provide us with a formula for the Laurent coefficients that will involve other
Laurent coefficients $A_j$ of higher order and derivatives of $\mathscr R_Q(s)$ in the $s$ parameter. But as 
$\partial_s \mathscr R_Q(s)= - \mathscr{R}_Q(s)^2$, the wavefront set of its derivatives is contained in the same set. \end{proof}

\section{Explicit computations for the geodesic flow}
\label{sec:explicit-computations}

In this section, we come back to the case of admissible bundles over $S^\ast N$ with $N$ an admissible cusp manifold. Let us denote by $A_{\max}$ the maximum of $\Re(\lambda)$ when $\lambda$ ranges in the eigenvalues of the endomorphisms $A_\ell$. Then we define
\[
 \rho_{\max,L}(\tau)=\max(0,A_{\max} - \tau - d/2)
\]
(Note that for functions i.e. $L$ being the trivial bundle, we have $A_{\max}=0$.
We prove 
\begin{theorem}\label{thm:admissible-bundles}
Let $N$ be an admissible cusp manifold, 
$L\to M= S^\ast N$ be an admissible bundle and
$\mathcal X$ an admissible lift of the geodesic flow vectorfield (see Definition~\ref{def:admissible_cusp_manifold} and \ref{def:admissible_vector_bundle}).
 
Then the resolvent $\mathscr R(s):=(\mathcal{X}-s)^{-1}$ which is definde on $L^2(M,L)$ for $\Re s\gg 0$ has a meromorphic continuation to $\C$ as a family of continuous operatos $\mathscr R(s):C^\infty_c(M,L)\to \mathcal{D}'(M,L)$. 

More precisely for any $\tau<0$, $N\in \R$ there is a suficiently large $\escapeparam$ such that on
$\Re(s)>\tau$, $|\Im(s)|\leq h^{1/2}$ the resolvent is a meromorphic family of bounded operators
\[
 \mathscr R(s): H^{\gamma \mathbf m+N}_{-\rho_{\max,L}(\tau)} \to  H^{\gamma \mathbf m+N}_{\rho_{\max,L}(\tau)}.
\]
Finally, the wavefront set of $\mathscr R(s)$ satisfies estimate \eqref{eq:wavefront-set-resolvent} and its polar part satisfies \eqref{eq:wavefront-set-Laurent-expansion} as in Theorem \ref{thm:full-theorem-resolvent}.
\end{theorem}

According to the proof of Theorem \ref{thm:full-theorem-resolvent} it suffices to show that the roots are affine in the sense of Definition~\ref{def:roots-are-affine}. This will be shown in Lemma~\ref{cor:indicial_roots_holom_vectorbundle}.

We will explicitly calculate the indicial 
roots for an admissible lift of the geodesic flow in the sense
of Definition~\ref{def:admissible_vector_bundle}. We do 
this in three steps: First we compute the family of 
indicial operators for admissible lifts. Then we determine the
indicial roots for the scalar case, and finally deduce the
precise formula for the indicial roots of an admissible 
vector bundle.

\subsection{The Indicial operator for admissible lifts}

From now on let $M=S^*N$ be the sphere bundle over an admissible
cusp manifold, $L\to M$ an admissible vector bundle and 
$\mathcal{X}$ an admissible lift in the sense of 
Definition~\ref{def:admissible_vector_bundle}. Set $\mathbf X
=h\mathcal X$ and fix a cusp $Z_\ell$. Then, as a first step towards the indicial
family, we want to calculate the b-operator $\mathbf X_{b,\ell}$
acting on sections of $\R\times \FibreL_\ell\to \R\times \FibreM$. Recall that in 
Example~\ref{exmpl:admissible_bundle_reduction} we have already determined
that $\FibreL_\ell = \groupK\times_{\tau_\ell}V_\ell \to \FibreM=
\groupK/\groupM=\Ss^d$. In order to give an explicit expression of the 
operator we use the coordinates $r\in \R$ and and spherical coordinates
$(\varphi, u)\in [0,\pi]\times \Ss^{d-1}$ on $\Ss^d$ as introduced in 
Section~\ref{sec:geodesic_flow_on_cusps}.
\begin{lemma}\label{lem:X_bell}
$\mathbf X \in \Psi_{b, 1}(M,L)$ is a cusp-b-operator and its associated
b-operator $\mathbf X_{b,\ell}$ in $\Psi_{b,1}(\R\times \FibreL_\ell)$ as
defined in Definition~\ref{def:b-operator} is 
given by
\begin{equation}\label{eq:X_b_ell}
\mathbf X_{b,\ell} = h\left[\cos(\varphi)\partial_r +\frac{d}{2}\cos(\varphi)+
\nabla^{(\ell)}_{\Xgr}+ A_\ell\right],
\end{equation}
where $\nabla^{(\ell)}$ is the canonical connection on $\FibreL_\ell$, $A_\ell \in
\textup{End}(V_\ell)^{\groupM}$  is given by 
Definition~\ref{def:admissible_vector_bundle} and acts as a zero-th order
operator on $\FibreL_\ell$ and $\Xgr = \sin\varphi\partial_\varphi$ is the vector field
of the gradient flow on $\Ss^d$.
\end{lemma}

\begin{proof}
 Let us fix a cusp $Z_\ell$ and consider a section $f\in C^\infty (S^*Z_\ell, L)$
 supported in $\{y>\mathsf{a}\}$. Recall that $L_{|S^*Z_\ell} = \Lambda_\ell\backslash 
 \groupG\times_{\tau_\ell}V_\ell$, thus we can identify $f$ with a function $\tilde f:\Lambda_\ell\backslash\groupG\to
 V_\ell$, that is
 right $\groupM$-equivariant, i.e. $\tilde f(\Lambda_\ell gm)= \tau_\ell(m^{-1})\tilde f(\Lambda_\ell g)$. Note that the geodesic flow on $S^*Z_{\ell,f}\cong \Lambda_\ell\backslash
 \groupG/\groupM$ is given by the right $\groupA$-action and we can write\footnote{
 The identification of the canonical connection on reductive homogenous spaces can be found
 in many geometry textbooks. For a short exposition in the context of geodesic flows on 
 vector bundles over locally symmetric spaces we refer to \cite[Section 1.1.5]{KW17}.}
 \begin{equation}\label{eq:action_of_X_NAK}
  (\mathbf X\tilde f)(\Lambda_\ell g) = h\left[\frac{d}{dt}_{|t=0}\tilde f(\Lambda_\ell g e^{Ht})
  +A_\ell\tilde f(\Lambda_\ell g)\right]
 \end{equation}
 for a suitably normalized $H\in \mathfrak a=\textup{Lie}(\groupA)$.
 Let us check that $\mathbf X$ preserves sections that are independent of the $\theta$ variable. Note that w.r.t. the $\groupN\groupA\groupK$ decomposition, this means that $\tilde f(\Lambda_\ell ng) = \tilde f(\Lambda_\ell g)$ (cf. Section~\ref{sec:admissible_vb}). That such functions are preserved under $\mathbf X$ is obvious by \eqref{eq:action_of_X_NAK}.
 Consequently $\mathbf X$ is a black-box operator according to Definition~\ref{def:b-operator}.

Let us thus remove the dependencies in $\theta\in \Lambda_\ell\backslash\groupN$ and 
consider the operator $\mathbf X_{b,\ell}^0$ acting on sections  $f\in C^\infty(\R\times\FibreM,\R\times\FibreL_\ell)$. Further identify these sections with right $\groupM$-invariant functions
$\tilde f:\groupA\times\groupK\to V_\ell$. By the $\groupN\groupA\groupK$-Iwasawa decomposition we can write any $g\in \groupG$ in a unique way as $g=n_{NAK}(g)a_{NAK}(g)k_{NAK}(g)$. With this notation we can write
\[
\begin{split}
 \mathbf{X}^0_{b,\ell}\tilde f(a,k) &= h\left[\frac{d}{dt}_{|t=0} \tilde f (a_{NAK}(ake^{Ht}), k_{NAK}(ake^{Ht}))+A_\ell\tilde f(a,k)\right]\\
 &=h\left[\frac{d}{dt}_{|t=0} \tilde f (aa_{NAK}(ke^{Ht}), k_{NAK}(ke^{Ht}))+A_\ell\tilde f(a,k)\right]\\
 \end{split}
\]
where we used the identities $a_{NAK}(ag)=a\cdot a_{NAK}(g)$ and $k_{NAK}(ag)=k_{NAK}(g)$.
This formula shows directly that $\mathbf{X}^0_{b,\ell}$ commutes with translations in the $\groupA$ direction
and we have thus shown that $\mathbf X$ is a cusp-b-operator according to Definition~\ref{def:b-operator}. It finally remains to express $\mathbf{X}_{b,\ell}^0$ in the coordinates $r,\varphi,u$ of $\R\times\Ss^d
\cong\groupA\times\groupK/\groupM$ as introduced above. In particular we have to identify the 
differential operators 
\[
\frac{d}{dt}_{|t=0}aa_{NAK}(ke^{Ht}) \text{ on } \groupA\cong\R,\ \text{ and } \frac{d}{dt}_{|t=0}k_{NAK}(ke^{Ht}) \text{ on } \groupK/\groupM\cong\Ss^d.
\]
As these differential operators
are independent of the choice of the vector bundle we can simply restrict to the scalar case 
and compare to the expression of the geodesic flow vector field in coordinates that have been
calculated in Example~\ref{exmpl:b-op-of_geodesic_flow} (cf. also Equation \eqref{eq:def-X_k}). This yields $\frac{d}{dt}_{|t=0}aa_{NAK}(ke^{Ht})\cong \cos\varphi\partial_r$ and $\frac{d}{dt}_{|t=0}k_{NAK}(ke^{Ht}) \cong\sin\varphi\partial_\varphi =\Xgr$. Taking into account the definition of the canonical connection 
on $\FibreL_\ell=\groupK\times_{\tau_\ell}V_\ell$ we get
\[
 \mathbf X_{b,\ell}^0 = h\left[\cos\varphi\partial_r + \nabla^{\FibreL_\ell}_{\Xgr} + A_\ell \right]
\]
In order to pass from $\mathbf X_{b,\ell}^0$ to $\mathbf X_{b\ell}$  one simply has to conjugate the differential 
operator by $e^{-rd/2}$ which creates the additional $d/2\cos\varphi$ term in (\ref{eq:X_b_ell})
\end{proof}

Now from Equation \eqref{eq:X_b_ell} and the definition~\ref{def:indicial_family} of the indicial family
we directly obtain:
\begin{corollary}
For $\mathbf X_{b\ell}$ as in Lemma~\ref{lem:X_bell} one has
\begin{equation}\label{eq:indicical_op_for_admissible_lift}
 I(\mathbf X_{b,\ell},\lambda) = \lambda\cos\varphi + h\left[\frac{d}{2}\cos\varphi + \nabla^{\FibreL_\ell}_{\Xgr} + A_\ell\right].
\end{equation}
\end{corollary}

\subsection{Finding the indicial roots for functions}

In this section, we focus on the action on functions. In that case, $\mathcal{X}=X$ and $\mathbf{X}=hX$. Since the flow is the same for each cusp, we can safely drop the dependence in the index $\ell$. We compute the indicial roots of $I(\mathbf{X}_b,\lambda) - hs$.
As this operator will frequently show up in the sequel we introduce the shorter notation 
\begin{equation}\label{eq:def-P-lambda}
P_\lambda := I(\mathbf{X}_b,\lambda)=h\sin\varphi\partial_\varphi +[\lambda +hd/2]\cos\varphi.
\end{equation}

Le us introduce some notations which we will need to formulate the spectral properties of $P_\lambda$.
Recall that we have introduced the coordinates $(\varphi, u)\in [0,\pi]\times \Ss^{d-1}$ on 
$\Ss^d$. Consider the projection of $\Ss^d$ to the equatorial plane. It is a smooth chart on both strict hemispheres. We denote these smooth restrictions by
\begin{align}
 \kappa_{\mathcal N}:\{(\varphi,u) \in \Ss^d|\varphi<\pi/2\}&\to \{x\in\R^d:\|x\|<1\}\\
\intertext{ and }
  \kappa_{\mathcal S}:\{(\varphi,u) \in \Ss^d|\varphi>\pi/2\}&\to \{x\in\R^d:\|x\|<1\}
\end{align}
Note that $(\rho,u):=(\sin\varphi, u)\in [0,1]\times \Ss^{d-1}$ 
are exactly the radial coordinates in both charts.

For further reference we recall that the Taylor expansion in radial coordinates at $0$ for $f\in C^n(\R^d)$ can be written in the following fashion
\begin{equation}\label{eq:taylor_radial}
 f(\rho, u) = \sum_{|\mu|\leq n}\frac{\partial_x^\mu f(0)}{\mu!}\cdot \rho^{|\mu|}\cdot \Upsilon_\mu(u) + o(\rho^n) \textup{, as }\rho\to 0.
\end{equation}
Here $\mu\in \N^d$ is a multindex, $\Upsilon_\mu\in C^\infty(\Ss^{d-1})$ is the monomial 
$x^\mu$, $x\in \R^d$, of degree $|\mu|$ restricted to the unit sphere $\Ss^{d-1}\subset \R^d$.

Let us come back to $P_\lambda$. According to Lemma \ref{lemma:equivalence-indicial-spaces}, to determine the indicial roots, it suffices to consider the action of $P_\lambda$ on $\mathsf{H}^{\escapeparam m_b}_0(\Ss^d)$, that we denote just $\mathsf{H}^{\escapeparam m_b}(\Ss^d)$. Inspecting the Formula \eqref{eq:def-P-lambda}, we see that is a gradient vector field plus a complex potential. Pollicott-Ruelle resonances for Morse-Smale gradient flows were studied in detail by Dang and Rivière \cite{Dang-Riviere-16}. In particular, the spaces they defined are quite similar to $\mathsf{H}^{\escapeparam m_b}(\Ss^d)$. Recall that $C_G$ was defined in Lemma \ref{lemma:escape-function}
\begin{lemma}\label{lemma:Identifying-the-indicial-space}
There is an $\epsilon>0$ such that the following holds.
\begin{itemize}
	\item Let $f\in \mathcal{D}'(\Ss^d)$ be supported in the $\epsilon$-neighbourhood of the North Pole. Then $f\in \mathsf{H}^{\escapeparam m_b}(\Ss^d)$ if and only if $f \in H^{-C_G \escapeparam}(\Ss^d)$.
	\item Let $f\in \mathcal{D}'(\Ss^d)$ be supported in the $\epsilon$-neighbourhood of the South Pole. Then $f\in \mathsf{H}^{\escapeparam m_b}(\Ss^d)$ if and only if $f \in H^{C_G \escapeparam}(\Ss^d)$.
\end{itemize}
\end{lemma}

\begin{proof}
Let us prove the first assertion: By Definition~\ref{def:anisotropic_indicial_spaces} of $\mathsf{H}^{\escapeparam m_b}$ and standard Pseudodifferential operator arguments it is enough to prove that 
$m_{b,0}\in S^0(\FibreL)$ is constant equal to $-C_G$ in a neighbourhood of $\mathcal N$ modulo some lower order terms $S^{-1+\varepsilon}(\FibreL)$. By Lemma~\ref{lemma:Indicial-symbol} and Proposition~\ref{prop:S_b_gives_Psi_b} the leading term is given by $m_{b,0}(\zeta,\eta) = m(r,\theta, \zeta, \lambda=0,J=0 \eta) \mod S^{-1+\varepsilon}(\FibreL)$ where $m\in S^0(M, L)$ is the order function constructed in Lemma~\ref{lemma:escape-function}. By construction of $m$ we know that high enough in the cusp $m = -C_G$ in a neighbourhood of $E^*_{c,s} := (E^s_c\oplus E^0_c)^\perp$. Here $E^s_c$ and $E^0_c$ are the stable and neutral bundles corresponding to constant curvature (see discussions in the proof of Lemma~\ref{lemma:invariance-escape-function}). Now if we consider some point $(r, \theta, \mathcal N)\in SZ$ and write $T(r, \theta, \mathcal N)(SZ) \cong  T_r\R\oplus T_\theta(\R^d/\Lambda)\oplus T_{\mathcal N}\mathbb S^d$ then we have by standard hyperbolic geometry $(E^0_c)_{(r, \theta, \mathcal N)} = T_r\R$ and $(E^s_c)= T_\theta(\R^d/\Lambda)$. Consequently $E^*_{c,s}$ is precisely given by $\lambda=0, J=0$. Putting everything together we know that at leading order $m_{b,0}$ is constant equal to $-C_G$ around $\mathcal N$ which implies the first assertion. 

The second statement follows from similar arguments. 
\end{proof}

The following lemma shows that in the charts $\kappa_{\mathcal{N},\mathcal{S}}$, the operator $P_\lambda$ takes a 
particularly simple form --- recall that here $\rho = \sin\varphi$
\begin{lemma}\label{lem:conjugate_P_lambda}
On the Northern Hemisphere, the function $2\tan(\varphi/2)/\sin\varphi$ is an analytic, 
 nonzero function and expressed in the $(\rho,u)$-charts, defined above we have
 \[
\begin{split}
  &\left(\frac{2\tan(\varphi/2)}{\sin\varphi}\right)^{-s}(P_\lambda - hs)\left(\frac{2\tan(\varphi/2)}{\sin\varphi}\right)^{s} \\
   &\qquad\qquad\qquad\qquad= \sqrt{1-\rho^2}\left(h\rho\partial_\rho -hs+\lambda+ h\frac{d}{2}\right).
\end{split}
 \]
 On the Southern Hemisphere, the function $2\tan(\varphi/2)\sin\varphi$ is an analytic, 
 nonzero function and expressed in the $(\rho,u)$-charts, defined above we have
 \[
\begin{split}
  &\left(2\tan(\varphi/2)\sin\varphi\right)^{-s}(P_\lambda - hs)\left(2\tan(\varphi/2)\sin\varphi\right)^{s} \\
  &\qquad\qquad\qquad\qquad= \sqrt{1-\rho^2}\left(-h\rho\partial_\rho -hs-\lambda-h\frac{d}{2}\right).
\end{split}
 \]
\end{lemma}

\begin{proof}
 The results follows from a straightforward calculation using standard trigonometric identities.
\end{proof}
Let $\mu\in \N^d$ be a multindex, then the we define the standard dirac distributions on $\R^d$
by $\delta_0^{(\mu)}: C_c^\infty(\R^d)\ni f\mapsto (\partial_x^\mu f)(0)$. Recall that 
all distributions on $\R^d$, supported in $0$ are linear combinations of finitely many 
$\delta_0^{(\mu)}$. Furthermore
\begin{equation}\label{eq:dirac_zero_ev}
\rho\partial_\rho\delta_0^{(\mu)} = -(|\mu|+d)\delta_0^{(\mu)} .
\end{equation}
If $\kappa_{\mathcal N}$ is the chart of the Northern Hemisphere, then we define for any $\lambda\in\C$ 
the distribution
\[
 \delta_{\mathcal N,\lambda}^{(\mu)}:=\left(\frac{2\tan(\varphi/2)}{\sin\varphi}\right)^{\lambda/h - (|\mu|+d/2)}
 \kappa_{\mathcal N}^* \delta_0^{(\mu)} \in\mathcal D'(\Ss^d),
\]
and for $\kappa_{\mathcal S}$ the chart of the Southern Hemisphere, we define the distribution
\[
 \delta_{\mathcal S,\lambda}^{(\mu)}:=\left(2\tan(\varphi/2)\sin\varphi\right)^
 {-\lambda/h + (|\mu|+d/2)}
 \kappa_{\mathcal S}^* \delta_0^{(\mu)} \in\mathcal D'(\Ss^d).
\]
Combining \eqref{eq:dirac_zero_ev} with Lemma~\ref{lem:conjugate_P_lambda} we obtain
\begin{equation}\label{eq:delta_eigenvalues}
\begin{split}
 [P_\lambda - (\lambda - h(|\mu| +d/2))]&\delta_{\mathcal N,\lambda}^{(\mu)}=0 \\
 [P_\lambda + (\lambda - h(|\mu| +d/2))]&\delta_{\mathcal S,\lambda}^{(\mu)}=0
\end{split}
\end{equation}
and up to linear combinations these are the only eigendistributions of $P_\lambda$ supported
in the North or South Pole. 

We next want to study the kernels of the operators $P_\lambda$ on $\mathsf{H}^{\escapeparam m_b}(\Ss^d)$. According to 
Proposition~\ref{prop:constant-domain-order1-operators}, each $P_\lambda$ has a unique closed extension, and the domain 
$D^{\escapeparam m_b}(\Ss^d)$ does not depend on $\lambda$, so that $\lambda\mapsto P_\lambda$ is a type (A) family as in 
Kato \cite{Kato-80}. Further, to prove Proposition \ref{prop:meromorphic indicial_resolvent}, we proved that $P_\lambda - hs$ is Fredholm of index $0$ when $\Re(s)>-\escapeparam C_G +d/2 + |\Re(\lambda/h)|$. The rest of this section is devoted to the proof of
\begin{proposition}\label{prop:indicial_roots}
The indicial roots of $\mathbf X_b$ acting on functions are affine and they are given by 
\[
 \specb(s) = \{\pm (s+(d/2 + n)), n\in \mathbf N\}.
\]
\end{proposition}

We start with
\begin{lemma}\label{lem:P_lambda_finite_kernel}
 Let $\escapeparam>0$, $\lambda\in\C$, then for $\Re(s)>-\escapeparam C_G +d/2 + |\Re(\lambda/h)|$ the operator
 $(P_\lambda -hs):D^{\escapeparam m_b}(\Ss^d)\to \mathsf{H}^{\escapeparam m_b}(\Ss^d)$ is injective unless
\begin{equation}\label{eq:scalar-indicial-roots}
\lambda = \pm h[s + (d/2+n)].
\end{equation}
for some $n \in \N$.
\end{lemma}

\begin{proof}
 Assume that $w\in \mathcal D'(\Ss^d)$ is a distribution that fulfills $(P_\lambda-hs)w=0$.  Then 
 we can distinguish two cases: Either $w_{|\Ss^d\setminus\{\mathcal N, \mathcal S\}} =0$
 or not. 
 
\textbf{In the first case}, $w$ must be a linear combination of 
$\delta_{\mathcal N/\mathcal S}^{(\mu)}$ and from \eqref{eq:delta_eigenvalues} 
we deduce that the possible solutions are either $\lambda = h(s + n+d/2)$ and $w$ is a 
linear combination of $\delta_{\mathcal N,\lambda}^{(\mu)}$ with $|\mu|=n$ or  $\lambda = - h(s + n+d/2)$
and $w$ is a linear combination of $\delta_{\mathcal S,\lambda}^{(\mu)}$, again with $|\mu|=n$. 
 
Whether these distributional solutions belong to $\mathsf{H}^{\escapeparam m_b}(\Ss^d)$, or not,
depends on $\escapeparam$. Suppose that $\escapeparam C_G > n+d/2$, $n\in\N$, then locally around the South Pole, according to Lemma \ref{lemma:Identifying-the-indicial-space},
distributions in $\mathsf{H}^{\escapeparam m_b}(\Ss^d)$ have to be of positive Sobolev order, so none of the Dirac distributions 
$\delta_{\mathcal S,\lambda}^{(\mu)}$ is allowed. Near the North Pole, distributions are 
allowed to be in $H^{-n-d/2-\varepsilon}(\Ss^d)$, again from Lemma \ref{lemma:Identifying-the-indicial-space}
and consequently, all the distributions $\delta_{\mathcal N,\lambda}^{(\mu)}$ with $|\mu|\leq n$ 
are contained in $\mathsf{H}^{\escapeparam m_b}(\Ss^d)$.

\textbf{In the second case}, i.e. $w_{|\Ss^d\setminus\{\mathcal N, \mathcal S\}} \neq0$ 
we work in $\Ss^d\setminus\{\mathcal N,\mathcal S\}$  with the coordinates $(\varphi, u)\in]0,\pi[\times \Ss^{d-1}$. As $P_\lambda$ is independent of 
$u$ we can choose a product form $w_{|\Ss^d\setminus\{\mathcal N,\mathcal S\}} = f\otimes g$ with
$f\in\mathcal D'(]0,\pi[)$ and $g\in\mathcal D'(\Ss^{d-1})$. Thus the 
PDE reduces to the (ordinary) differential equation $(P_\lambda-hs)f=0$. By
ellipticity, $f$ has to be a smooth function on $]0,\pi[$ and for every $\lambda, s$ there is
a unique solution 
\[
f(\varphi) := (\sin\varphi)^{-d/2-\lambda/h}(2\tan(\varphi/2))^{s}.
\] 
 
 We now have to discuss under what conditions $f\otimes g$ can be extended to a distribution in 
 $\mathsf{H}^{\escapeparam m_b}(\Ss^d)$. For $d/2>\epsilon>0$, let $\alpha=\escapeparam C_G -d/2 - \epsilon$, then from Lemma \ref{lemma:Identifying-the-indicial-space} and 
 the Sobolev embedding theorem, distributions have to be $C^\alpha(\Ss^d)$ in a neighbourhood 
 of the South Pole. Going to the charts $\kappa_{\mathcal S}$ we obtain 
 \[
 \left(2\tan(\varphi/2)\sin\varphi\right)^{-s}w=  \rho^{-d/2-\lambda/h - s}\otimes g.
 \]
As in a neighbourhood of the South Pole, $\left(2\tan(\varphi/2)\sin\varphi\right)^{-s}$
is a smooth nonvanishing $C^\infty(\Ss^d)$ function we have to extend 
$\rho^{-d/2-\lambda/h - s}\otimes g$ to a $C^\alpha$-function on $\R^d$. According to \eqref{eq:taylor_radial}
 this is possible if either $\Re(-d/2-\lambda/h - s) \geq \alpha$ or if
$-d/2-\lambda/h - s = n$ for some $n\in \N$ and $g$ is a linear combination of $\Upsilon_\mu$
with $|\mu|=n$ (or in other words, $g$ is a homogeneous polynomial of degree $n$). Note that the 
first case is ruled out since we assumed that $-\Re( s + \lambda/h) < - d/2 + C_G\escapeparam$ so that we would have $\alpha <  C_G \escapeparam - d$, and $\epsilon>d/2$, contrary to our assumption.
\end{proof}

Now, to complete the proof of Proposition \ref{prop:indicial_roots}, we have to check that for $\lambda = \pm h(s + d/2 +n)$, the kernel is not empty. We have already done this in the proof for $\lambda = h(s+d/2+n)$ for $n\in\N$, so we concentrate on the case that $\lambda = - h(s+d/2+n)$.

The question is whether the functions $f \otimes \Upsilon_\mu$, $|\mu|=n$, can be extended to distributions over the whole sphere $\Ss^d$. We have $hs  =-\lambda -h (n+d/2)$. Let $\Upsilon\in \R_k(\mathbb S^{d-1})$ where $\R_k(\mathbb S^{d-1})$ denotes the space of homogeneous 
polynomial of degree $n$ restricted to the unit sphere. We introduce the notation 
\begin{equation}\label{eq:f_S_def}
f^0_{\mathcal S,\Upsilon,\lambda} = (\sin\varphi)^{-d/2-\lambda/h}(2\tan(\varphi/2))^{-(n+d/2+\lambda/h)}\otimes \Upsilon(u).
\end{equation}
If we express these functions in the $(\rho,u)$ coordinates in a neighbourhood of the North Pole
we get
\[
 \left(\frac{2\tan(\varphi/2)}{\sin\varphi}\right)^{n+d/2+\lambda/h} f^0_{\mathcal S,\Upsilon,\lambda}= \rho^{-d-2\lambda/h-n} \Upsilon(u).
\]
Since $2\tan(\varphi/2)/\sin\varphi$ is a nonvanishing $C^\infty(\Ss^d)$ 
function near $\mathcal N$, we conclude that $f^0_{\mathcal S, \Upsilon,\lambda}$ is in $L^1_{loc}(\Ss^d)$ and a legitimate distribution whenever $n+2\Re\lambda/h< 0$. Now, using the ideas of Hadamard regularization as in \cite[Theorem 3.2.4]{Hormander-1}, we can show that $f^0_{\mathcal{S}, \Upsilon, \lambda}$ extends from $\Re \lambda < - hn/2$ to the whole of $\C$ as meromorphic family of distributions $F_{\mathcal{S},\Upsilon, \lambda}$. When $\psi \in C^\infty(\Ss^d)$ is not supported around $\mathcal{N}$, the value of $F_{\mathcal{S},\Upsilon, \lambda}(\psi)$ is given by $f^0_{\mathcal{S}, \Upsilon, \lambda}(\psi)$, so we can concentrate on the case of $\psi$ supported around $\mathcal{N}$, and consider a smooth function $\psi$ supported in $\{\rho<\epsilon\}$ in $\R^d$ such that when $\Re \lambda < - hn/2$, 
\[
\left(\frac{2\tan(\varphi/2)}{\sin\varphi}\right)^{n+d/2+\lambda/h}f^0_{\mathcal{S}, \Upsilon, \lambda}(\psi) = \int_{0}^\epsilon\int_{\mathbb S^{d-1}} \rho^{-2\lambda/h - n -1} {\psi}(\rho u)\Upsilon(u)dud\rho.
\]
Integrating by parts in the $\rho$ variable, $N$ times when $\Re \lambda < - hn/2$, we obtain that
\[
\rho^{-2\lambda/h - n -1}\Upsilon(\psi) = \prod_{j=0}^{N-1} \frac{1}{2\lambda /h  +n - j}\int_{0}^\epsilon \int_{\mathbb S^{d-1}} \rho^{-2\lambda/h - n + N -1} \Upsilon(u) (\partial_{\rho}^{N}{\psi}(\rho u))dud\rho.
\]
The expression in the RHS is obviously meromorphic for $\Re \lambda < h(N-n)/2$. The poles are situated at $\lambda = h(j-n)/2$, with $j= 0,\dots, N-1$, and they are of order $1$. At such a point, we find $s= - (n+d+j)/2$, and $\lambda = h( s + d/2 + j )$. In other words, the poles of $F_{\mathcal{S},\Upsilon,\lambda}$ correspond to root crossings. This is sufficient to ensure that the indicial roots are exactly the $\pm h(s+d/2+n)$ with $n\in\N$, and finishes the proof of Proposition \ref{prop:indicial_roots}. \qed

Now, while not necessary for the proof of the main theorem, we want here to describe the Jordan block structure at the root crossings. Since the residue of $F_{\mathcal{S},\Upsilon,\lambda}$ at a pole does not depend on the level of regularization $N$ (as long as $N\geq j+1$), we can choose $N = j+1$. Then the residue is given by
\[
\frac{ h }{ 2 j !}\int_{0}^\epsilon \int_{\mathbb S^{d-1}} \Upsilon(u) (\partial_{\rho}^{j+1}{\psi}(\rho u)) d\rho du.
\]
But as $\psi$ is supported in $\{\rho<\epsilon\}$ this is just
\[
\frac{ h }{ 2 j !} \int_{\Ss^{d-1}} \Upsilon(u) \left(\left(\frac{\partial}{\partial\rho}\right)^{j}_{|\rho=0}{\psi}(\rho u)\right) du = \frac{h}{2}\sum_{|\mu|=j}  \frac{1}{\mu!} \left(\int_{\mathbb S^d} \Upsilon_\mu(u)\Upsilon(u)du \right)\delta_0^{(\mu)}(\psi).
\]
where the equality can be read off \eqref{eq:taylor_radial}. Writing $a_\mu = 1/\mu! \int_{\Ss^{d-1}} \Upsilon(u) \Upsilon_\mu(u) du$, the residue of $F_{\mathcal{S},\Upsilon, \lambda}$ at $\lambda=h(j-n)/2$ is thus given by 
\[
\frac{h}{2}\sum_{|\mu|=j}  a_\mu \delta_{\mathcal N,\lambda}^{(\mu)}.
\] 
The finite part of $F_{\mathcal{S},\Upsilon, \lambda}$ at such a point is a distribution $A_j$ such that $A_j$ coincides with $f_{\mathcal{S},\Upsilon, h(j-n)/2}$ in $\Ss^d\setminus{\mathcal{N}}$. Additionally, since we have for all $\lambda$
\[
(P_\lambda +\lambda + h(n+d/2) ) F_{\mathcal{S},\Upsilon,\lambda} = 0,
\]
Differentiating in the parameter $\lambda$, we deduce that 
\[
\Big(P_{h(j-n)/2} + h \frac{n+j+d}{2}\Big)A_j =- (1+\cos \varphi)\frac{ h }{ 2}\sum_{|\mu|=j} a_\mu \delta^{(\mu)}_{\mathcal{N},h(j-n)/2}.
\]
Consequently the finite part of $F_{\mathcal S,\Upsilon,\lambda}$ is an eigendistribution of $P_\lambda$ if all the $a_\mu$ vanish. If $\Upsilon$ is chosen such that this is not the case, we can however modify $A_j$ in order to get a generalized eigendistribution: Consider $E_j$ the space of distributions supported in $\{\mathcal{N}\}$, of order $< j$. Choosing a basis of such distributions of decreasing order, we find that $P_{h(j-n)/2}$ acts on $E_j$ in an upper triangular fashion, and the diagonal coefficients are non singular. We deduce that $P_{h(j-n)/2}$ is invertible on $E_j$. In particular, since $(P_{h(j-n)/2} + h \frac{n+j+d}{2})^2 A_j \in E_j$, we can find $e_j\in E_j$ such that $(P_{h(j-n)/2} + h \frac{n+j+d}{2})^2 (A_j+ e_j) = 0$. In particular, the kernel is non empty, and there is an order 2 Jordan block.

\begin{definition}\label{def:fSUlambda}
When $\lambda \neq h(j-n)/2$ with $j$, $n$ some integers, and $\Upsilon \in \R_n(\Ss^d)$, we denote by $f_{\mathcal{S}, \Upsilon,\lambda}$ the continuation $F_{\mathcal{S},\Upsilon, \lambda}$. When $\lambda = h(j-n)/2$, $f_{\mathcal{S},\Upsilon, \lambda}$ will instead refer to the distribution $A_j+ e_j$ thus defined.
\end{definition}

Before we proceed, it will be useful to introduce some notations. Given $g,f\in C^\infty(\Ss^{d-1})$ real valued, we let
\[
\langle f, g \rangle = \int_{\Ss^{d-1}} f g.
\]
We recall that $\R_n(\Ss^{d-1})$ is the set of functions on the sphere that are restrictions of real homogeneous polynomials of order $n$ on $\R^{d}$.

As a consequence of the proof of Lemma \ref{lem:P_lambda_finite_kernel}, we get the following explicit description of the
generalized eigenstates of $P_\lambda$:

\begin{lemma}
Let $\lambda\in\C$, $n\in \N$ and $\escapeparam C_G>d+n+2|\Re(\lambda/h)|$,
consider the operator $P_\lambda :D^{\escapeparam m_b}(\Ss^d)\to \mathsf{H}^{\escapeparam m_b}(\Ss^d)$, and the kernel of $P_\lambda + h(d/2+n) \pm \lambda$:
 \begin{itemize}
	\item If $\lambda \notin h\mathbb Z/2$, then for any $n\in \N$, there are no Jordan Blocks, and
\begin{align*}
&\ker(P_\lambda+h(d/2+n)+\lambda) = \Span\left\{ f_{\mathcal S, \Upsilon_\mu,\lambda}\ \middle|\  |\mu| = n \right\}\\
&\ker(P_\lambda+h(d/2+n)-\lambda) = \Span\left\{\delta^{(\mu)}_{\mathcal N,\lambda}\ \middle|\ |\mu|=n\right\}
\end{align*}
	\item If $\lambda= h k/2$, $k\in\N$, and $n=0,\ldots k-1$, there are no Jordan Blocks for $P_\lambda + h(d/2+n)-\lambda$, and one has
\begin{equation*}
\ker(P_\lambda+ h(d+2n-k)/2) = \Span\left\{\delta^ {(\mu)}_{\mathcal N, \lambda}\ \middle|\  |\mu|= n\right\}.
\end{equation*}
	\item For $\lambda = h k/2,k\in \N$ and $n=k,k+1,\ldots$ one has Jordan Blocks of index $2$ for $P_\lambda + h(d/2+n)-\lambda$, and
\begin{align*}
\ker(P_\lambda+h(d+2n-k)/2)^2 	&= \Span\left\{\delta_{\mathcal N,\lambda}^{(\mu)}, f_{\mathcal S, \Upsilon_\nu,\lambda}\ \middle|\ |\mu|=n, |\nu|=n-k \right\}.\\
\ker(P_\lambda+h(d+2n-k)/2)		&= \Span \Big\{\delta_{\mathcal N,\lambda}^{(\mu)}\ \Big|\ |\mu|=n\Big\}  \\
					&\cup \Big\{ f_{\mathcal S,\Upsilon, \lambda}\ \Big|\ \Upsilon\in \R_{n-k}(\Ss^{d-1}),  \langle \Upsilon, \Upsilon_\nu\rangle = 0,\ \forall |\nu| = n\Big\}.
\end{align*}
	\item If $\lambda= -h k/2$, $k\in\N$, and $n=0,\ldots k-1$, there are no Jordan Blocks for $P_\lambda + h(d/2+n) + \lambda$, and one has
\begin{equation*}
\ker(P_\lambda+ h(d + 2n-k)/2) = \Span\left\{f_{\mathcal S,\Upsilon_\mu,\lambda}\ \middle|\ |\mu|=n\right\}.
\end{equation*}
\item For $\lambda = -h k/2$, $k\in\N$ and $n=k,k+1,\ldots$ one has Jordan Blocks of index $2$ for $P_\lambda + h(d/2+n)+\lambda$, and
\begin{align*}
\ker(P_\lambda+h(d+2n-k)/2)^2 &= \Span\left\{f_{\mathcal S, \Upsilon_\mu,\lambda},\delta_{\mathcal N,\lambda}^{(\nu)}\ \middle|\ |\mu|=n, |\nu|=n-k\right\}.\\
   \ker(P_\lambda+h(d+2n-k)/2)&=\Span \Big\{\delta_{\mathcal{N},\lambda}^{(\nu)}\ \Big| \  |\nu|=n-k\Big\}\\
					\cup & \Big\{f_{\mathcal S, \Upsilon,\lambda}\ \Big|\ \Upsilon\in \R_{n}(\Ss^{d-1}),\ \langle \Upsilon, \Upsilon_\mu\rangle = 0,\ \forall |\mu| = n-k\Big\}.
\end{align*}
\end{itemize}

\end{lemma}

\subsection{Indicial roots for fiber bundles}

After this study of the action on functions, we come back to the 
action on admissible vector bundles 
$\FibreL=\groupK\times_{\tau_\ell}V_\ell \to \groupK/\groupM\cong 
\Ss^d$. For the moment let us fix a cusp and drop the index $\ell$.
Note that Definition~\ref{def:admissible_vector_bundle} does 
not assume that $\tau$ is an
irreducible $\groupM$ representation. However, we can reduce the 
problem to the irreducible case: Consider the complexified representaion
$(\tau, V_\C)$ which decomposes into irreducible unitary representations 
$(\sigma_i,W_i)$. We then get
\begin{equation}\label{eq:irreducible splitting}
L^2(\Ss^d, \groupK\times_\tau V) = \bigoplus_{i=1}^n
L^2(\Ss^d, \groupK\times_{\sigma_i} W_i).
\end{equation}
Furthermore, using the explicit form of $I(\mathbf{X}_b,\lambda)$
from \eqref{eq:indicical_op_for_admissible_lift} and the fact that
according to Definition~\ref{def:admissible_vector_bundle}
the nonscalar zero order term $A$ is $\groupM$ equivariant, we conclude 
that $I(\mathbf{X}_b,\lambda)$ preserves this splitting. Finally, we have to 
take into account that we do not want to study the operator acting 
on $L^2$ but rather on the anisotropic spaces $\mathsf{H}^{\escapeparam m_b}(\Ss^d,\FibreL)$. 
Recall that the escape function is a purely scalar symbol. We can pick the quantization so that scalar symbols are mapped to operators that preserve the decomposition \eqref{eq:irreducible splitting} --- that is a lower order term condition. In particular, 
$I(\Op^b(e^{-rG_b}), \lambda)$ acts on $L^2(\Ss^d, \groupK\times_{\sigma_i} W_i)$ as a principally scalar operator. Thus we assume from now on
that we have fixed a cusp and that $(\tau, V)$ is unitary and irreducible.
Since it is irreducible, the $\groupM$ equivariant 
term $A$ has to be scalar by Schur's Lemma and the indicial operator
\[
I(\mathbf{X}_b,\lambda) = \lambda\cos\varphi + h\left[\frac{d}{2}\cos\varphi + \nabla_{X_{\textup{gr}}}^ {\FibreL} +A \right]
\]
becomes the sum of a covariant derivative and a scalar term. It will thus be convenient
to study its action on local trivializations by orthogonal parallel frames:
Let $b^{\mathcal N}_1,\ldots, b^{\mathcal N}_{\dim V}$ be an orthonormal basis 
of the fibre $\FibreL_{\mathcal N}$ over the North Pole $\mathcal N\in \Ss^ d$. 
Any point $(\varphi,u)\in \Ss^d \setminus
\{\mathcal S\}$ can be connected to $\mathcal N$ by a path $[0,1]\owns t\mapsto (t\varphi, u)$ 
in a unique way and via parallel transport along these paths we can can define the 
orthonormal basis $b^{\mathcal N}_i(\zeta)$ of the fibre over $\zeta \in \Ss^d\setminus \{\mathcal S\}$. By definition, this means
\[
\nabla_{X_{\textup{gr}}}^{\FibreL} b^{\mathcal{N}}_i = 0.
\]
Similarly we chose a orthonormal parallel frame $b^{\mathcal S}_i(\zeta)$
on $\Ss^d\setminus\{\mathcal N\}$. Comparing these two orthonormal frames on the 
equator $\varphi=\pi/2, u\in \Ss^{d-1}$ we get a smooth gluing function $\mathscr g: \Ss^{d-1} \to U(V)$
such that 
\[
\mathscr g(u) b_i^{\mathcal N}(\pi/2, u) = b_i^ {\mathcal S}(\pi/2, u),\textup{ for }i= 1,\ldots,{\dim V}.
\]
With this gluing function we can express the transformation under the change of trivialisation
for $w\in \mathcal D'(\Ss^d\setminus\{\mathcal N, \mathcal S\}, \FibreL)$ as follows
\[
\begin{split}
 w &= \sum_{k=1}^{\dim V} w_k^ {\mathcal S}b_k^{\mathcal S}(\varphi, u) \\
 &= \sum_{l=1}^{\dim V} \left(\sum_{k=1}^{\dim V} 
 \underbrace{ \langle \mathscr g(u) b_k^{\mathcal N}(\pi/2,u),b_l^{\mathcal N}(\pi/2,u)\rangle_V}_{=:\mathscr g_{l,k}(u)} w^ {\mathcal S}_k \right)b_l^{\mathcal N}(u,\varphi).
 \end{split}
\]
Having introduced these orthonormal frames we can prove.
\begin{lemma}
 If we fix a cusp and consider $\FibreL=\groupK\times_{\tau}V$ for an
 irreducible unitary $\groupM$ representation $(\tau, V)$, then the
 operator $I(X_b -hs,\lambda):D^{\escapeparam m_b}(\Ss^d, \FibreL)\to 
 \mathsf{H}^{\escapeparam m_b}(\Ss^d,\FibreL)$ is injective unless
 \[
  \lambda= \pm h[s-A +(d/2+n)],
 \]
 where $A\in \textup{End}(V)^{\groupM}$ has been identified with a scalar by Schur's lemma
 and $n\in\N$.
\end{lemma}
\begin{proof}
Let us reduce the problem to the case of functions, dealt with by Lemma \ref{lem:P_lambda_finite_kernel}:
 Suppose that $w\in D^{\escapeparam m_b}(\Ss^d, \FibreL)\setminus\{0\}$
 with $I(\mathbf{X}_b-hs,\lambda)w=0$. Then one of the following cases holds:
 
 \textbf{First case: $\supp(w)\setminus\{\mathcal N,\mathcal S\}\neq\emptyset$}. Then 
 we can expand the restriction of $w$ to $\Ss^d\setminus\{\mathcal N\}$ in the orthonormal 
 trivialisation $b_k^{\mathcal S}$ and get
 \[
w_{\Ss^d\setminus\{\mathcal N\}} = \sum w_k^{\mathcal S} b_k^{\mathcal S}(\varphi,u)
 \]
 for scalar distributions $w_k^{\mathcal S}\in \mathcal D'(\Ss^d\setminus\{\mathcal N\})$.
 From the fact that $\nabla_{\Xgr} b_k^{\mathcal S}=0$ we deduce that 
  \[
 \left[h(\Xgr +d/2\cos\varphi) + \lambda\cos\varphi - h(s-A)\right]w^{\mathcal S}_k=0
 \]
 Next, using that $w\in \mathsf{H}^{\escapeparam m_b}(\Ss^d,\FibreL)$ and Lemma~\ref{lemma:Identifying-the-indicial-space} we conclude that $w_k^{\mathcal S}
 \in H^{C_G\escapeparam}(\Ss^d)$, in a small neighbourhood around $\mathcal S$. 
 Furthermore at least one $w_k^{\mathcal S}$ must be nonvanishing on $\Ss^d\setminus\{\mathcal N,\mathcal S\}$. We are thus precisely in the setting of the second case in Lemma~\ref{lem:P_lambda_finite_kernel} and we deduce with the same arguments that such a distribution only exists if $s-A= -d/2-\lambda/h - n$ for some $n\in\N$ and the eigendistributions are precisely given by a linear combination of $f_{\mathcal S,\lambda,\Upsilon_\mu}$ with $|\mu|=n$.
 
 \textbf{Second case: $\supp(w)=\{\mathcal S\}$}. Then 
 we use the same trivialisation as above. This would require distributions 
 $w^{\mathcal S}_k\in \mathsf{H}^{\escapeparam m_b}(\Ss^d)$ with $\supp{w^{\mathcal S}_k}=\mathcal S$.
 But as Lemma~\ref{lemma:Identifying-the-indicial-space} requires these distributions to have positive Sobolev regularity, they have to be zero.
 
 \textbf{Third case: $\supp(w)=\mathcal N$}. Then 
 using the trivialisation on $\Ss^d\setminus\{S\}$ we write:
 \[
  w=\sum w_k^ {\mathcal N} b_k^ {\mathcal N}(\varphi, u), \textup{ with } w_k^ {\mathcal N}\in \mathsf{H}^{\escapeparam m_b}(\Ss^d),~\supp(w_k^{\mathcal N}) = \mathcal N
 \]
 and 
 \[
  \left[h(\Xgr +d/2\cos\varphi) + \lambda\cos\varphi - h(s-A)\right]w^{\mathcal N}_k=0.
 \]
We are thus precisely in the setting of the first case in Lemma~\ref{lem:P_lambda_finite_kernel} and
we deduce that such distributions only exist if $h(s-A) = \lambda - h(n+d/2)$ for some $n\in\N$ and they are
precisely given by linear combinations of $\delta^ {(\mu)}_{\mathcal N,\lambda}$ with $|\mu|=n$. 
\end{proof}
As in the case of functions, we have to care about the extension of those distributions coming from $f_{\mathcal S, \Upsilon_\mu,\lambda}$ and check which still remain in the kernel of the indicial operator. Therefore the following notation is convenient: Given $\underline \Upsilon = (
\Upsilon^{(1)}, \ldots,\Upsilon^{(\dim V)})\in (\R_n(\Ss^{d-1}))^{\dim V}$, define the section
$F_{\mathcal S,\underline \Upsilon, \lambda}
:=\sum_{l=1}^{\dim V} f_{\mathcal S,\Upsilon^{(l)}, \lambda}b_l^{\mathcal S}$.
In order to understand the extension in the sense of homogenous distributions at the North Pole we use the definition of $f_{\mathcal S,\Upsilon,\lambda}$ (Eq. \eqref{eq:f_S_def}) and pass to the trivialisation $b_l^{\mathcal N}$:
\[
 F_{\mathcal S,\underline \Upsilon, \lambda} = (\sin\varphi)^{-\frac{d}{2}-\frac{\lambda}{h}}(2\tan(\varphi/2))^{-(n                                                                +\frac{d}{2}+\frac{\lambda}{h})} \sum_{l=1}^{\dim V} \left(\sum_{i=1}^{\dim V} \mathscr g_{l,i}(u) \Upsilon^{(i)}(u) \right)b_l^{\mathcal N}
\]
We see that each coefficient in front of $b_l^{\mathcal N}$ is again a homogenous distribution around $\mathcal N$ of degree $\rho^ {-d-2\lambda/h - n}$ and we can apply the discussion before Definition~\ref{def:fSUlambda} to extend each of the coefficient distributions. We conclude that the extension remains in the kernel of the indicial operator if and only if one of the following condition holds:
\begin{itemize}
 \item $\lambda \notin h\Z/2$
 \item $(2\Re\lambda/h+n)<0$
 \item $\lambda = -hk/2$, $k\in \Z$, $k\leq n$ and 
 \begin{equation}\label{eq:no_Jordan_condition}
   \int_{\Ss^{d-1}}\left(\Upsilon_\mu(u) \sum_{i=1}^{\dim V} \mathscr g_{l,i}(u)\Upsilon^{(i)}(u)\right) du = 0 
 \end{equation}
for all $l=1,\ldots,\dim V$, $|\mu|=n-k$, and $\Upsilon_\mu \in \R_{n-k}(\Ss^{d-1})$.
\end{itemize}
We will denote the set of all $\underline{\Upsilon} \in (\R_n(\Ss^{d-1}))^{\dim V}$ that fulfill \eqref{eq:no_Jordan_condition} by $\mathscr N_{n,n-k}$. Obviously $\mathscr N_{n,n-k}\subset 
(\R_n(\Ss^{d-1}))^{\dim V}$ is a subvectorspace.
\begin{lemma}\label{lem:indicial_root_spaces_vector}
Fix a cusp and a unitary irreducible representation $(\tau,V)$. Consider $\FibreL =\groupK\times_\tau V \to \Ss^d$. Let $\lambda\in\C$, $n\in \N$ and $\escapeparam C_G>d+n+2|\Re(\lambda/h)|$,
consider the operator $I(\mathbf{X}_b,\lambda) :D^{\escapeparam m_b}(\Ss^d, \FibreL)\to \mathsf{H}^{\escapeparam m_b}(\Ss^d, \FibreL)$ 
and identify $A\in \textup{End}(V)^{\groupM}$ with a complex number by Schur's Lemma.
We give the following description of the generalized eigenspaces 
\[
\mathscr K^j_{\lambda, n,\pm} :=\ker(I(\mathbf{X}_b,\lambda) + h(d/2+n-A) \pm \lambda)^j
\]
by distinguishing the following cases:
 \begin{itemize}
	\item If $\lambda \notin h\mathbb Z/2$ then for all $n\in \N$
	there are no Jordan Blocks, i.e. $\mathscr K^2_{\lambda, n,\pm} = \mathscr K^1_{\lambda, n,\pm}$, and
\begin{align*}
&\mathscr K^1_{\lambda, n,+} = 
\Span\left\{ f_{\mathcal S, \Upsilon_\mu,\lambda}b_l^ {\mathcal S}~\middle|\  l=1,\ldots,\dim V,|\mu| = n \right\}\\
&\mathscr K^1_{\lambda, n,-} = \Span\left\{\delta^{(\mu)}_{\mathcal N,\lambda}b_l^ {\mathcal N}\ \middle|\  l=1,\ldots,\dim V,|\mu|=n\right\}
\end{align*}
	\item If $\lambda= h k/2$, $k\in\N$, and $n=0,\ldots, k-1$, there are no Jordan Blocks, i.e.  $\mathscr K^2_{\lambda, n,-} = \mathscr K^1_{\lambda, n,-}$ and one has
\begin{equation*}
\mathscr K^1_{\lambda, n,-} = \Span\left\{ \delta^ {(\mu)}_{\mathcal N,\lambda}b_l^ {\mathcal N}~\middle|\  l=1,\ldots,\dim V,|\mu| = n \right\}.
\end{equation*}
	\item For $\lambda = h k/2$ and $n=k,k+1,\ldots$ one has Jordan Blocks of index $2$ i.e. 
	$\mathscr K^3_{\lambda, n,-} = \mathscr K^2_{\lambda, n,-}$
  \begin{align*}
\mathscr K^2_{\lambda, n,-} &= \Span\left\{\delta_{\mathcal N,\lambda}^{(\mu)} b_l^{\mathcal N}, F_{\mathcal S, \underline{\Upsilon},\lambda} \ \middle|\ |\mu|=n,\ l\leq\dim V, \underline{\Upsilon} \in\R_{n-k}(\Ss^{d-1})^{\dim V} \right\}.\\
\mathscr K^1_{\lambda, n,-} &= \Span \left\{\delta_{\mathcal N,\lambda}^{(\mu)} b_l^{\mathcal N}, F_{\mathcal S, \underline{\Upsilon},\lambda} \ \middle|\ |\mu|=n,\ l\leq \dim V, \underline{\Upsilon} \in\mathscr N_{n-k,n}  \right\}. 
  \end{align*}
	\item If $\lambda= -h k/2$, $k\in\N$, and $n=0,\ldots k-1$, there are no Jordan Blocks $\mathscr K^2_{\lambda, n,+} = \mathscr K^1_{\lambda, n,+}$, and one has
\begin{equation*}
\mathscr K^1_{\lambda, n,+} = \Span\left\{F_{\mathcal S,\underline{\Upsilon}, \lambda}\ \middle|\ \underline{\Upsilon} \in\R_{n}(\Ss^d)^{\dim V}\right\}.
\end{equation*}
	\item For $\lambda = -h k/2$ and $n=k,k+1,\ldots$ one has Jordan Blocks of index $2$ i.e. $\mathscr K^3_{\lambda, n,+} = \mathscr K^2_{\lambda, n,+}$, and
\begin{align*}
\mathscr K^2_{\lambda, n,+}\!\!&= \Span
\left\{ \delta_{\mathcal N,\lambda}^{(\mu)}b_l^{\mathcal N}, 
F_{\mathcal S,\underline{\Upsilon}, \lambda}\ 
\middle|\ |\mu|=n-k, l\leq \dim V, 
\underline{\Upsilon} \in\R_{n}(\Ss^{d-1})^{\dim V}\right\}.\\
   \mathscr K^1_{\lambda, n,+}\!\!&= \Span
\left\{ \delta_{\mathcal N,\lambda}^{(\mu)}b_l^{\mathcal N}, 
F_{\mathcal S,\underline{\Upsilon}, \lambda}\ 
\middle|\ |\mu|=n-k, l\leq \dim V, 
\underline{\Upsilon} \in\mathscr N_{n,n-k}\right\}.\\
\end{align*}
\end{itemize}
\end{lemma}
Taking into account that an admissible cusp manifolds has only finitely many cusps and
that over each cusp, the finite dimensional unitary representation $\tau_\ell,V_\ell$ that describes
the admissible vector bundle over the cusp, splits into finitely many irreducible subrepresentations, 
we obtain the following
\begin{corollary}\label{cor:indicial_roots_holom_vectorbundle}
For an admissible cusp manifold and an admissible vector bundle in the sense of 
Definition~\ref{def:admissible_vector_bundle}, the indicial roots are affine. Their multiplicities are finite and can be calculated by 
Lemma~\ref{lem:indicial_root_spaces_vector}.
\end{corollary}

\newpage

\appendix

\section{Quantization on manifolds with cusps and propagation of singularities}
\label{appendix:microlocal-tools}

\subsection{Symbols on non-compact spaces}

Since we are working with pseudo-differential operators acting on fiber-bundles over non-compact manifolds, it is important to clarify what notion of symbols we are using. We want to use symbols in the usual Kohn-Nirenberg class, but we have to be slightly careful to take into account the lack of compactness of the manifold. Throughout our arguments, we refer to $C^k$ functions as functions with $C^k$ regularity, and $\mathscr{C}^k$ functions as elements of the corresponding Banach space. The notation $\mathscr{C}^k$ implies the use of a metric to measure the size of the derivatives. Given a Riemannian or Hermitian vector bundle $L\to M$ over a Riemannian manifold, endowed with a compatible connection, we can also define $\mathscr{C}^k(M,L)$ spaces as well as Sobolov spaces $H^s(M,L)$. We introduce
\begin{definition}[Kohn-Nirenberg metric]\label{def:Kohn-Nirenberg-metric}
Assume that $(M,g)$ is a Riemannian manifold. Then its cotangent bundle decomposes as
\[
T(T^\ast M) = H \oplus V,
\]
where $V = \ker d\pi$, with $\pi : T^\ast M \to M$ the usual projection. The so-called \emph{horizontal} space $H$ is given by the Levi-Civita connection. We have natural identifications $V \simeq H \simeq T M$, so we can define horizontal and vertical lifts --- see \cite{Gudmundsson-Kappos}. We define the metric $\overline{g}$ on $T^\ast M$ by
\[
\overline{g}_{(x,\xi)}(X^v + Y^h, W^v + Z^h) = g_x(Y,Z) + \frac{1}{1+g(\xi,\xi)}g_x(X,W)
\]
\end{definition}

\begin{lemma}\label{lemma:bounded-curvature-Kohn-Nirenberg}
Assume that the curvature tensor of $(M,g)$ is bounded and so are all its covariant derivatives.
Then the same holds for $(T^\ast M, \overline{g})$.
\end{lemma}
This can be proved using the expressions for the curvature tensor of such a metric presented in \cite{Gudmundsson-Kappos}. From now on, whenever $M$ is a Riemannian manifold, its cotangent bundle will be endowed with $\overline{g}$.
\begin{definition}[Symbol classes]\label{def:symbol-classes}
Let $(L,\|\cdot\|)\to (M,g)$ be a Riemannian or Hermitian vector bundle over $M$ with compatible
connection $\nabla$. Assume that both the curvatures of $L$ and $M$ are bounded, as are all their covariant derivatives. Then, the \emph{semiclassical Kohn-Nirenberg symbols} $S^n(M,L)$ on $L$ of order $n$ are family of sections $\sigma_h : T^\ast M \mapsto \mathscr{L}(L,L)$ parametrized by a parameter $0<h\leq h_0$ such that for all $k\in \N$, there is $C_k$ independent of $h$ such that
\[
\| \nabla^k \sigma_{h}(x,\xi) \| \leq C_k \langle \xi \rangle^n.
\]
Note that $\nabla^k \sigma_{h}$ is a section of the bundle $(T^*(T^*M))^{\otimes k}\otimes \mathcal L(L,L)\to T^*M$ and the norm $\|\bullet \|$ is constructed by the operator norm on 
$\mathcal L(L,L)$ and the Kohn-Nirenberg metric $\overline g$ on $T^*(T^*M)$.
\end{definition}

For the definition of the anisotropic Sobolev spaces we need the following class of 
anisotropic symbol classes.
\begin{definition}
Let $m\in S^0(M)$ be an order zero Kohn Nirenberg symbol which we call an \emph{order function}. The space of \emph{anisotropic symbols} $S^m_{\log}(M,L)$ consist of those sections $\sigma_h : T^\ast M \mapsto \mathscr{L}(L,L)$ parametrized by a parameter $0<h\leq h_0$ such that for all $k\in \N$, there is $C_k$ independent of $h$ such that
\[
\| \nabla^k \sigma_h\| \leq C_k |\log(1+\langle\xi\rangle)|^k \langle\xi\rangle^{m(x,\xi)},
\]
\end{definition}
Note that the loss of $\log \langle\xi\rangle$ is necessary for the space of anisotropic symbols to contain sufficiently interesting elements such as for example $\langle\xi\rangle^{m(x,\xi)}$. 

In the sequel we will usually drop the parameters $h$ to simplify the notation unless we want to emphasize dependence on $h$

Consider two symbols $\sigma,\varsigma\in S(M,L)$. We say that $\sigma$ is \emph{scalar} if it takes the form $\sigma' \mathbb{1}$ with $\sigma'\in S(M,\R)$. In this case, we define the Poisson bracket:
\[
\{ \sigma, \varsigma \} := \nabla_{H_{\sigma'}} \varsigma,
\]
where $H_{\sigma'}$ is the Hamiltonian vector field of $\sigma'\in S(M,L)\subset C^\infty(T^*M)$. 
\begin{proposition}
Under the assumptions of the definition, the sets of symbols, i.e the union $S(M,L):= \cup_n S^n(M,L)$ and $S_{\log}(M,L):=\cup_n S^n_{\log}(M,L)$  satisfy all the usual properties. They are stable by product, sum, division by elliptic symbols, and graded by the order in the usual sense. They are also stable under the Poisson bracket provided one of the symbols is scalar.
\end{proposition}

It is important to notice that the set of symbols $S(N,\R)$, where $N$ is an admissible cusp manifold is exactly the same class of symbols as was described in the paper \cite{Bonthonneau-2}. It is straight forward to show that the proofs therein apply to $S_{\log}$ (the largest of all classes here). 

To close this section, we consider the radial compactification of the cotangent space.
\begin{definition}\label{def:C-infinity-structure-compactified}
Let $\overline{T^\ast} M$ be the radial compactification of the cotangent space. It has a structure of continuous manifold, but not of $C^\infty$ manifold a priori. We consider the map 
\[
\comp:(x,\xi) \to \left(x,\frac{\xi}{1+ \langle \xi\rangle}\right) = (x,\xi').
\]
This is a homeomorphism of $\overline{T^\ast} M$ to $\overline{B(0,1)}$ in $T^\ast M$ and it endows $\overline{T^\ast} M$ with a the structure of a smooth manifold with boundary. 
Let $\widetilde{g} = \comp^\ast \overline{g}$ and define $\mathscr{C}_{\tilde g}^k$ norms on $\overline{T^\ast} M$ using $\widetilde{g}$. Then we define the \emph{classical symbols} as
\[
S^0_{cl}(M):= \mathscr{C}_{\tilde g}^\infty(\overline{T^\ast}M)
\]
and for $k\in \Z$ we set
$S^k_{cl}(M):=\langle\xi\rangle^k\mathscr{C}_{\tilde g}^\infty(\overline {T^*}M)$.

\end{definition}
Note that for the prescribed smooth structure on $\overline{T^*}M$, $\langle\xi\rangle^{-1}$ is a boundary defining function. In particular, because classical symbols are smooth up to the boundary, they have a homogeneous expansion as $\xi\to \infty$.
\begin{proposition}
 We have the inclusion
 \[
  S^0_{cl}(M)\subset S^0(M).
 \]
 \end{proposition}
\begin{proof}
 We set $\overline{g}'= \comp_\ast \overline{g}$ --- it is a metric on the open ball 
 $B(0,1)$. Then close to $|\xi'| = 1$, 
\[
\overline{g}' = \frac{1}{1-|\xi'|_x} (\xi' d\xi')^2 + g',
\]
where $g'$ is a smooth symmetric $2$-form, and $\overline{g} \leq C \overline{g}'$. Passing back to $\overline{T}^*M$ we deduce $\tilde g\leq C \overline g$.  This is sufficient to deduce that the $\mathscr{C}^k$ norms of $\tilde{g}$ control those of $\overline{g}$. 

Since $\tilde{g}$ has bounded curvature, and derivatives thereof, one can estimate its $\mathscr{C}^k$ norms using flat derivatives in exponential coordinates in balls of size $\sim 1$. In other words, we can restrict our attention to the open unit ball in $\R^n$, and assume that $\overline{g}\geq 1$. We then have to show that the exponential map $\exp^{\overline{g}}_0$ has uniformly bounded derivatives (on the unit ball for $\overline{g}$).

We start by observing that it maps the unit ball for $\overline{g}$ inside the standard unit ball (because geodesics for $\overline{g}$ travel at speed $\sim 1/||g||^{1/2} \leq 1$). The derivative of the exponential map can be expressed in terms of Jacobi fields along geodesics through $0$, and these fields satisfy equations involving only the curvature tensor of $\overline{g}$. For the higher derivatives, we also have a description using fields satisfying a Jacobi equation, with a forced term this time. The forcing term is itself a covariant derivative of the curvature tensor along Jacobi fields. In this way, one sees that the derivatives of the exponential map are all controled using only the curvature of $\overline{g}$ and its covariant derivatives. Since those are bounded with respect to $\overline{g}$, they are also bounded with respect to $1$, and we are done.  

This proves that $C^\infty_{\tilde{g}} \subset C^\infty_{\overline{g}}$, and thus the inclusion 
\[
S^0_{cl}(M)=\mathscr{C}_{\tilde g}^\infty(\overline{T^\ast}M) \subset \mathscr{C}_{\overline g}^\infty(T^\ast M)=S^0(M).
\]
\end{proof}

\subsubsection{Symbols on cusps}\label{sec:symbols-on-cusps}

Symbols on cusps have a particular structure that is central to all the arguments of the article. Consider a cusp $Z$ and a symbol $\sigma \in S^n(Z)$. For the extension to trivially fibred cusps, nothing different happens, so we concentrate on $S^n(Z)$. The symbol estimates take the form (recall that $\xi=Ydy+Jd\theta$)
\begin{equation*}
\left|(y\partial_y)^\alpha (y\partial_\theta)^\beta (y^{-1}\partial_\xi)^{\alpha'} (y^{-1} \partial_J)^{\beta'} \sigma \right| \leq C (1+y^2(Y^2 + J^2))^{\frac{n- \alpha' -|\beta'|}{2}}.
\end{equation*}
We change variables to $r= \log y$, $yY = \lambda$. We get that
\begin{equation*}
\left|(\partial_r)^\alpha (e^r\partial_\theta)^\beta (\partial_\lambda)^{\alpha'} (e^{-r} \partial_J)^{\beta'} \sigma \right| \leq C (1+\lambda^2 + e^{2r}J^2)^{\frac{n- \alpha' -|\beta'|}{2}}.
\end{equation*}
Now, a case of special importance will be symbols that do not depend on $\theta$. When that is the case, we deduce from the estimate above that
\begin{equation}\label{eq:structure-symbols-cusp}
\sigma = \widetilde{\sigma}(r; \lambda, e^r J)
\end{equation}
where $\widetilde{\sigma}$ is a symbol on $\R_r \times \R^2_\xi$ in the usual sense that 
\[
\left| \partial_r^\alpha \partial_\xi^\beta \widetilde{\sigma} \right| \leq C \langle \xi\rangle^{n-|\beta|}.
\]

\subsection{Quantization on fibred cusps}\label{app:quantization_on_cups}

We will use a quantization procedure similar to that presented in \cite{Bonthonneau-2}. For most of the technical details, we refer to that article; We will only clarify a few points.

We want to obtain operators on trivially fibred cusp $Z\times\FibreM$ (see Definition~\ref{def:fibred_cusp}). The Schwartz kernels will be understood as taken with reference to the \emph{Euclidean} volume form on the cusp, $dy d\theta d\vol_{\FibreM}(\zeta)$.

First off, let us write $k=\dim \FibreM$, given an open set $U\subset \R^k$, and a symbol 
$\sigma$ on $T^\ast (Z\times U)$, we define an operator $\Op^{w,0}_{h,Z\times U}(\sigma)$ on $\mathbb{H}^{d+1}_{(y,\theta)}\times U_\zeta$ by the kernel
\begin{equation}\label{eq:def-loc-quantization}
\frac{1}{(2\pi h)^{d+1+k}}\int e^{\frac{i}{h}\Phi} \sigma\left(\frac{y+y'}{2},\frac{\theta+\theta'}{2}, \frac{\zeta+\zeta'}{2}, Y,J,\eta \right) d\xi \left(\frac{y}{y'}\right)^{(d+1)/2},
\end{equation}
where $\Phi$ is the usual phase function $\langle y-y',Y\rangle + \langle \theta - \theta',J\rangle + \langle \zeta - \zeta', \eta \rangle$. Here, we have identified the symbol $\sigma$ with the corresponding $\Lambda_Z$-periodic function on $T^\ast (\Hh^{d+1} \times U)$. The operator we obtain is well defined on $\Lambda_Z$-periodic functions, preserves them, and so we obtain an operator acting on $Z\times U$ --- as was explained on page 319 of \cite{Bonthonneau-2}.

Actually, there is the slight inconvenience that we do not know exactly how the kernel of the quantization $\Op^{w,0}_{h,Z\times U}$ decays far from the diagonal. To avoid this discussion altogether, we take a function $\chi^{\Op} \in C^\infty_c(]-C,C[)$ equal to $1$ around $0$, and we let $\Op^w_{h,Z\times U}(\sigma)$ be the operator on $Z\times U$ whose kernel is
\begin{equation}\label{eq:def-loc-quantization-cutoff}
K_{\Op^{w,0}_{h,Z\times U}(\sigma)} \chi^{\Op}\left( \log \frac{y}{y'} \right).
\end{equation}
Taking local charts on $\FibreM$, we can thus build a quantization $\Op^{w}_{h,Z\times\FibreM}$ on a trivially fibred cusp. 

We next want to define a quantization on general admissible vector bundles $L\to M$ in the sense of Definition~\ref{def:admissible-bundle}:  
Given a relatively compact open set $U\subset M$, we can take a 
coordinate patch to $\R^{d+1}\times \R^k$, that maps the volume form to the 
standard volume form of $\R^{d+1+k}$. Such a chart will be called a \emph{compact
chart} and we will use  the ordinary Weyl quantization on these compact charts 
(see e.g. \cite[\S 4.1.1 and Theorem 14.1]{Zworski-book}). Now, we have another type of charts: 
they are supported on open sets of the form $U^Z_{\mathsf a}\times \mathsf{U}$ with $U^Z_\mathsf{a}=\{ z\in Z\ |\ y(z)> \mathsf{a}\}$ and $\mathsf{U}\subset \FibreM$ with $\mathsf{a}\geq \mathbf{a}$. They can be mapped to open sets of the form $U^Z_{\mathsf a} \times U$ with $U$ relatively compact in $\R^k$. We will also impose that the volume form on the fibres $\FibreM$ is sent to the standard volume of $\R^k$, which is possible because the metric takes a product form. Such a chart will be called a
\emph{cusp chart}.

On any such open chart we can define a quantization for sections of $L$ by tensorizing a quantization on functions with a local orthogonal frame for $L$. This can be done over cusp charts because of the product structure of $L$.

In particular, we can choose $\mathsf{a}\geq \mathbf{a}$, and find a corresponding finite cover $U_\ell$ of $M$ by compact charts or cusp charts and a corresponding partition of unity $\sum \chi_\ell^2 = 1$. Then we define for $\sigma \in S(M,L)$ its quantization $\Op^w_{h,L}(\sigma): C^\infty_c(M,L) \to C^\infty(M,L)$ by
\begin{equation}\label{eq:def-quantization-global}
\Op^w_{h,L}(\sigma) f := \sum_\ell \chi_\ell \Op^w_{h,U_\ell,L}(\sigma) \chi_\ell f.
\end{equation}
The notation will soon be shortened to just $\Op$, and we obtain:
\begin{proposition}\label{prop:properties-quantization}
Let $L$ be an admissible bundle and consider (for $j=1,2$), $\sigma_{j} \in S^{n_j}(M,L)$ for $n_j\in\R$ (respectively in $S^{n_j}_{\log}$ with $n_j\in S^0(M,L)$) We have several properties, valid in the limit $h \to 0$:
\begin{enumerate}
\item \label{itm:product-formula} There exists a third symbol $\varpi\in S^{n_1+n_2}(M,L)$ (respectively $S_{\log}^{n_1+n_2}(M,L)$) such that
\[
\Op(\sigma_1)\Op(\sigma_2) = \Op(\varpi) + R,
\]
where the remainder $\|R\|_{H^{-N} \to H^{N}} = \mathcal O(h^\infty)$  for all $N\in \N$ and $R$ is uniformly properly supported. Furthermore we have $\varpi - \sigma_1 \sigma_2 \in hS^{n_1+n_2-1}$ (resp.  $h\log (1+\langle\xi\rangle) S^{n_1+n_2-1}_{\log}$). 
\item Assume that $\sigma_1$ is scalar. Then we have the commutator formula:
\[
[\Op(\sigma_1),\Op(\sigma_2)] = \frac{h}{i} \Op(\{ \sigma_1, \sigma_2\}) + \mathcal{O}(h^2S^{n_1+n_2-2}).
\]
The remainder worsens to $h^2 \log(1+\langle\xi\rangle)^2S^{n_1+n_2-2}_{\log}$ in the case of exotic symbols.

\item If $\sigma$ is hermitian valued, $\Op(\sigma)$ is symmetric.

\item \label{itm:L2_estimate} For $\sigma \in S^0_{\log}(M,L)$, $s\in\R$, $\Op(\sigma):H^s(M,L) \to H^{s}(M,L)$ is bounded uniformly in $h>0$.

\item For $n\in S^0(M,L)$, if $\sigma\in S^n_{\log}(M,L)$ is elliptic, i.e. if $\sigma^{-1} \in S^{-n}_{\log}(M,L)$, then $\Op(\sigma)$ is invertible on Sobolev spaces for $h$ small enough, so that we can let
\[
H^n(M,L) := \Op(\sigma)^{-1}L^2(M,L). 
\]

\item \label{itm:fourier_invariance} Assume that $\partial_\theta \sigma = 0$ for $y\geq \mathsf{a}$. Consider $f$ supported in some fibred cusp end $M_\ell=Z_\ell\times \FibreM$ and assume that $f$ takes the form $e^{i k \theta} g(y, \zeta)$ where $k\in \Lambda'_\ell$ and $g$ supported in $\{y > C \mathsf{a}\}$. Taking $C>1$ large enough depending only on the partition of unity appearing in \eqref{eq:def-quantization-global}, $\Op(\sigma)f$ has the same form except that it is now supported in $\{ y > \mathsf{a}\}$. In particular, $\Op(\sigma)$ preserves Fourier modes exactly.
\end{enumerate}
\end{proposition}

The stabilization of Fourier modes is a nice feature from which we profit because we have assumed that the curvature is constant $-1$ in the cusps. In a more general case of curvature tending to $-1$, one would have to look for more subtle estimates. 
\begin{remark}\label{rem:reminder_pdo}
The remainder $R$ in the product formula can actually be written as a $\Op'(r)$, with $r$ a $\mathcal{O}(h^\infty S^{-\infty})$ symbol, if $\Op'$ is another quantization built in the same fashion, but where the cutoff away from the diagonal has been changed to another one with sufficiently larger support.
\end{remark}

We will also need a sharp G\r{a}rding lemma:
\begin{lemma}[Sharp G\r{a}rding]\label{lemma:sharp-Garding}
Let $\sigma\in S^1(M,L)$. Assume that $\Re (\sigma) \geq 0$. Then
\[
\Re(\langle \Op(\sigma)u,u \rangle) \geq - C h \| u \|_{L^2}^2.
\]
\end{lemma}

We will prove together Proposition \ref{prop:properties-quantization} and Lemma \ref{lemma:sharp-Garding}.
\begin{proof}
Proofs for 1.-5. can be found in \cite{Bonthonneau-2} for non-exotic symbols. The arguments, however all transfer to exotic symbols. Note that the key argument in \cite{Bonthonneau-2} is that for a symbol on a cusp, in an interval at height $y_0$, the cusps can be rescaled such that one transfers the problem to an euclidean cylinder. The crucial point is that the symbols transform uniformly in $y_0$ under this rescaling (see \cite[Section 1.3]{Bonthonneau-2}). Furthermore, since we introduced a cutoff away from the diagonal in our quantization \eqref{eq:def-loc-quantization-cutoff}, we can rescale the whole operator from a neighbourhood of $y=y_0$ to a fixed Euclidean cylinder with uniform estimates. Much as in \cite{Bonthonneau-2}, the proof of boundedness and other estimates follow from the estimates holding on $\R^n$. It is also the case for the sharp G\r{a}rding estimate.

Let us say a word on the property 6. in Proposition~\ref{prop:properties-quantization}. Inspecting the formula \eqref{eq:def-loc-quantization}, we observe that in the $\theta$ variable, the kernel is just a Fourier transform of $\sigma$ in the $J$ variable. Such an operator commutes with $\partial_\theta$ and thus preserves Fourier modes. To be able to use this formula, we just need that the support of $f$ does not intersect the support of the cutoff function $\chi_\ell$ corresponding to compact charts, hence the condition that $g$ is supported in 
$\{y\geq C \mathsf{a}\}$.
\end{proof}

Following Lemma 1.8 in \cite{Bonthonneau-2}, we can prove that our operators actually act as pseudo-differential operators, and that our quantization is a quantization in the usual sense:
\begin{defprop}
Take $m\in S^0(M,\R)$ scalar. We let $\Psi^m(M,L)$ be the algebra of operators generated by operators of the form $\Op(\sigma)$ with $\sigma\in S_{\log}^m(M,L)$. We call them the algebra of \emph{semiclassical pseudodifferential operators} (or also just \emph{pseudodifferential operators} in short).

On $\Psi^m$, we have a principal symbol map $\sigma_m^0$ which is defined independently of the choice of quantization $\Op$ as a map $\sigma_m^0 : \Psi^m \to S_{\log}^m/ h S_{\log}^{m-1}$, with $\sigma^0(\Op(\sigma)) = [\sigma]$.

Once we have fixed a a quantization, we obtain by iterations a full symbol map $\sigma:\Psi^m \to S^m/h^\infty S^{-\infty}$.
\end{defprop}

\subsection{Semiclassical ellipticity and wavefront sets}
Let us recall the notions of wavefront set and ellipticity:
\begin{definition}\label{def:ell-WF}
Let $A\in \Psi^m(M,L)$. We say that $(x,\xi)\in \overline{T^\ast} M$ is not in the wavefront set $\WF_h(A)$ of $A$ if and only if $\|\sigma(A)(x',\xi')\| 
= \mathcal{O}( h^\infty \langle \xi' \rangle^{-\infty})$ in an open neighbourhood of $(x,\xi)$. 
We say that $A$ is \emph{microsupported} in a set $S\subset \overline{T^*}M$ iff 
$\WF_h(A)\subset S$.

For an order $m$ pseudor $A\in \Psi^m(M,L)$, we also define the $\delta$-elliptic set for some $\delta>0$:
\[
\Ell_{\delta}(A) = \{ (x,\xi) \ |\ \| \langle\xi\rangle^m \sigma(A)^{-1}\| < \delta^{-1}\},
\]
and we define the set of elliptic points by $\Ell(A) = \cup_{\delta>0} \Ell_\delta(A)$.

A family of distributions $u_h\in \mathcal D'(M,L)$ parametrized by $0<h\leq h_0$ is
called $h$-tempered if there is $N\in \N$ such that $\|u_h\|_{H^{-N}}=\mathcal O(h^{-N})$. For any $h$-tempered family of distributions $u_h$ we say that $(x,\xi)$ is \emph{not} in $\WF_h(u)$ if and only if there is a $A\in\Psi^0(M,L)$, $\delta>0$ such that $(x,\xi)\in \Ell_\delta(A)$ and 
\[
A u= \mathcal{O}_{H^\infty}(h^\infty).
\]
We let $\WF(u) = \WF_h(u) \cap \partial \overline{T^\ast} M$. As usual, we call $\WF$ the \emph{classical wavefront set}, and $\WF_h$ the \emph{semi-classical wavefront set}. The first measures the regularity in terms of $C^k$ spaces, while the second additionally measures a finer regularity as $h\to 0$.

Finally, we also have a concept of wavefront set for operators. Given an operator $K$ with kernel $K(x,x')$, we let
\[
\WF_h' K := \left\{(x,\xi;x',-\xi')\ |\ (x,\xi;x',\xi')\in \WF_h (K(\cdot,\cdot))\right\}\subset \overline{T^\ast}(M\times M).
\]
\end{definition}

When $A\in \Psi(M,L)$, $\WF'_{(h)}(A)$ is the image of $\WF_{(h)}(A)$ under the diagonal embedding $\overline{T^\ast} M \to \overline{T^\ast} (M\times M)$.

\begin{lemma}\label{lemma:property-WF'}
Let $K$ be an operator on sections of $L\to M$. Then $(x,\xi;x',\xi')\in T^\ast(M\times M)$ is \emph{not} in $\WF_h'(K)$ if and only if there are pseudors $A$ and $B$, $1$-elliptic respectively at $(x,\xi)$ and $(x',\xi')$ so that 
\[
A K B = \mathcal{O}_{H^{-\infty} \to H^\infty}(h^\infty).
\]
\end{lemma}
\begin{proof} 
 See e.g. \cite[Lemma 2.3]{Dyatlov-Zworski-16}.
\end{proof}

\begin{proposition}[Elliptic regularity]\label{prop:Elliptic-regularity}
Take $P\in \Psi^k(M,L)$, $A\in \Psi^0(M,L)$ and $u$ a tempered family of distributions. Then
\begin{enumerate}
	\item Let $\delta > 0$ and $\WF_h(A)\subset \Ell_\delta(P)$, then there is $Q\in \Psi^{-k}(M, L)$ with $\WF_h(Q)\subset \WF_h(A)$ such that 
	\[
	A=QP+\mathcal O(h^\infty\Psi^{-\infty}).
	\]
	\item Let $\delta > 0, r\in \R$ and $\WF_h(A)\subset \Ell_\delta(P)$, then there is 
	a constant $C$ such that
	\[
\| Au\|_{H^r} \leq C_\delta \| Pu \|_{H^{r-k}} + \mathcal{O}_{H^\infty}(h^\infty).
	\]
	\item As a consequence,
	\[
\WF_h(u)\cap \Ell(P) \subset\WF_h(Pu)
	\]
\end{enumerate}
\end{proposition}

\begin{proof}
1. follows from a standard inductive parametrix construction (see e.g. \cite[Proposition E.32]{Dyatlov-Zworski-book}). The  notion of $\delta$-elliptic set has been introduced precisely to assure that the construction yields symbol in the uniform symbol classes. 

2. follows from $\|Au\|_{H^r} = \|QPu\|_{H^r} + \mathcal O(h^\infty)$ after applying the uniform operator norm estimate (Proposition~\ref{prop:properties-quantization}(\ref{itm:L2_estimate})) to $Q$. 

For 3. assume that $(x,\xi)\in \Ell(P)$, then this particular point is also in $\Ell_\delta(P)$ for some $\delta>0$. Assume furthermore 
$(x,\xi)\notin\WF_h(Pu)$. By definition there is $B\in \Psi^0(M, L)$ with $(x,\xi)\in \Ell_\delta(B)$ such that $BPu\in \mathcal O_{H^\infty}(h^\infty)$. Now we apply (2) to the operator $B$ and use that $(x,\xi)\in \Ell_\delta(BP)$. We thus get $A\in \Psi^0$ with $(x,\xi)\in \Ell_\delta(A)$ fulfilling $Au\in \mathcal O_{H^\infty}(h^\infty)$.
\end{proof}
\subsection{Propagation of singularities and other estimates}

Throughout the paper, to obtain results on the wavefront sets of several operators, we have used lemmas that were almost identical to some lemmas in \cite{Dyatlov-Zworski-16}. In this section we give the versions on admissible cusp manifolds. 

For the most part, the proofs given in the appendix of \cite{Dyatlov-Zworski-16} are also valid in our case. As a consequence, this is a cursory review of some special technicalities, destined to the reader already acquainted with the detail of the arguments in \cite{Dyatlov-Zworski-16}.

\begin{remark}
The only difference in our setting from the usual quantization on compact manifolds is that we do \emph{not} have the conclusion of Beal's theorem, i.e, we cannot incorporate smoothing but not properly supported $\mathcal{O}(h^\infty)$ remainders in the symbols. However, that is not a problem, because Beal's theorem is not invoked in \cite{Dyatlov-Zworski-16}.
\end{remark}

Since we will refer to \cite{Dyatlov-Zworski-16} for details, we explain the correspondance between our lemmas and theirs. Proposition \ref{prop:wavefront-R_Q} is where the following lemmas will be used. It is the equivalent of the wavefront set part of proposition 3.4 in \cite{Dyatlov-Zworski-16}. Its proof employs Lemma \ref{lemma:property-WF'} and Propositions \ref{prop:Elliptic-regularity}, \ref{lemma:Propagation-of-singularities} and \ref{prop:Sink-estimate}.
Lemma \ref{lemma:property-WF'} is equivalent to their lemma 2.3; Proposition \ref{prop:Elliptic-regularity} is similar to their proposition 2.4, and Proposition \ref{prop:Sink-estimate} to their proposition 2.6.
Now, we turn to the most involved one, the Propagation of Singularity Lemma \ref{lemma:Propagation-of-singularities}, equivalent to their Lemma 2.5.
\begin{lemma}[Propagation of singularities]\label{lemma:Propagation-of-singularities}
Let $\mathbf{X}\in \Psi^1(M,L)$ have a scalar principal symbol of the form
\[
[ ip - q] \in S^1/ h S^0,
\]
with $p,q$ real, and $q\geq 0$. Also assume that $p\in S^1_{cl}(M)$. Take a tempered family $u_h$, and $\delta>0$.
\begin{enumerate}
	\item Consider $A,B,B_1\in \Psi^0(M,L)$, such that $B,B_1$ $\delta$-control $A$ in time $T_0$. That is, whenever $(x,\xi) \in\WF_h(A)$, there exists $0<T < T_0$ such that $e^{T H_p}(x,\xi)\in \Ell_\delta(B)$ and $e^{t H_p}(x,\xi)\in \Ell_\delta(B_1)$ for all $t\in [0,T]$. Then for each weight $m\in S^0(M,\R)$,
	\[
\| A u \|_{H^m(M,L)} \leq C_\delta \| B u \|_{H^m(M,L)} + \frac{C_\delta}{h} \| B_1 \mathbf{X} u\|_{H^{m}(M,L)} + \mathcal{O}(h^\infty).
	\]
	\item As a consequence, if $(x,\xi) \notin\WF(u)$ and $e^{-t H_p}(x,\xi)\notin\WF(\mathbf{X}u)$ for $t\in [0,T]$, $e^{-T H_p}(x,\xi)\notin\WF(u)$.
\end{enumerate}
\end{lemma}

The constants are $\mathcal{O}(1)e^{\mathcal{O}(T_0)}$, but we will not need this fact. One can mimick the proof in \cite{Dyatlov-Zworski-16} step by step. Be mindful that $\Re \mathbf{P}$ has to be replaced by $-\Im \mathbf{X}$, and $\Im \mathbf{P}$ by $\Re \mathbf{X}$.

\begin{proof}
In the whole proof, when working on subsets of $\overline{T^\ast} M$, we will be working with the notion of distance obtained on $\overline{T^\ast} M$ obtained by pulling back the distance on $B(0, 1) \subset T^\ast M$ by the map $\comp$ defined in Definition~\ref{def:C-infinity-structure-compactified}. Since $p\in S^1_{cl}$, $e^{tH_p}$ is a smooth flow for this structure. Additionally, we can always assume that the symbols of $A$, $B$ and $B_1$ are in $S^0_{cl}$, i.e smooth up to the boundary of $\overline{T^\ast} M$.

To start with, applying a partition of unity argument, we can assume that $A$ is microsupported in a ball with small radius $\epsilon_0>0$. Then we can also assume that $B$ is microsupported in a $3\epsilon_0$-neighbourhood of the image $e^{T H_p} (\WF_h(A))$ for some $T\in [0,T_0]$, and $B_1$ is microsupported in a $3\epsilon_0$-neighbourhood of the union $\cup_{t\in[0,T]} e^{tH_p}(\WF_h(A))$.

Since the proof in \cite{Dyatlov-Zworski-16} is based on local considerations along the trajectories of the flow in bounded time, and we are \emph{not} seeking to determine the behaviour of the constants when the time $T_0$ goes to infinity, we already observe that the estimate holds if $A$ is supposed to be microsupported in a fixed compact set of $M$, with constants that depend on the compact set. As a consequence, we can restrict our attention to the case when $A$, $B$, $B_1$ are supported in a fibred cusp end $M_\ell$, above a set of the form $\{ y > y_0\}$ with $y_0$ arbitrary large, and satisfy symbol estimates with constants \emph{not} depending on $y_0$.

From the structure of symbols estimates in the cusps --- see equation \eqref{eq:structure-symbols-cusp} --- we deduce that we can find operators $\widetilde{B}$, $\widetilde{B}_1$ such that 
\[
\WF_h(\widetilde{B}) \subset \Ell_{\delta/2}(B),\quad e^{T H_p}(\WF_h(A)) \subset \Ell_{\delta/2}(\widetilde{B})\text{ and }[\partial_\theta,\widetilde{B}] =0.
\]
and similarly for $\widetilde{B}_1$ and $B_1$. The idea is that the symbols of $B$ and $B_1$ are almost invariant under rotations in the $\theta$ variable high in the cusp, so that we can forget that variable altogether. Indeed, then we replace $A$ by $\tilde{A}$ such that $|A| \leq \tilde{A}$, and $\tilde{A}$ also is invariant under rotations, and its wavefront set takes the form
\[
\WF_h(\tilde{A})=  \{ (y',\theta,\PointM')\ |\ (y',\theta_0,\PointM')\in B((y,\theta_0,\PointM) ; 2\epsilon_0), \theta \in \R^d/\Lambda_Z \}.
\]
($\PointM$ designs a generic point in the generic fiber $\mathsf{M}$ of $M\to N$). 

Let us do some more reduction. The vector field $H_p$ acts in a uniform $C^\infty$ fashion on $\overline{T^\ast} M$, and as such $|\nabla H_p|_{L^\infty}<\infty$. Additionally, by symbol estimates, we know that $\partial_\theta p = \mathcal{O}(y^{-\infty})$. As a consequence, for $y$ large enough, an escape function for $\overline{p} = \int p d\theta$ is also an escape function for $p$. In other words, we can assume that $p$ does not depend on $\theta$. Then $H_p$ commutes with $\partial_\theta$. 

Consider that in the cusp, we have an additional fiber structure. Indeed, write $M_\ell = (\R^d/\Lambda_\ell)_\theta \times \R_r \times \FibreM_{\PointM}$. Then we can see $T^\ast M_\ell$ as a fiber bundle, 
\[
\begin{split}
\proj : T^\ast M_\ell = \left[(\R^d/\Lambda_\ell)_\theta\times \R^d\right] &\times T^\ast\left[ \R_r \times \FibreM_{\PointM}\right] \\
 	&\to \mathsf{M}^0:= \R^d \times T^\ast\left(\R \times \FibreM\right)
\end{split}
\]
by forgetting the $\theta$ variable. Seeing $\mathsf{M}^0$ as a vector bundle over $\R\times \FibreM$, we can also extend $\proj$ as a map $\overline{T^\ast}M_\ell \to \overline{\mathsf{M}^0}$. Since $H_p$ commutes with $\partial_\theta$, it projects to a vector field $H_p^0$ on the base $\overline{\mathsf{M}^0}$. Then, for $\delta'>0$, let
\[
U_{\delta'} :=\{ (x,\xi)\in \overline{T^\ast} M_\ell\ |\ |H_p^0(x,\xi)| < \delta' e^{-C T_0} \},
\]
with $C/|\nabla H_p|_{L^\infty} >1$. These are $\theta$ invariant sets.

Provided $C$ was chosen large enough, when $(x,\xi)\in U_{\delta'}$, $e^{t H_p}(x,\xi)\in U_{e^{CT_0} \delta'}$ for $t\in [0,T_0]$, so that \[
d(\proj(x,\xi), \proj(e^{t H_p}(x,\xi)))= \mathcal{O}(\delta').
\]
Since the symbol estimates on the symbol of $B$ are uniform over the whole manifold, we deduce that when $e^{T H_p}(x,\xi) \in \Ell_\delta(B)$ for some $T\in[0,T_0]$ and $(x,\xi) \in U_{\delta'}$, then $(x,\xi) \in \Ell_{\delta/2}(B)$, provided $\delta'$ is small enough --- and smaller and smaller as the symbol of $B$ is allowed to become more singular. In such a case, we can apply directly the elliptic estimate (proposition \ref{prop:Elliptic-regularity}) to conclude.

Now, we can concentrate on the case when $\WF_h(A) \cap U_{\delta'} = \emptyset$. But the injectivity radius of $\R \times \FibreM$ is positive, and the vector field $H_p^0$ is $\mathscr{C}^\infty$. As a consequence, on the complement of $U_{\delta'}$, we can apply a formal form for non-vanishing vector fields to obtain tubular coordinates.

We can build a local section of the flow $z\to (x(z),\xi(z))$ from a small open set $U_{tube}\subset \R^{\ell}$ around $0$ to $\overline{\mathsf{M}^0}$ with $(x(0),\xi(0))=(x,\xi)$, and a local diffeomorphism:
\[
\begin{split}
\coord: (z,\tau)\in U_{tube}\times &]-1/2, T_0 + 1/2[ \\
			&\mapsto e^{\tau H_p^0}(x'(z)) \in \proj(U_{\delta' e^{-CT_0}}^c).
\end{split}
\]
We can choose these coordinates so that they satisfy $\mathscr{C}^k$ estimates that do not depend on the central point $(x,\xi)$, and the size of the open set $U_{tube}$ is fixed also indepently of $(x,\xi)$. If the point $(x,\xi)$ is close to a periodic orbit, this map is not injective, but the map is injective on each set of the form $\{ |\tau - \tau_0| < \delta''\}$ for $\delta''>0$ small enough.

Consider a function $\chi \in C^\infty_c(U_{tube})$ equal to $1$ around $0$, and $\psi\in C^\infty_c(]-1/2, T + 1/2 [)$ such that $\psi(\tau)> 1$ for $\tau \in [0,T_0]$, $\psi \geq 0$ everywhere, and $\psi'(\tau) \geq C \psi(\tau)$ for $\tau \in [-1/2, T - 1/2]$. Finally, let 
\begin{equation}\label{eq:def-escape-prop-sing}
f(x',\xi') := \sum_{\coord(z,\tau) = (x',\xi')} \chi(z) \psi(\tau).
\end{equation}
When the trajectory of $(x,\xi)$ is sufficiently far from periodic trajectories, the sum is reduced to $1$ element, but there \emph{may} be periodic points. Now, we need to check that $f$ thus defined is satisfies symbol estimates independently of $(x,\xi)$. There are two things to verify. First, since the tubular coordinates were constructed with uniform $\mathscr{C}^n$ norms each branch in the equation \eqref{eq:def-escape-prop-sing} satisfies uniform $\mathscr{C}^n$ estimates. Then, we need to check that there are a finite number of such branches. But from the local injectivity of the tubular coordinates, the sum has at most $T/\delta''$ non-vanishing terms.

Now that we have an escape function adapted to problem, the rest of the proof in \cite{Dyatlov-Zworski-16} follows through.
\end{proof}
Before going to the equivalent of Proposition 2.6 in \cite{Dyatlov-Zworski-16}, 
let us recall the definition of radial sinks:
\begin{definition}\label{def:sources-sinks}
Let $L$ be a conic subset of $T^\ast M \setminus \{0\}$. Assume that it is 
invariant under $\Phi_t$. Also assume that for some $\epsilon>0$ its $\epsilon$-conic 
neighbourhood $U_\epsilon$ is such that if $\kappa$ is the projection on 
$\partial \overline{T^\ast} M$,
\[
d( \kappa(e^{tH_p} U), \kappa(L)) \to 0\text{ as } t\to +\infty,
\]
and for some constant $C_0>0$, $|e^{tH_p}(x,\xi)| > C e^{C_0 t}|\xi|_x$ 
whenever $(x,\xi)\in U$. Then $L$ is a \emph{radial sink}. 
\end{definition}

Note that $E^*_u \subset T^*M$ is a radial sink (cf. Lemma~\ref{lem:S_star_M_dynamic}).
Now we can state the high regularity radial sink estimate analogous to 
\cite[Prop 2.6]{Dyatlov-Zworski-16} (note that their terminology of sink and source
is reversed compared to ours, as they propagate in the opposite time direction). We will not introduces sources, since we will not use them.
\begin{proposition}[Sink estimate]\label{prop:Sink-estimate}
Let $\mathbf{X}$ be as in lemma \ref{lemma:Propagation-of-singularities}. Let $L$ be a radial sink. Then there exists $k_0>0$ such that for some $\epsilon>0$,
\begin{enumerate}
	\item For all $C\in \Psi^0$, with $\kappa(L)\subset ell_\epsilon(C)$, there exists $C_1\in \Psi^0$ also $\epsilon$-elliptic around $\kappa(L)$ such that whenever $u$ is tempered and $k\geq k_0$,
	\[
C_1 u\in H^{k_0}\ \Rightarrow\ \| C_1 u\|_{H^k} \leq C h^{-1} \| C \mathbf{X} u\|_{H^k} + \mathcal{O}(h^\infty).
	\]
	\item As a consequence, if $C u \in H^{k_0}$ and $\WF(\mathbf{X} u)\cap \kappa(L) = \emptyset$, then $\WF(u)\cap \kappa(L) = \emptyset$.
\end{enumerate}
\end{proposition}

\begin{proof}
Inspecting the proof in \cite{Dyatlov-Zworski-16}, the arguments are very similar to those in the proof of lemma \ref{lemma:Propagation-of-singularities}. The only novelty is the introduction of a lemma `C.1' on the construction of escape functions. These escape functions are simplified versions of the escape function we built in section \ref{sec:escape-function}, which itself  is adapted from \cite{Faure-Sjostrand-10}. Since we have put in the definition of sinks that the neighbourhood $U$ is actually a uniform $\epsilon$-neighbourhood, the constructions are valid.

\end{proof}

\def\dbar{\leavevmode\hbox to 0pt{\hskip.2ex \accent"16\hss}d} \def\cprime{$'$}
  \def\cprime{$'$}

\end{document}